\documentclass[11pt]{amsart}
\usepackage[english]{babel}

\usepackage[textwidth=17.5cm,textheight=22cm,
]{geometry}

\newcommand{\lvt}{\left|\kern-1.35pt\left|\kern-1.3pt\left|}
\newcommand{\rvt}{\right|\kern-1.3pt\right|\kern-1.35pt\right|}

\usepackage{tensor}
\usepackage[foot]{amsaddr}
\usepackage[english]{babel}
\usepackage[
pdftex,hyperfootnotes]{hyperref}
\usepackage[usenames,dvipsnames,svgnames,table,x11names,table]{xcolor}
\usepackage{pgfplots}
\usepackage{tikz}
\usepackage{tikz-3dplot}
\usetikzlibrary{calc,shadows,shapes.callouts,shapes.geometric,shapes.misc,positioning,patterns,decorations.pathmorphing,decorations.markings,decorations.fractals,decorations.pathreplacing,shadings,fadings,arrows.meta}

\usepackage{longtable}
\usepackage{enumerate}
\allowdisplaybreaks

\usepackage{fancybox}
\usepackage[utf8]{inputenc}

\usepackage{tcolorbox}

\usepackage{rotating,multirow,pdflscape}
\usepackage{amssymb,latexsym,amsmath,amsthm}
\usepackage{mathrsfs}
\usepackage{mathtools,arydshln,mathdots}
\mathtoolsset{showonlyrefs}
\usepackage{bigints}
\usepackage{comment}

\usepackage{nicematrix}

\makeatletter
\renewcommand*\env@matrix[1][*\c@MaxMatrixCols c]{%
 \hskip -\arraycolsep
 \let\@ifnextchar\new@ifnextchar
 \array{#1}}
\makeatother

\mathtoolsset{centercolon}
\usepackage[toc,page]{appendix}

\usepackage{tocvsec2}

\usepackage{Baskervaldx}
\usepackage{newtxmath}

\usepackage{framed}
\theoremstyle{plain}

\newtheorem{teo}{Theorem}[section]
\newtheorem{coro}[teo]{Corollary}
\newtheorem{lemma}[teo]{Lemma}
\newtheorem{pro}[teo]{Proposition}

\theoremstyle{defi}

\theoremstyle{remark}
\newtheorem{rem}[teo]{Remark}

\newcommand{\ii}{\operatorname{i}}

\renewcommand{\d}{\operatorname{d}}
\newcommand{\Exp}[1]{\operatorname{e}^{#1}}

\newcommand{\N}{\mathbb{N}}
\newcommand{\R}{\mathbb{R}}
\newcommand{\Z}{\mathbb{Z}}

\newcommand{\1}{\text{\fontseries{bx}\selectfont \textup 1}}
\newcommand{\sz}[1]{\left| \vec{#1} \right|}
\newcommand{\BO}{\operatorname{O}}

\allowdisplaybreaks[1]

\makeatletter
\DeclareRobustCommand{\gaussk}{\DOTSB\gaussk@\slimits@}
\newcommand{\gaussk@}{\mathop{\vphantom{\sum}\mathpalette\bigcal@{K}}}

\newcommand{\bigcal@}[2]{%
 \vcenter{\m@th
 \sbox\z@{\(#1\sum\)}%
 \dimen@=\dimexpr\ht\z@+\dp\z@
 \hbox{\resizebox{!}{0.8\dimen@}{\( \mathcal{K} \)}}%
 }%
}
\newcommand{\cfracplus}{\mathbin{\cfracplus@}}
\newcommand{\cfracplus@}{%
 \sbox\z@{\( \dfrac{1}{1}\)}%
 \sbox\tw@{\( +\)}%
 \raisebox{\dimexpr\dp\tw@-\dp\z@\relax}{\(+\)}%
}
\newcommand{\cfracdots}{\mathord{\cfracdots@}}
\newcommand{\cfracdots@}{%
 \sbox\z@{\(\dfrac{1}{1}\)}%
 \sbox\tw@{\(+\)}%
 \raisebox{\dimexpr\dp\tw@-\dp\z@\relax}{\(\cdots\)}%
}
\makeatother

\makeatletter
\newcommand*{\relrelbarsep}{.386ex}
\newcommand*{\relrelbar}{%
 \mathrel{%
 \mathpalette\@relrelbar\relrelbarsep
 }%
}
\newcommand*{\@relrelbar}[2]{%
 \raise#2\hbox to 0pt{\(\m@th#1\relbar\)\hss}%
 \lower#2\hbox{\(\m@th#1\relbar\)}%
}
\providecommand*{\rightrightarrowsfill@}{%
 \arrowfill@\relrelbar\relrelbar\rightrightarrows
}
\providecommand*{\leftleftarrowsfill@}{%
 \arrowfill@\leftleftarrows\relrelbar\relrelbar
}
\providecommand*{\xrightrightarrows}[2][]{%
 \ext@arrow 0359\rightrightarrowsfill@{#1}{#2}%
}
\providecommand*{\xleftleftarrows}[2][]{%
 \ext@arrow 3095\leftleftarrowsfill@{#1}{#2}%
}
\makeatother

\newcommand*\pFqskip{8mu}
\catcode`,\active
\newcommand*\pFq{\begingroup
 \catcode`\,\active
 \def ,{\mskip\pFqskip\relax}%
 \dopFq
}
\catcode`\,12
\def\dopFq#1#2#3#4#5{%
 {}_{#1}F_{#2}\biggl[\genfrac..{0pt}{}{#3}{#4};#5\biggr]%
 \endgroup
}

\newcommand{\KF}[5]{F^{#1}_{#2}\left[{#3\atop #4}\Bigg\vert #5\right]}

\usepackage{xcolor} 

\usepackage[normalem]{ulem}
\usepackage{soul} 

\usepackage{comment} 

\usepackage{tikz}
\usetikzlibrary{shapes,arrows}
\usepackage{verbatim}

\tikzstyle{block} = [draw, rectangle, minimum height=3em, minimum width=2em]

\usepackage{mathrsfs} 

\allowdisplaybreaks
\title[Integral and hypergeometric representations of multiple discrete orthogonality]{
 Classical discrete multiple orthogonal polynomials:\\ hypergeometric and integral representations
}

\keywords{Discrete multiple orthogonal polynomials, hypergeometric representations, Kampé de Fériet series, integral representations, Hahn, Meixner, Kravchuk, Charlier, AT systems, Askey scheme}

\author[Branquinho]{Amílcar Branquinho\( ^1\)}
\address{\( ^1\)CMUC, Departamento de Matem\'atica,
 Universidade de Coimbra, Largo D. Dinis, 3000-143 Coimbra, Portugal}
\email{\(^1\)ajplb@mat.uc.pt}

\author[Díaz]{Juan E.F. Díaz\(^{2}\)}
\address{\(^2\)CIDMA, Departamento de Matemática, Universidade de Aveiro, 3810-193 Aveiro, Portugal}
\email{\(^2\)juan.enri@ua.pt}

\author[Foulquié]{Ana Foulquié-Moreno\(^3\)}
\address{\(^3\)CIDMA, Departamento de Matemática, Universidade de Aveiro, 3810-193 Aveiro, Portugal}
\email{\(^3\)foulquie@ua.pt}

\author[Mañas]{Manuel Mañas\(^4\)}
\address{\(^4\)Departamento de Física Teórica, Universidad Complutense de Madrid, 28040-Madrid, Spain}
\email{\(^4\)manuel.manas@ucm.es}

\author[Wolfs]{Thomas Wolfs\(^5\)}
\address{\(^5\)Department of Mathematics, KU Leuven, 3001 Leuven, Belgium}
\email{\(^5\)thomas.wolfs@kuleuven.be} 

\subjclass{42C05,33C45,33C47}

\begin{document}

\maketitle

\begin{abstract}
This work explores classical discrete multiple orthogonal polynomials, including Hahn, Meixner of the first and second kinds, Kravchuk, and Charlier polynomials, with an arbitrary number of weights. Explicit expressions for the recursion coefficients of Hahn multiple orthogonal polynomials are derived. By leveraging the multiple Askey scheme and the recently discovered explicit hypergeometric representation of type I multiple Hahn polynomials, corresponding explicit hypergeometric representations are provided for the type I polynomials and recursion coefficients of all the aforementioned descendants within the Askey scheme. Additionally, integral representations for these families within the Hahn class in the Askey scheme are presented. The multiple Askey scheme is further completed by providing the corresponding limits for the weights, polynomials, and recurrence coefficients.
\end{abstract}

\tableofcontents

\section{Introduction}

The first four coauthors embarked on an ambitious project to derive explicit hypergeometric expressions for classical multiple orthogonal polynomials in both their continuous and discrete forms. Their initial focus, detailed in \cite{HahnI}, involved identifying all Hahn polynomial families within the multiple Askey scheme for the simpler case of two weights, excluding the Hermite family. This effort successfully covered the Jacobi--Piñeiro, Laguerre (both first and second kinds), Meixner (both first and second kinds), Kravchuk, and Charlier polynomials. Subsequently, their research expanded to handle an arbitrary number of weights. In \cite{HS:JP-L1}, they presented formulas for the Jacobi--Piñeiro and Laguerre polynomials of the first kind. Further, \cite{CMOPI} provided explicit expressions for the Laguerre polynomials of the second kind and for type I Hermite polynomials. The latter was particularly noteworthy as explicit formulas for type I Hermite polynomials, even with two weights, had previously eluded researchers. Additionally, they were able to determine the recurrence coefficients for all these cases.

Despite these advancements, finding explicit hypergeometric expressions for the Hahn polynomials of type I with an arbitrary number of weights remained unresolved by the original team. However, a recent breakthrough in \cite{BDFMW}, achieved in collaboration with the final author, has provided the sought explicit hypergeometric formulas and contour integral representations for these Hahn polynomials. These developments build upon the work presented in earlier papers \cite{VAWolfs, Wolfs}.

In this paper, we derive explicit expressions for the near-neighbor recurrence coefficients of the multiple Hahn polynomials. Building on the multiple Askey scheme, we obtain hypergeometric representations for the Hahn's descendants, specifically for the type I discrete multiple orthogonal polynomials within the Meixner of the first and second kinds, Kravchuk, and Charlier families. Additionally, we provide explicit formulas for the near-neighbor recurrence coefficients for these families. Furthermore, drawing on insights from our previous work \cite{BDFMW}, we offer contour integral representations for all these discrete polynomial families. Additionally, the multiple Askey scheme is completed by providing the corresponding limits for the weights, polynomials, and recurrence coefficients, thereby completing the scheme and ensuring a comprehensive integration of the discrete and continuous families.

\subsection{Discrete Multiple Orthogonal Polynomials}

To establish some fundamental concepts regarding discrete multiple orthogonal polynomials, see \cite{Ismail,nikishin_sorokin}, consider the system defined by \( p \geqslant 2 \) weight functions:
\[ 
w_1, \ldots, w_p : \Delta \subseteq \mathbb{Z} \rightarrow \mathbb{R}_{\geqslant}.
\]

Under certain conditions, a sequence of type II polynomials \( B_{\vec{n}}^{({\rm II})} \) can be found. These polynomials are monic and have degrees \(\deg B_{\vec{n}}^{({\rm II})} \leqslant |\vec{n}|\), satisfying the orthogonality relations:
\begin{align}
 \label{DO:II}
 \sum_{k \in \Delta} k^j B_{\vec{n}}^{({\rm II})}(k) w_i(k) = 0, \quad j \in \{0, \ldots, n_i - 1\}, \quad i \in \{1, \ldots, p\}.
\end{align}

In addition, there exist \( p \) sequences of type I polynomials, denoted \( A^{(1)}_{\vec{n}}, \ldots, A^{(p)}_{\vec{n}} \), each with \(\deg A^{(i)}_{\vec{n}} \leqslant n_i - 1\), satisfying the orthogonality conditions:
\begin{align}
 \label{DO:I}
 \sum_{i=1}^p \sum_{k \in \Delta} k^j A^{(i)}_{\vec{n}}(k) w_i(k) =
 \begin{cases}
 0, & \text{if } j \in \{0, \ldots, |\vec{n}| - 2\}, \\
 1, & \text{if } j = |\vec{n}| - 1.
 \end{cases}
\end{align}

In this context, let \( \vec{n} = (n_1, \ldots, n_p) \in \mathbb{N}_0^p \) with \( |\vec{n}| \coloneqq n_1 + \cdots + n_p \). The notation \( \mathbb{N} \coloneqq \{1, 2, 3, \ldots\} \) and \( \mathbb{N}_0 \coloneqq \{0\} \cup \mathbb{N} \) will be used. These conditions uniquely define the polynomials to fulfill the biorthogonality conditions:
\begin{equation}
 \label{biorthogonality}
 \sum_{i=1}^p \sum_{k \in \Delta} B_{\vec{n}}^{({\rm II})}(k) A^{(i)}_{\vec{m}}(k) w_i(k) =
 \begin{cases}
 0, & \text{if } m_i \leqslant n_i \text{ for } i = 1, \ldots, p, \\
 1, & \text{if } |\vec{m}| = |\vec{n}| + 1, \\
 0, & \text{if } |\vec{n}| + 1 < |\vec{m}|.
 \end{cases}
\end{equation}

Sometimes it can be convenient to rewrite the previous expressions using the type I linear form
\begin{equation}
 \label{LF}
 A_{\vec{n}}^{({\rm I})}(x)\coloneq\sum_{i=1}^p A_{\vec{n}}^{(i)}(x)w_i(x).
\end{equation}

\subsection{Recurrence Relations}
Both type I and type II polynomials follow a recurrence relation. According to the notation in \cite[\S 23.1.4]{Ismail}, let \( \big(\pi(1), \pi(2), \ldots, \pi(p)\big) \) represent a permutation of \( (1, 2, \ldots, p) \), and let \( \vec{e}_k \in \mathbb{R}^p \) denote the \( k \)-th vector of the canonical basis in \( \mathbb{R}^p \). Define the vectors:
\[
\vec{s}_0 \coloneqq \vec{0}, \quad \vec{s}_j \coloneqq \sum_{i=1}^j \vec{e}_{\pi(i)}, \quad j \in \{1, \ldots, p\}.
\]

The type II and type I polynomials satisfy the following recurrence relations:
\begin{equation}
 \label{NNRR}
 \begin{aligned}
 x B_{\vec{n}}^{({\rm II})}(x) &= B_{\vec{n} + \vec{e}_k}^{({\rm II})}(x) + b_{\vec{n}}^0(k) B_{\vec{n}}^{({\rm II})}(x) + \sum_{j=1}^p b^j_{\vec{n}} B_{\vec{n} - \vec{s}_j}^{({\rm II})}(x), \\
 x A^{(i)}_{\vec{n}}(x) &= A^{(i)}_{\vec{n} - \vec{e}_k}(x) + b_{\vec{n} - \vec{e}_k}^0(k) A^{(i)}_{\vec{n}}(x) + \sum_{j=1}^p b^j_{\vec{n} + \vec{s}_{j-1}} A^{(i)}_{\vec{n} + \vec{s}_j}(x),
 \end{aligned}
\end{equation}
where \( i \in \{1, \ldots, p\} \).

From the biorthogonality condition \eqref{biorthogonality}, the recurrence coefficients can be expressed as:
\begin{align}
 \label{NextNeighbourRecurrenceIntegralRep}
 b_{\vec{n}}^0(k) &= \sum_{x \in \Delta} x B_{\vec{n}}^{({\rm II})}(x) A^{({\rm I})}_{\vec{n} + \vec{e}_k}(x), \\
 b_{\vec{n}}^j &= \sum_{x \in \Delta} x B_{\vec{n}}^{({\rm II})}(x) A^{({\rm I})}_{\vec{n} - \vec{s}_{j-1}}(x) ,\quad j \in \{1, \ldots, p\} .
\end{align}

\begin{rem}
 The multi-index \( \vec{n}-\vec{s}_{j-1}\) corresponds to subtracting \( 1 \) to the \( j-1\) different entries 
\begin{align*}
 \{{n}_{\pi(1)},{n}_{\pi(2)},\ldots,{n}_{\pi(j-1)}\}
\end{align*}
of the multi-index \( \vec{n} \). Consequently,
\begin{align*}
 |\vec{n}+\vec{e}_k|&=|\vec{n}|+1, &
&
\begin{aligned}
 |\vec{n}-\vec{s}_{j-1}|&=|\vec{n}|-j+1,&j&=\{1,\ldots,p\}.
\end{aligned}
\end{align*}
 The condition \( \left(\vec{n}-\vec{s}_{j-1}\right)_i=n_i\) is equivalent to write \( j\leqslant \pi^{-1}(i)\) while \( \left(\vec{n}-\vec{s}_{j-1}\right)_i=n_i-1\) is equivalent to \( j>\pi^{-1}(i)\). Since these conditions will be relevant along the paper, let us define the following sets
\begin{align*}
 S(\pi,j)&\coloneq\left\{i \in \{1, \ldots, p\} \mid j \leqslant \pi^{-1}(i)\right\}, &
 S^{\textsf c}(\pi,j)&\coloneq \{1, 2, \ldots, p\} \setminus S(\pi,j).
\end{align*}

For example, with \( p = 4 \) and the permutation
\begin{align*}
 \pi = \begin{pNiceMatrix}
 1 & 2 & 3 & 4 \\
 4 & 2 & 1 & 3
 \end{pNiceMatrix}
\end{align*}
the sets are:
\begin{align*}
 S(\pi,1)&=\{1, 2, 3, 4\}, & S(\pi,2)&=\{1, 2, 3\}, & S(\pi,3)&=\{1, 3\}, & S(\pi,4)&=\{3\}.
\end{align*}
\end{rem}

In this paper, explicit hypergeometric and integral expressions for the type I polynomials \eqref{DO:I} and their recurrence coefficients \eqref{NNRR} will be derived for the following families within the partial multiple Askey scheme, see Figure \ref{figure:Askey_discrete}.

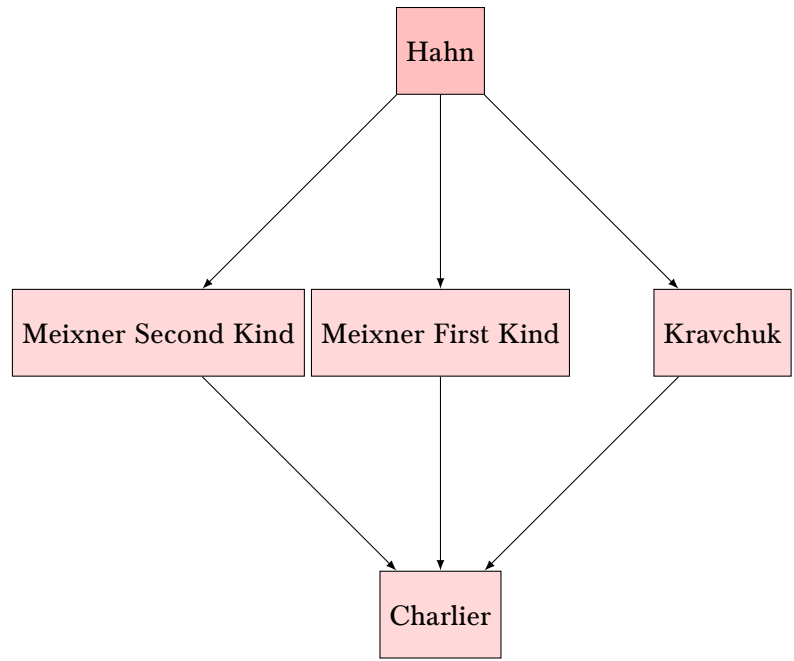
\begin{figure}[htbp]
 \begin{tikzpicture}[node distance=3.75cm]

 \node[ fill=red!25,block,] (a) {Hahn};

 \node[ fill=red!15,block, below of = a] (b) {Meixner First Kind};

 \node[ fill=red!15,block, left of = b] (c) {Meixner Second Kind};
 \node[ fill=red!15,block, right of = b] (d) {Kravchuk};
 
 \node[ fill=red!15,block, below of =b] (e) {Charlier};
 
 \draw[-latex] (a)--(b); 
 \draw[-latex] (a)--(c);
 \draw[-latex] (a)--(d);
 \draw[-latex] (b)--(e);
 \draw[-latex] (d)--(e);
 \draw[-latex] (c)--(e);
 \end{tikzpicture}
	\caption{Askey scheme for Hahn discrete descendants}
	\label{figure:Askey_discrete}
\end{figure}

\subsection{Hypergeometric Series}
Before delving into the main results, it's important to recall that, in many cases, these polynomials can be represented using generalized hypergeometric series \cite{andrews,slater},
\begin{align}
\label{GHF}
 \pFq{p}{q}{a_1,\ldots,a_p}{\alpha_1,\ldots,\alpha_q}{x}\coloneqq \sum_{l=0}^{\infty}\dfrac{(a_1)_l\cdots(a_p)_l}{(\alpha_1)_l\cdots(\alpha_q)_l}\dfrac{x^l}{l!}.
\end{align}

We are going to need a generalization of these previous ones; known as the multiple Kampé de Fériet series \cite[Chapter 1.4]{Srivastava}
\begin{multline}
 \label{MultipleKF}
 \KF{n:r_1;\cdots;r_p}{q:k_1;\cdots;k_p}{(a_1,\ldots,a_n):(b_1,\ldots,b_{r_1});\cdots;(c_1,\ldots,c_{r_p})}{(\alpha_1,\ldots,\alpha_q):(\beta_1,\ldots,\beta_{k_1});\cdots;(\gamma_1,\ldots,\gamma_{k_p})}{x_1,\ldots,x_p}\\
 \coloneq\sum_{l_1=0}^{\infty}\cdots\sum_{l_p=0}^{\infty}\dfrac{(a_1)_{l_1+\cdots+l_p}\cdots(a_n)_{l_1+\cdots+l_p}}{(\alpha_1)_{l_1+\cdots+l_p}\cdots(\alpha_q)_{l_1+\cdots+l_p}}\dfrac{(b_1)_{l_1}\cdots(b_{r_1})_{l_1}}{(\beta_1)_{l_1}\cdots(\beta_{k_1})_{l_1}}\cdots\dfrac{(c_1)_{l_p}\cdots(c_{r_p})_{l_p}}{(\gamma_1)_{l_p}\cdots(\gamma_{k_p})_{l_p}}\dfrac{x_1^{l_1}}{l_1!}\cdots\dfrac{x_p^{l_p}}{l_p!}.
\end{multline}
All these series are defined through the Pochhammer symbols
\begin{align*}
 (x)_n\coloneqq\dfrac{\Gamma(x+n)}{\Gamma(x)}=\begin{cases}
 x(x+1)\cdots(x+n-1)\;\text{if}\;n\in\N,\\
 1\;\text{if}\;n=0.
 \end{cases}
\end{align*}

The Chu--Vandermonde formula, that follows the umbral paradigm correspondences to Newton's binomial formula for Pochhammer symbols, is
\begin{equation}\label{eq:Chu--Vandermonde}
 \begin{aligned}
 (x+y)_n&=\sum_{k=0}^n\dbinom{n}{k}(x)_k(y)_{n-k},& n&\in\mathbb N_0.
 \end{aligned}
\end{equation}
This formula can be deduced from this particular case of the celebrated Gauss' formula
\begin{equation}\label{eq:Gauss_hypergeometric}
 \begin{aligned}
 \pFq{2}{1}{-n,x}{y}{1}&=\sum_{k=0}^{n}\dfrac{(-n)_k}{k!}\dfrac{(x)_k}{(y)_k}=\dfrac{(y-x)_n}{(y)_n},& n&\in\mathbb N_0.
 \end{aligned}
\end{equation}

In order to prove the recurrence coefficients
of the multiple Hahn polynomials
we will use the following generalization of the Pfaff--Saalschütz formula, see \cite[Theorem 9]{GeneralizedPS}:
\begin{equation}\label{GeneralizedPfaff-Saalschutz}
 \begin{aligned}
 \pFq{3}{2}{a,b, -n}{c+k, a+b+1-n-c}{1}
 &=\dfrac{(c-a)_n(c-b+k)_n}{(c+k)_n(c-a-b)_n}\pFq{3}{2}{-k,b,-n}{c-a, b-c-k-n+1}{1}, & k,n&\in\mathbb N_0 .
\end{aligned}
\end{equation}
We also recall Stirling's formula
\begin{align}
 \label{stirling}
\begin{aligned}
 \Gamma(z) & \sim \sqrt{2\pi} z^{z-\frac{1}{2}} \operatorname{e}^{-z}, & z & \rightarrow\infty, & |\arg z| < \pi,
\end{aligned}
\end{align}
for the asymptotic behavior of the gamma function. 
We will use the convention that whenever an exponential function has a branch cut, we assume it is taken along the negative real line.

Sometimes it will be convenient to use the following asymptotics of a ratio of gamma functions directly
\begin{equation} \label{GFB}
\begin{aligned}
 \frac{\Gamma(az+b)}{\Gamma(az+c)}&\sim (az)^{b-c},& z&\rightarrow\infty,& |\arg az| &< \pi.
\end{aligned}
\end{equation}

\subsection{Hahn Hypergeometric Representations 
}

The weight functions for the Hahn family are defined as
\begin{equation}
 \label{WeightsHahn}
\begin{aligned}
 w_{i}(x;\alpha_i,\beta,N)&=\dfrac{\Gamma(\alpha_i+x+1)}{\Gamma(\alpha_i+1)\Gamma(x+1)}\dfrac{\Gamma(\beta+N-x+1)}{\Gamma(\beta+1)\Gamma(N-x+1)}, & i &\in\{1,\ldots,p\},& \Delta &=\{0,\ldots,N\},
\end{aligned}
\end{equation}
with
\( \alpha_1,\ldots,\alpha_p,\beta>-1\). In order to establish an AT system, we require that \( \sz{n} \leqslant N\in\mathbb N_0\) and \( \alpha_i-\alpha_j\not\in\mathbb Z\) for \( i\neq j\). We will denote \(\vec{\alpha}\coloneq(\alpha_1,\ldots,\alpha_p)\).

The corresponding type II polynomials were first found for \( p=2\) in \cite{Arvesu} and then generalized for \( p \geqslant 2\) in~\cite{AskeyII}. One showed that
\begin{equation}
\label{HahnTypeII}
 Q_{\vec{n}}^{({\rm II})} (x;\vec{\alpha},\beta,N)=\sum_{l_1=0}^{n_1}\cdots\sum_{l_p=0}^{n_p}
 C_{\vec{n}}^{l_1,\ldots,l_p}\, (-x)_{l_1+\cdots+l_p},
 \end{equation}
 where
 \begin{multline}\label{Coefficients_HahnTypeII}
C_{\vec{n}}^{l_1,\ldots,l_p} \coloneq
 \dfrac{(-N)_{|\vec{n}|}}{(-N)_{l_1+\cdots+l_p}} \prod_{i=1}^p \dfrac{(\alpha_i+1)_{n_i}}{(\alpha_i+\beta+|\vec{n}|+1)_{n_i}}\dfrac{(-n_i)_{l_i}}{l_i!}
 \dfrac{(\alpha_i+\beta+\sum_{k=1}^i n_k+1)_{\sum_{j=i}^p l_j}}{(\alpha_i+1)_{\sum_{j=i}^{p}l_j}} \\\times \prod_{i=1}^{p-1} \dfrac{(\alpha_i+n_i+1)_{\sum_{j=i+1}^p l_j}}{(\alpha_i+\beta+\sum_{k=1}^i n_k +1)_{\sum_{j=i+1}^{p}l_j}}. 
 \end{multline}

In \cite{BDFMW}, the following alternative representation was proven
\begin{multline}
 \label{HahnTypeIIWeighted}
 Q_{\vec{n}}^{({\rm II})}(x;\vec{\alpha},\beta,N)
 =
 {
 (-1)^{|\vec{n}|}}\dfrac{\Gamma(N+\beta+1)}{(N-|\vec{n}|)!}\prod_{i=1}^p\dfrac{(\alpha_i+1)_{n_i}}{(\alpha_i+\beta+|\vec{n}|+1)_{n_i}} \dfrac{\Gamma(N-x+1)}{\Gamma(\beta+N-x+1)} \\
 \times \pFq{p+2}{p+1}{-|\vec{n}|-\beta,-x,\vec{\alpha}+\vec{n}+\vec{1}_p}{-N-\beta,\vec{\alpha}+\vec{1}_p}{1}.
 \end{multline}
Also in \cite{BDFMW}, the first explicit hypergeometric representation of the Hahn polynomials of type I for \( p \geqslant 2 \) was presented as follows:
\begin{multline}
\label{HahnTypeI}
 Q_{\vec{n}}^{(i)}(x;\vec{\alpha},\beta,N)
 =\dfrac{(-1)^{|\vec{n}|-1}(N+1-|\vec{n}|)!}{(n_i-1)!(\beta+1)_{|\vec{n}|-1}(\alpha_i+\beta+|\vec{n}|)_{N+2-|\vec{n}|}}\dfrac{\prod_{k=1}^p({\alpha}_k+\beta+|\vec{n}|)_{n_k}}{\prod_{k=1,k\neq i}^p({\alpha}_k-{\alpha}_i)_{{n}_k}}\\
 \times\pFq{p+2}{p+1}{-n_i+1,\alpha_i+\beta+|\vec{n}|,(\alpha_i+1)\vec{1}_{p-1}-\vec{\alpha}^{*i}-\vec{n}^{*i},x+\alpha_i+1}{\alpha_i+1,(\alpha_i+1)\vec{1}_{p-1}-\vec{\alpha}^{*i},\alpha_i+\beta+N+2}{1} ,
 \end{multline}
 or equivalently
\begin{align*}
 Q_{\vec{n}}^{(i)}(x;
 \vec{\alpha} ,\beta,N)
 = \sum_{l=0}^{n_i-1}C_{\vec{n}}^{(i),l}\,(x+\alpha_i+1)_l,
\end{align*}
with
\begin{multline}
 \label{coefHahnTypeI}
 C_{\vec{n}}^{(i),l}\coloneq\dfrac{(-1)^{|\vec{n}|-1}(N+1-|\vec{n}|)!}{(n_i-1)!(\beta+1)_{|\vec{n}|-1}(\alpha_i+\beta+|\vec{n}|)_{N+2-|\vec{n}|}}\dfrac{\prod_{k=1}^p({\alpha}_k+\beta+|\vec{n}|)_{n_k}}{\prod_{k=1,k\neq i}^p({\alpha}_k-{\alpha}_i)_{{n}_k}}\\
 \times\dfrac{(-n_i+1)_l(\alpha_i+\beta+|\vec{n}|)_l}{l!(\alpha_i+1)_l(\alpha_i+\beta+N+2)_l}\prod_{k=1,k\neq i}^p\dfrac{(\alpha_i-{\alpha}_k-{n}_k+1)_l}{(\alpha_i-{\alpha}_k+1)_l} .
\end{multline}

\begin{rem}
Notice that we used the following notation. We denoted the unit vector by
\(\vec{1}_p\coloneq (1,\ldots,1)\in\mathbb R^p\). Given a vector $\vec{a}\in\mathbb R^p$, we wrote $\vec{a}^{\ast q}\in\mathbb R^{p-1}$ for the vector obtained from $\vec{a}$ after removing its $q$-th component. This notation will often appear along the rest of the text.
\end{rem}

\subsection{Hahn
Integral Representations 
}

In \cite{BDFMW}, the following integral representations for the Hahn multiple orthogonal polynomials were obtained.

\begin{teo}
 \label{HI_IR}
 Let $\Sigma$ be a clockwise contour in $\{t\in \mathbb{C} \mid \operatorname{Re}(t) > -1\}$ enclosing $\cup_{j=1}^p[\alpha_j,n_j+\alpha_j-1]$ exactly once. Then the type I linear forms are given by
 \begin{multline*}
 Q_{\vec{n}}^{({\rm I})}(x;\vec{\alpha},\beta,N) = (-1)^{\sz{n}} (N-\sz{n}+1)! \frac{\prod_{j=1}^p (\alpha_j+\beta+\sz{n})_{n_j}}{\Gamma(\beta+\sz{n})} \frac{\Gamma(N-x+\beta+1)}{\Gamma(x+1)\Gamma(N-x+1)} \\
 \times \bigintsss_\Sigma \frac{\Gamma(t+\beta+\sz{n})}{\Gamma(t+1) \Gamma(t+\beta+N+2)} \frac{\Gamma(x+t+1)}{\prod_{j=1}^p (\alpha_j-t)_{n_j}} \frac{\d t }{2 \pi \ii }.
 \end{multline*}
In particular, the type I polynomials are given by
\begin{multline*}
 Q_{\vec{n}}^{(i)}(x;\vec{\alpha},\beta,N) = (-1)^{\sz{n}} (N-\sz{n}+1)! \frac{\prod_{j=1}^p (\alpha_j+\beta+\sz{n})_{n_j}}{(\beta+1)_{\sz{n}-1}} \frac{\Gamma(\alpha_i+1)}{\Gamma(\alpha_i+x+1)} \\
 \times \bigintsss_{\Sigma_i} \frac{\Gamma(t+\beta+\sz{n})}{\Gamma(t+1) \Gamma(t+\beta+N+2)} \frac{\Gamma(x+t+1)}{\prod_{j=1}^p (\alpha_j-t)_{n_j}} \frac{\d t }{2 \pi \ii },
 \end{multline*}
where \(\Sigma_i\) is a clockwise contour in $\{t\in \mathbb{C} \mid \operatorname{Re}(t) > -1\}$ enclosing \(\cup_{k=0}^{n_i-1} \{\alpha_i+k\}\) exactly once without enclosing any of the other points in \(\cup_{j=1}^p \cup_{k=0}^{n_j-1} \{\alpha_j+k\}\).
\end{teo}

\begin{teo}
 \label{HII_IR}
 Let $\mathcal{C}$ be a counterclockwise contour enclosing $[-N,0]$ exactly once. Then, for $x \in
 \{0,\ldots,N\}$, the type II polynomials are given by
 \begin{multline*}
 Q_{\vec{n}}^{({\rm II})}(x;\vec{\alpha},\beta,N) = 
 \frac{(-1)^{\sz{n}} \Gamma(\sz{n}+\beta+1)}{(N-\sz{n})! \prod_{j=1}^p (\alpha_j+\beta+\sz{n}+1)_{n_j}} \frac{\Gamma(x+1) \Gamma(N-x+1)}{\Gamma(N-x+\beta+1)} \\
 \times \bigintsss_{\mathcal{C}} \frac{\Gamma(s)\Gamma(s+\beta+N+1)}{\Gamma(s+\beta+\sz{n}+1)}\frac{\prod_{j=1}^p (\alpha_j+1-s)_{n_j}}{\Gamma(x+s+1)} \frac{\d s}{2 \pi \ii}.
 \end{multline*}

\end{teo}
\section{Hahn Recurrence Coefficients 
}

In this section, we proceed to derive explicit expressions for the recurrence coefficients given in \eqref{NNRR} for the Hahn family. For that aim we need to establish the following technical lemmas.

\begin{lemma}
\label{lemma1}
Let \(C_{\vec{n}}^{l_1,\ldots,l_p}\) be as defined in \eqref{Coefficients_HahnTypeII} and \(x\in\mathbb R \) then,
\begin{align*}
\sum_{l_1=0}^{n_1}\cdots\sum_{l_p=0}^{n_p}\dfrac{(-N)_{l_1+\cdots+l_p}(x)_{l_1+\cdots+l_p}}{(x+\beta+1)_{l_1+\cdots+l_p}}
C^{l_1,\ldots,l_p}_{\vec{n}}=\dfrac{(\beta+1)_{|\vec{n}|}(-N)_{|\vec{n}|}}{(x+\beta+1)_{|\vec{n}|}}\prod_{q=1}^{p}\dfrac{(\alpha_q-x+1)_{{n}_q}}{(\alpha_q+\beta+|\vec{n}|+1)_{n_q}}.
\end{align*}
\end{lemma}

\begin{proof}
Consider the Hahn coefficients \eqref{Coefficients_HahnTypeII} and transform them as follows:
\begin{align*}
 \tilde{C}_{\vec{n}}^{l_1,\ldots,l_p} \coloneq (-1)^{|\vec{n}|} \frac{(-N)_{l_1 + \cdots + l_p}}{(-N)_{|\vec{n}|}} C_{\vec{n}}^{l_1,\ldots,l_p}.
\end{align*}
The transformed coefficients \(\tilde{C}_{\vec{n}}^{l_1,\ldots,l_p}\) are exactly the coefficients of the monic type II polynomials for the Jacobi--Piñeiro family, as shown in \cite[Theorem 3.2]{AskeyII}. Specifically, the Jacobi--Piñeiro type II polynomials can be expressed as:
\begin{align*}
 P_{\vec{n}}^{({\rm II})}(z) = \sum_{l_1=0}^{n_1} \cdots \sum_{l_p=0}^{n_p} \tilde{C}_{\vec{n}}^{l_1,\ldots,l_p} z^{l_1 + \cdots + l_p}.
\end{align*}

In \cite[Lemma 2.1]{CMOPI}, we demonstrated that for \(m \in \mathbb{N}_0\):
\begin{align*}
 \sum_{l_1=0}^{n_1} \cdots \sum_{l_p=0}^{n_p} \frac{(\alpha_i + n_i + m)_{l_1 + \cdots + l_p}}{(\alpha_i + \beta + n_i + m + 1)_{l_1 + \cdots + l_p}} \tilde{C}_{\vec{n}}^{l_1,\ldots,l_p} = (-1)^{|\vec{n}|} \frac{(\beta + 1)_{|\vec{n}|}}{(\alpha_i + \beta + n_i + m + 1)_{|\vec{n}|}} \prod_{q=1}^p \frac{(\alpha_q - \alpha_i - n_i - m + 1)_{n_q}}{(\alpha_q + \beta + |\vec{n}| + 1)_{n_q}}.
\end{align*}
This result can be generalized by replacing \(\alpha_i + n_i + m\) with any \(x \in \mathbb{R}\). By applying this replacement and using the coefficients \(\tilde{C}_{\vec{n}}^{l_1,\ldots,l_p}\), we obtain the desired result.
\end{proof}

\begin{lemma}
\label{lemma3}
Let consider \( x,\alpha_1,\ldots,\alpha_p\in\mathbb C\) and \( n_1,\ldots,n_p\in\mathbb N_0\) then
\begin{align*}
\sum_{l=1}^p\dfrac{n_l}{(\alpha_l-x)}\prod_{q=1}^{l-1}(\alpha_q-x+1)_{n_q}\prod_{q=l}^p(\alpha_q-x)_{n_q}&\stackrel{}{=}\prod_{q=1}^p(\alpha_q-x+1)_{n_q}-\prod_{q=1}^p(\alpha_q-x)_{n_q} .
\end{align*}
\end{lemma}

\begin{proof}
By removing the common factor \(\prod_{q=1}^p (\alpha_q - x + 1)_{n_q - 1}\) from both sides of the equation, it reduces to:
\begin{align*}
 \sum_{l=1}^p n_l \prod_{q=1}^{l-1} (\alpha_q - x + n_q) \prod_{q=l+1}^p (\alpha_q - x) = \prod_{q=1}^p (\alpha_q - x + n_q) - \prod_{q=1}^p (\alpha_q - x),
\end{align*}
which is much simpler.

We will proceed by induction over the number of parameters \(p\). For \(p = 1\), the identity is easily proven. Now, assume that it is satisfied for some \(p\geq 1\). Then,
\begin{align*}
 & \sum_{l=1}^{p+1} n_l \prod_{q=1}^{l-1} (\alpha_q - x + n_q) \prod_{q=l+1}^{p+1} (\alpha_q - x) \\
 &\hspace{2.5cm}= (\alpha_{p+1} - x) \sum_{l=1}^{p} n_l \prod_{q=1}^{l-1} (\alpha_q - x + n_q) \prod_{q=l+1}^{p} (\alpha_q - x) + n_{p+1} \prod_{q=1}^{p} (\alpha_q - x + n_q) \\
 &\hspace{2.5cm}= (\alpha_{p+1} - x) \left(\prod_{q=1}^p (\alpha_q - x + n_q) - \prod_{q=1}^p (\alpha_q - x)\right) + n_{p+1} \prod_{q=1}^{p} (\alpha_q - x + n_q) \\
 &\hspace{2.5cm}= (\alpha_{p+1} - x) \prod_{q=1}^p (\alpha_q - x + n_q) - (\alpha_{p+1} - x) \prod_{q=1}^p (\alpha_q - x) + n_{p+1} \prod_{q=1}^p (\alpha_q - x + n_q) \\
 &\hspace{2.5cm}= \prod_{q=1}^{p+1} (\alpha_q - x + n_q) - \prod_{q=1}^{p+1} (\alpha_q - x).
\end{align*}
Hence, the identity also holds for \(p + 1\).
\end{proof}

\begin{lemma}
\label{lemma2}
Let \( C_{\vec{n}}^{l_1,\ldots,l_p}\) be as in Equation \eqref{HahnTypeII} and let \(x\in\mathbb R\), then
\begin{multline*}
\sum_{l_1=0}^{n_1}
\cdots\sum_{l_p=0}^{n_p}\dfrac{(-N)_{l_1+\cdots+l_p}(x)_{l_1+\cdots+l_p}}{(x+\beta+2)_{l_1+\cdots+l_p}}
C^{l_1,\ldots,l_p}_{\vec{n}}\\
=\dfrac{(\beta+2)_{|\vec{n}|-1}(-N)_{|\vec{n}|}}{(x+\beta+2)_{|\vec{n}|}\prod_{q=1}^p(\alpha_q+\beta+|\vec{n}|+1)_{n_q}}
\bigg(
(x+\beta+|\vec{n}|+1)\prod_{q=1}^p(\alpha_q-x+1)_{n_q}
-x\prod_{q=1}^p(\alpha_q-x)_{n_q}\bigg).
\end{multline*}
\end{lemma}

\begin{proof}
Replacing \(C_{\vec{n}}^{l_1,\ldots,l_p}\) by its expression given in \eqref{Coefficients_HahnTypeII} the LHS can be written as
\begin{multline*}
(-N)_{|\vec{n}|}\prod_{q=1}^p\dfrac{(\alpha_q+1)_{n_q}}{(\alpha_q+\beta+|\vec{n}|+1)_{n_q}}\sum_{l_p=0}^{n_p}\cdots\sum_{l_2=0}^{n_2}\dfrac{(-n_p)_{l_p}}{l_p!}\cdots\dfrac{(-n_2)_{l_2}}{l_2!}\dfrac{(x)_{l_2+\cdots+l_p}}{(x+\beta+2)_{l_2+\cdots+l_p}}
\dfrac{(\alpha_1+\beta+n_1+1)_{l_2+\cdots+l_p}}{(\alpha_1+1)_{l_2+\cdots+l_p}}
 \\
 \times \dfrac{(\alpha_2+\beta+n_1+n_2+1)_{l_2+\cdots+l_p}\cdots(\alpha_{p}+\beta+|\vec{n}|+1)_{l_p}}{(\alpha_2+1)_{l_2+\cdots+l_p}\cdots(\alpha_{p}+1)_{l_p}} \\\times
\dfrac{(\alpha_1+n_1+1)_{l_2+\cdots+l_p}\cdots(\alpha_{p-1}+n_{p-1}+1)_{l_p}}{(\alpha_1+\beta+n_1+1)_{l_2+\cdots+l_p}\cdots(\alpha_{p-1}+\beta+n_1+\cdots+n_{p-1}+1)_{l_p}}\\
\times \pFq{3}{2}{-n_1,x+l_2+\cdots+l_p,\alpha_1+\beta+n_1+1+l_2+\cdots+l_p}{x+\beta+2+l_2+\cdots+l_p,\alpha_1+1+l_2+\cdots+l_p}{1}.
\end{multline*}
Now, applying formula \eqref{GeneralizedPfaff-Saalschutz} for \( k=1\) we find that the \( _3F_2\) series in the previous equation can be written as
\begin{multline*}
 \pFq{3}{2}{-n_1,x+l_2+\cdots+l_p,\alpha_1+\beta+n_1+1+l_2+\cdots+l_p}{x+\beta+2+l_2+\cdots+l_p,\alpha_1+1+l_2+\cdots+l_p}{1}\\
 \quad=\dfrac{(\beta+2)_{n_1}(\alpha_1-x+1)_{n_1}}{(\alpha_1+1+l_2+\cdots+l_p)_{n_1}(x+\beta+2+l_2+\cdots+l_p)_{n_1}}
 +\dfrac{n_1(\beta+2)_{n_1-1}(\alpha_1-x+1)_{n_1-1}(x+l_2+\cdots+l_p)}{(\alpha_1+1+l_2+\cdots+l_p)_{n_1}(x+\beta+2+l_2+\cdots+l_p)_{n_1}} .
\end{multline*}
Replacing this and simplifying we find that the LHS reads
\begin{multline*}
\dfrac{(\beta+2)_{n_1}(\alpha_1-x+1)_{n_1}(-N)_{n_1}}{(x+\beta+2)_{n_1}(\alpha_1+\beta+|\vec{n}|+1)_{n_1}}\\\times
(-N+n_1)_{n_2+\cdots+n_p}\prod_{q=2}^p\dfrac{(\alpha_q+1)_{n_q}}{(\alpha_q+\beta+|\vec{n}|+1)_{n_q}}
\sum_{l_p=0}^{n_p}\cdots\sum_{l_2=0}^{n_2}\dfrac{(-n_p)_{l_p}}{l_p!}\cdots\dfrac{(-n_2)_{l_2}}{l_2!}\dfrac{(x)_{l_2+\cdots+l_p}}{(x+\beta+n_1+2)_{l_2+\cdots+l_p}}
 \\ \times
 \dfrac{(\alpha_2+\beta+n_1+n_2+1)_{l_2+\cdots+l_p}\cdots(\alpha_{p}+\beta+|\vec{n}|+1)_{l_p}
 (\alpha_2+n_2+1)_{l_3+\cdots+l_p}\cdots(\alpha_{p-1}+n_{p-1}+1)_{l_p}
 }{
 (\alpha_2+1)_{l_2+\cdots+l_p}\cdots(\alpha_{p}+1)_{l_p}(\alpha_2+\beta+n_1+n_2+1)_{l_3+\cdots+l_p}\cdots(\alpha_{p-1}+\beta+n_1+\cdots+n_{p-1}+1)_{l_p}
 }
\\
 +\dfrac{n_1x(\beta+2)_{n_1-1}(\alpha_1-x+1)_{n_1-1}(-N)_{n_1}}{(x+\beta+2)_{n_1}(\alpha_1+\beta+|\vec{n}|+1)_{n_1}}\\\times
 (-N+n_1)_{n_2+\cdots+n_p}\prod_{q=2}^p\dfrac{(\alpha_q+1)_{n_q}}{(\alpha_q+\beta+|\vec{n}|+1)_{n_q}}\sum_{l_p=0}^{n_p}\cdots\sum_{l_2=0}^{n_2}\dfrac{(-n_p)_{l_p}}{l_p!}\cdots\dfrac{(-n_2)_{l_2}}{l_2!}\dfrac{(x+1)_{l_2+\cdots+l_p}}{(x+\beta+n_1+2)_{l_2+\cdots+l_p}}
 \\\times
\dfrac{(\alpha_2+\beta+n_1+n_2+1)_{l_2+\cdots+l_p}\cdots(\alpha_{p}+\beta+|\vec{n}|+1)_{l_p}(\alpha_2+n_2+1)_{l_3+\cdots+l_p}\cdots(\alpha_{p-1}+n_{p-1}+1)_{l_p}
}{
(\alpha_2+1)_{l_2+\cdots+l_p}\cdots(\alpha_{p}+1)_{l_p}
(\alpha_2+\beta+n_1+n_2+1)_{l_3+\cdots+l_p}\cdots(\alpha_{p-1}+\beta+n_1+\cdots+n_{p-1}+1)_{l_p}}
 .
\end{multline*}
This last expression can be written through the type II coefficients \eqref{HahnTypeII} with the parameters shifted as follows
\begin{multline*}
\dfrac{(\beta+2)_{n_1}(\alpha_1-x+1)_{n_1}(-N)_{n_1}}{(x+\beta+2)_{n_1}(\alpha_1+\beta+|\vec{n}|+1)_{n_1}}
\sum_{l_p=0}^{n_p}\cdots\sum_{l_2=0}^{n_2}\dfrac{(-N+n_1)_{l_2+\cdots+l_p}(x)_{l_2+\cdots+l_p}}{(x+\beta+n_1+2)_{l_2+\cdots+l_p}}C^{l_2,\ldots,l_p}_{(n_2,\ldots,n_p)}\big(\alpha_2,\ldots,\alpha_p,\beta+n_1,N-n_1\big)
\\
 +\dfrac{n_1x(\beta+2)_{n_1-1}(\alpha_1-x+1)_{n_1-1}(-N)_{n_1}}{(x+\beta+2)_{n_1}(\alpha_1+\beta+|\vec{n}|+1)_{n_1}}
 \\
 \times\underbrace{\sum_{l_p=0}^{n_p}\cdots\sum_{l_2=0}^{n_2}\dfrac{(-N+n_1)_{l_2+\cdots+l_p}(x+1)_{l_2+\cdots+l_p}}{(x+\beta+n_1+2)_{l_2+\cdots+l_p}}
 C^{l_2,\ldots,l_p}_{(n_2,\ldots,n_p)}\big(\alpha_2,\ldots,\alpha_p,\beta+n_1,N-n_1\big).}_{=\dfrac{(\beta+n_1+1)_{|\vec{n}|-n_1}(-N+n_1)_{|\vec{n}|-n_1}}{(x+\beta+n_1+2)_{|\vec{n}|-n_1}}\prod_{q=2}^{p}\dfrac{(\alpha_q-x)_{{n}_q}}{(\alpha_q+\beta+|\vec{n}|+1)_{n_q}}\;\text{by Lemma \ref{lemma1}}}
\end{multline*}
Substituting and reducing, we find the following expression for the left-hand side:
\begin{multline*}
\sum_{l_p=0}^{n_p}
\cdots\sum_{l_1=0}^{n_1}\dfrac{(-N)_{l_1+\cdots+l_p}(x)_{l_1+\cdots+l_p}}{(x+\beta+2)_{l_1+\cdots+l_p}}C^{l_1,\ldots,l_p}_{\vec{n}}
=
\dfrac{(\beta+2)_{n_1}(\alpha_1-x+1)_{n_1}(-N)_{n_1}}{(x+\beta+2)_{n_1}(\alpha_1+\beta+|\vec{n}|+1)_{n_1}}
\\
\times
\sum_{l_p=0}^{n_p}\cdots\sum_{l_2=0}^{n_2}\dfrac{(-N+n_1)_{l_2+\cdots+l_p}(x)_{l_2+\cdots+l_p}}{(x+\beta+n_1+2)_{l_2+\cdots+l_p}}C^{l_2,\ldots,l_p}_{(n_2,\ldots,n_p)}\big(\alpha_2,\ldots,\alpha_p,\beta+n_1,N-n_1\big)
\\
 +\dfrac{n_1x(\beta+2)_{|\vec{n}|-1}(-N)_{|\vec{n}|}}{(\alpha_1-x)(x+\beta+2)_{|\vec{n}|}}{\prod_{q=1}^{p}\dfrac{(\alpha_q-x)_{{n}_q}}{(\alpha_q+\beta+|\vec{n}|+1)_{n_q}}}.
\end{multline*}
Hence, we have found a recursive formula for the multiple sum, applying it \(p\) times we arrive to
\begin{multline*}
\sum_{l_1=0}^{n_1}\cdots\sum_{l_p=0}^{n_p}\dfrac{(-N)_{l_1+\cdots+l_p}(x)_{l_1+\cdots+l_p}}{(x+\beta+2)_{l_1+\cdots+l_p}}
C^{l_1,\ldots,l_p}_{\vec{n}}=\dfrac{(\beta+2)_{|\vec{n}|-1}(-N)_{|\vec{n}|}}{(x+\beta+2)_{|\vec{n}|}\prod_{q=1}^p(\alpha_q+\beta+|\vec{n}|+1)_{n_q}}\\
\times \bigg(
(\beta+|\vec{n}|+1)\prod_{q=1}^p(\alpha_q-x+1)_{n_q}
+x
\underbrace{\sum_{l=1}^p\dfrac{n_l}{\alpha_l-x}\prod_{q=1}^{l-1}(\alpha_q-x+1)_{n_q}\prod_{q=l}^p(\alpha_q-x)_{n_q}}_{=\prod_{q=1}^p(\alpha_q-x+1)_{n_q}-\prod_{q=1}^p(\alpha_q-x)_{n_q}\;\text{by Lemma \ref{lemma3}}}\bigg) .
\end{multline*}
Simplifying we obtain the desired result.
\end{proof}

Equipped with these three lemmas, we are now ready to establish the following result.
\begin{teo}
The Hahn multiple orthogonal polynomials of type I, as expressed in \eqref{HahnTypeI}, and type II, as defined in \eqref{HahnTypeII}, each adhere to their respective recurrence relations, as outlined in \eqref{NNRR}, with respect to the coefficients:
\begin{align}
\label{HahnRecurrence}
b_{\vec{n}}^0(k)=&
\begin{multlined}[t][.8\textwidth]
(\alpha_k+n_k+1)\bigg(
\dfrac{\alpha_k+\beta+n_k+N+2}{\alpha_k+\beta+n_k+|\vec{n}|+2}\prod_{q=1}^p\dfrac{\alpha_k-{\alpha}_q+n_k+1}{\alpha_k-{\alpha}_q+n_k+1-n_q}-1\bigg)\\
+
(\alpha_k+\beta+n_k+|\vec{n}|+1)\sum_{i=1}^p 
\dfrac{(\alpha_i+n_i)(\alpha_i+\beta+n_i+N+1)}{(\alpha_i-\alpha_k-n_k-1+n_i)(\alpha_i+\beta+n_i+|\vec{n}|)_{2}}
\dfrac{\prod_{q=1}^p{(\alpha_i-\alpha_q+n_i)}}{\prod_{q=1,q\neq i}^p(\alpha_i-{\alpha}_q-n_q+n_i)},
\end{multlined}
\\
 b_{\vec{n}}^j=&\begin{multlined}[t][.8\textwidth]
 \dfrac{(N-|\vec{n}|+1)_{j}(\beta+1+|\vec{n}|-j)_{j}}{\prod_{q\in S^{\textsf c}(\pi,j)}({\alpha}_q+\beta+|\vec{n}|-j+n_q)}\prod_{q=1}^p\dfrac{({\alpha}_q+\beta+|\vec{n}|-j+1)_{n_q}}{(\alpha_q+\beta+|\vec{n}|+1)_{n_q}}\\
\begin{aligned}&\times
 \sum_{i\in S(\pi,j)}
\dfrac{(\alpha_i+n_i)(\alpha_i+\beta+n_i+N+1)}{(\alpha_i+\beta+n_i+|\vec{n}|-j)_{j+2}}
\dfrac{\prod_{q=1}^p(\alpha_i-\alpha_q+n_i)}{\prod_{q\in S(\pi,j),q\neq i}(\alpha_i-{\alpha}_q-n_q+n_i)},& j&\in\{1,\ldots,p\}.
\end{aligned}
 \end{multlined}
\end{align}
\end{teo}

\begin{proof}
According \eqref{NextNeighbourRecurrenceIntegralRep}, the recurrence coefficients can be expressed as
\begin{align*}
&\begin{aligned}
 b_{\vec{n}}^0(k)&=\sum_{i=1}^p \sum_{x\in\Delta} 
x\,Q_{\vec{n}}^{({\rm II})}(x) Q^{(i)}_{\vec{n}+\vec{e}_k}w_i(x)\\
&=(\beta+N+1)\underbrace{\sum_{i=1}^p \sum_{x\in\Delta} 
Q_{\vec{n}}^{({\rm II})}(x) Q^{(i)}_{\vec{n}+\vec{e}_k}w_i(x)}_{=1\;\text{by biorthogonality equation \eqref{biorthogonality}}}-\sum_{i=1}^p \sum_{x\in\Delta} 
(\beta+N-x+1)\,Q_{\vec{n}}^{({\rm II})}(x) Q^{(i)}_{\vec{n}+\vec{e}_k}w_i(x),
\\
 b_{\vec{n}}^j&=
 \sum_{i=1}^p \sum_{x\in\Delta} 
x\,Q_{\vec{n}}^{({\rm II})}(x) Q^{(i)}_{\vec{n}-\vec{s}_{j-1}}w_i(x)\\
&=(\beta+N+1)\underbrace{\sum_{i=1}^p \sum_{x\in\Delta} 
Q_{\vec{n}}^{({\rm II})}(x) Q^{(i)}_{\vec{n}-\vec{s}_{j-1}}w_i(x)}_{ =0\;\text{by biorthogonality equation \eqref{biorthogonality}}}-\sum_{i=1}^p \sum_{x\in\Delta} 
(\beta+N-x+1)\,Q_{\vec{n}}^{({\rm II})}(x) Q^{(i)}_{\vec{n}-\vec{s}_{j-1}}w_i(x),& j&\in\{1,\ldots,p\}.
\end{aligned}
\end{align*}

Let's recall that the type II polynomials, see \eqref{HahnTypeII} and \eqref{Coefficients_HahnTypeII}, and the type I polynomials, see \eqref{HahnTypeI} and \eqref{coefHahnTypeI}, can be written respectively as:
\begin{align*}
 Q_{\vec{n}}^{({\rm II})}(x)=\sum_{l_1=0}^{n_1}\cdots\sum_{l_p=0}^{n_p}C^{l_1,\ldots,l_p}_{\vec{n}} (-x)_{l_1+\cdots+l_p}, \quad
 Q_{\vec{n}}^{(i)}(x)=\sum_{l=0}^{n_i-1}C^{(i),l}_{\vec{n}} (x+\alpha_i+1)_l.
\end{align*}
This, along with the orthogonality conditions given by \eqref{DO:I}, allows us to write:
\begin{align}
\label{GeneralCoef}
&\begin{aligned}
 b_{\vec{n}}^0(k)&=\beta+N+1-\sum_{i=1}^p\sum_{l=0}^{n_i-1+\delta_{i,k}}C^{(i),l}_{\vec{n}+\vec{e}_k} \sum_{x\in\Delta} 
Q_{\vec{n}}^{({\rm II})}(x)(x+\alpha_i+1)_{l}(\beta+N-x+1)w_i(x)\\&
 =\begin{multlined}[t][.8\textwidth]\beta+N+1-C^{(k),n_k}_{\vec{n}+\vec{e}_k}\sum_{x\in\Delta}Q_{\vec{n}}^{({\rm II})}(x)(x+\alpha_k+1)_{n_k}(\beta+N-x+1)w_k(x)\\
 -
 \sum_{i=1}^p C_{\vec{n}+\vec{e}_k}^{(i),n_i-1}\sum_{x\in\Delta} 
Q_{\vec{n}}^{({\rm II})}(x)(x+\alpha_i+1)_{n_i-1}(\beta+N-x+1)w_i(x)\end{multlined}
\\
&=\begin{multlined}[t][.8\textwidth]\beta+N+1-C^{(k),n_k}_{\vec{n}+\vec{e}_k}\sum_{l_1=0}^{n_1}\cdots\sum_{l_p=0}^{n_p}C^{l_1,\ldots,l_p}_{\vec{n}}\sum_{x\in\Delta}(-x)_{l_1+\cdots+l_p}(x+\alpha_k+1)_{n_k}(\beta+N-x+1)w_k(x)\\-\sum_{i=1}^p C_{\vec{n}+\vec{e}_k}^{(i),n_i-1} 
\sum_{l_1=0}^{n_1}\cdots\sum_{l_p=0}^{n_p}C^{l_1,\ldots,l_p}_{\vec{n}}\sum_{x\in\Delta} (-x)_{l_1+\cdots+l_p}(x+\alpha_i+1)_{n_i-1}(\beta+N-x+1)w_i(x),
\end{multlined}
\end{aligned}\\
 & \begin{aligned}
 b_{\vec{n}}^j&=-\sum_{i=1}^p\sum_{l=0}^{\deg{A_{\vec{n}-\vec{s}_{j-1}}^{(i)}} }C^{(i),l}_{\vec{n}-\vec{s}_{j-1}} \sum_{x\in\Delta} Q_{\vec{n}}^{({\rm II})}(x)
(x+\alpha_i+1)_{l}(\beta+N-x+1)w_i(x)\\
&=-\sum_{i\in S(\pi,j)} C^{(i),n_i-1}_{\vec{n}-\vec{s}_{j-1}} \sum_{x\in\Delta} Q_{\vec{n}}^{({\rm II})}(x)
(x+\alpha_i+1)_{n_i-1}(\beta+N-x+1)w_i(x)\\
&=-\sum_{i\in S(\pi,j)} C^{(i),n_i-1}_{\vec{n}-\vec{s}_{j-1}}
\sum_{l_1=0}^{n_1}\cdots\sum_{l_p=0}^{n_p}C^{l_1,\ldots,l_p}_{\vec{n}}\sum_{x\in\Delta}
(-x)_{l_1+\cdots+l_p}(x+\alpha_i+1)_{n_i-1}(\beta+N-x+1)w_i(x),
\end{aligned}
\end{align}
\( j \in\{1,\ldots,p\} \),
where we have used that
\begin{align*}
 \deg{A_{\vec{n}-\vec{s}_{j-1}}^{(i)}} =\begin{cases}
 n_i-2\;\text{if}\;i\not\in S(\pi,j),
 \\
 n_i-1\;\text{if}\; i\in S(\pi,j).
 \end{cases}
\end{align*}

Let's now calculate the discrete integral
\begin{align*}
 \sum_{x\in\Delta}
(-x)_{l_1+\cdots+l_p}(x+\alpha_i+1)_{n}(\beta+N-x+1)w_i(x).
\end{align*}
By replacing the weight functions from \eqref{WeightsHahn} and simplifying, we get:
\begin{multline*}
 \sum_{x=0}^N
\dfrac{(x+\alpha_i+1)_{n}\Gamma(x+\alpha_i+1)(-x)_{l_1+\cdots+l_p}}{\Gamma(x+1)\Gamma(\alpha_i+1)}\dfrac{(\beta+N-x+1)\Gamma(\beta+N-x+1)}{\Gamma(\beta+1)\Gamma(N-x+1)}\\
\begin{aligned}
 &=(-1)^{l_1+\cdots+l_p}\dfrac{(\beta+1)(\alpha_i+1)_{n+l_1+\cdots+l_p}}{(N-l_1-\cdots-l_p)!}\underbrace{\sum_{x=0}^{N-l_1-\cdots-l_p}\dbinom{N-l_1-\cdots-l_p}{x}(\alpha_i+n+l_1+\cdots+l_p+1)_{x}(\beta+2)_{N-l_1\cdots-l_p-x}}_{\text{\( =(\alpha_i+\beta+n+l_1+\cdots+l_p+3)_{N-l_1-\cdots-l_p} \) by Chu--Vandermonde formula \eqref{eq:Chu--Vandermonde}}}\\
&=\dfrac{(\beta+1)(\alpha_i+1)_{n}(\alpha_i+\beta+n+3)_{N}}{N!}{\dfrac{(-N)_{l_1+\cdots+l_p}(\alpha_i+n+1)_{l_1+\cdots+l_p}}{(\alpha_i+\beta+n+3)_{l_1+\cdots+l_p}}}.
\end{aligned}
\end{multline*}
Replacing this expression into \eqref{GeneralCoef} we get
\begin{multline*}
b_{\vec{n}}^0(k)=
\beta+N+1\\-\dfrac{(\beta+1)(\alpha_k+1)_{n_k}(\alpha_k+\beta+n_k+3)_{N}}{N!}C^{(k),n_k}_{\vec{n}+\vec{e}_k}\sum_{l_1=0}^{n_1}\cdots\sum_{l_p=0}^{n_p}C^{l_1,\ldots,l_p}_{\vec{n}}{\dfrac{(-N)_{l_1+\cdots+l_p}(\alpha_k+n_k+1)_{l_1+\cdots+l_p}}{(\alpha_k+\beta+n_k+3)_{l_1+\cdots+l_p}}} 
\\
-\dfrac{(\beta+1)}{N!}\sum_{i=1}^p (\alpha_i+1)_{n_i-1}(\alpha_i+\beta+n_i+2)_{N}\,C_{\vec{n}+\vec{e}_k}^{(i),n_i-1} 
\sum_{l_1=0}^{n_1}\cdots\sum_{l_p=0}^{n_p}C^{l_1,\ldots,l_p}_{\vec{n}}{\dfrac{(-N)_{l_1+\cdots+l_p}(\alpha_i+n_i)_{l_1+\cdots+l_p}}{(\alpha_i+\beta+n_i+2)_{l_1+\cdots+l_p}}}
\end{multline*}
and, for $j\in\{1,\ldots,p\}$, that
\[
 b_{\vec{n}}^j=-\dfrac{(\beta+1)}{N!}\sum_{i\in S(\pi,j)} (\alpha_i+1)_{n_i-1}(\alpha_i+\beta+n_i+2)_{N}\,C^{(i),n_i-1}_{\vec{n}-\vec{s}_{j-1}}
\sum_{l_1=0}^{n_1}\cdots\sum_{l_p=0}^{n_p}C^{l_1,\ldots,l_p}_{\vec{n}}{\dfrac{(-N)_{l_1+\cdots+l_p}(\alpha_i+n_i)_{l_1+\cdots+l_p}}{(\alpha_i+\beta+n_i+2)_{l_1+\cdots+l_p}}}.\]
 Lemma \ref{lemma2} allows us to simplify the sums labeled by \( l_1,\ldots,l_p\):
\begin{align*}
b_{\vec{n}}^0(k)=&\begin{multlined}[t][.92\textwidth]
 \beta+N+1+\dfrac{(-N)_{|\vec{n}|}(\beta+1)_{|\vec{n}|}}{N!\prod_{q=1}^p(\alpha_q+\beta+|\vec{n}|+1)_{n_q}}
\dfrac{(\alpha_k+1)_{n_k}(\alpha_k+\beta+n_k+3)_{N}}{(\alpha_k+\beta+n_k+3)_{|\vec{n}|}}
C^{(k),n_k}_{\vec{n}+\vec{e}_k}\\
\times\bigg(
(\alpha_k+n_k+1)\prod_{q=1}^p(\alpha_q-\alpha_k-n_k-1)_{n_q}-
(\alpha_k+\beta+n_k+|\vec{n}|+2)\prod_{q=1}^p(\alpha_q-\alpha_k-n_k)_{n_q}\bigg)\\
+\dfrac{(-N)_{|\vec{n}|}(\beta+1)_{|\vec{n}|}}{N!\prod_{q=1}^p(\alpha_q+\beta+|\vec{n}|+1)_{n_q}}
\sum_{i=1}^p 
\dfrac{(\alpha_i+1)_{n_i-1}(\alpha_i+\beta+n_i+2)_{N}}{(\alpha_i+\beta+n_i+2)_{|\vec{n}|}}\,C_{\vec{n}+\vec{e}_k}^{(i),n_i-1} \\
\times\bigg(
(\alpha_i+n_i)\prod_{q=1}^p(\alpha_q-\alpha_i-n_i)_{n_q}-(\alpha_i+\beta+n_i+|\vec{n}|+1)\underbrace{\prod_{q=1}^p(\alpha_q-\alpha_i-n_i+1)_{n_q}}_{=0\;\text{because of the \(i\)-th factor}}
\bigg),
\end{multlined}
\\
 b_{\vec{n}}^j=&\begin{multlined}[t][.85\textwidth]\dfrac{(-N)_{|\vec{n}|}(\beta+1)_{|\vec{n}|}}{N!\prod_{q=1}^p(\alpha_q+\beta+|\vec{n}|+1)_{n_q}}\sum_{i\in S(\pi,j)}
\dfrac{(\alpha_i+1)_{n_i-1}(\alpha_i+\beta+n_i+2)_{N}}{(\alpha_i+\beta+n_i+2)_{|\vec{n}|}}\,C^{(i),n_i-1}_{\vec{n}-\vec{s}_{j-1}}\\
\begin{aligned}
& \times\bigg(
(\alpha_i+n_i)\prod_{q=1}^p(\alpha_q-\alpha_i-n_i)_{n_q}-
(\alpha_i+\beta+n_i+|\vec{n}|+1)\underbrace{\prod_{q=1}^p(\alpha_q-\alpha_i-n_i+1)_{n_q}}_{=0\;\text{because of the \( i\)-th factor}}
\bigg),& j&\in\{1,\ldots,p\}.
\end{aligned}
\end{multlined}
\end{align*}
Finally, we replace the type I coefficients using \eqref{coefHahnTypeI}. We have
\begin{align*}
C_{\vec{n}+\vec{e}_k}^{(k),n_k}=&\dfrac{N!(\alpha_k+\beta+n_k+2)_{|\vec{n}|}}{(-N)_{|\vec{n}|}(\beta+1)_{|\vec{n}|}(\alpha_k+1)_{n_k}(\alpha_k+\beta+n_k+2)_{N}}\prod_{q=1}^p\dfrac{(\alpha_k-{\alpha}_q+n_k+1)({\alpha}_q+\beta+|\vec{n}|+1)_{n_q}}{(\alpha_k-{\alpha}_q+n_k+1-n_q)(\alpha_q-{\alpha}_k-n_k-1)_{n_q}},
 \\
C_{\vec{n}+\vec{e}_k}^{(i),n_i-1}=&\begin{multlined}[t][.9\textwidth]
 -\dfrac{N!}{(-N)_{|\vec{n}|}(\beta+1)_{|\vec{n}|}(\alpha_i+1)_{n_i-1}(\alpha_i+\beta+|\vec{n}|+n_i)_{N+1-|\vec{n}|}}\\
 \times \dfrac{(\alpha_k+\beta+|\vec{n}|+n_k+1)}{(\alpha_k-\alpha_i-n_i+n_k+1)\prod_{q=1,q\neq i}^p(\alpha_i-{\alpha}_q+n_i-n_q)}
\prod_{q=1}^p\dfrac{(\alpha_i-\alpha_q+n_i)({\alpha}_q+\beta+|\vec{n}|+1)_{n_q}}{(\alpha_q-\alpha_i-n_i)_{n_q}},
\end{multlined}
\\
 C_{\vec{n}-\vec{s}_{j-1}}^{(i),n_i-1}=&\begin{multlined}[t][.9\textwidth]
 \dfrac{(-1)^{j}N!}{(-N)_{|\vec{n}|-j}(\beta+1)_{|\vec{n}|-j}(\alpha_i+1)_{n_i-1}(\alpha_i+\beta+|\vec{n}|-j+n_i)_{N+1-|\vec{n}|+j}}\\\times
\dfrac{\prod_{q\in S(\pi,j)}({\alpha}_q+\beta+|\vec{n}|-j+1)_{n_q}\prod_{q\in S^{\textsf c}(\pi,j)}({\alpha}_q+\beta+|\vec{n}|-j+1)_{n_q-1}}{\prod_{q\in S(\pi,j),q\neq i}(\alpha_i-{\alpha}_q+n_i-n_q)}
\prod_{q=1}^p\dfrac{\alpha_i-\alpha_q+n_i}{(\alpha_q-{\alpha}_i-n_i)_{n_q}}.
 \end{multlined}
\end{align*}
Simplifying, we arrive at the expressions given in \eqref{HahnRecurrence}.
\end{proof}

\section{Meixner of the Second Kind 
}

For the multiple Meixner polynomials of the second kind, the weight functions are defined as
\begin{equation}
 \label{WeightsM2}
\begin{aligned}
 w_i^{(M2)}(x;\beta_i,c) &\coloneq \dfrac{\Gamma(\beta_i+x)}{\Gamma(\beta_i)\Gamma(x+1)} c^x, & i &\in \{1,\ldots,p\}, & \Delta &= \mathbb{N}_0,
\end{aligned}
\end{equation}
with \( \beta_1, \ldots, \beta_p > 0 \), \( 0 < c < 1 \) and, to ensure an AT system, \( \beta_i - \beta_j \not\in \mathbb{Z} \) for \( i \neq j \). We will denote \(\vec{\beta}\coloneqq(\beta_1,\ldots,\beta_p)\). As suggested in \cite{AskeyII}, we find a limiting relation of the following form. 

\begin{pro}
 The following limiting relations between the Hahn \eqref{WeightsHahn} and the Meixner of the second kind weights hold
\begin{equation}
 \label{H->M2:W}
 w_i^{(M2)}(x;{\beta_i},c)=\lim_{N\to\infty} \sqrt{2\pi N} c^N (1-c)^{\frac{1-c}{c}N+\frac{1}{2}}\,w_i\left(x;
 {\beta_i}-1,\frac{1-c}{c}N,N\right),\quad i\in\{1,\ldots,p\}.
 \end{equation}
\end{pro}
\begin{proof}
 Note that
\[w_i\left(x;\beta_i-1,\frac{1-c}{c}N,N\right) = \frac{\Gamma(x+\beta_i)}{\Gamma(\beta_i)\Gamma(x+1)} \frac{\Gamma\left(\frac{N}{c}-x+1\right)}{\Gamma\left(\frac{1-c}{c}N+1\right)\Gamma(N-x+1)}.
\]
 An application of Stirling's formula \eqref{stirling}, gives
\[ \begin{aligned}
 \frac{\Gamma\left(\frac{N}{c}-x+1\right)}{\Gamma\left(\frac{1-c}{c}N+1\right)\Gamma(N-x+1)} &\sim 
 \dfrac{c^x}{\sqrt{2\pi N} c^N (1-c)^{\frac{1-c}{c}N+\frac{1}{2}}},& N&\to\infty,
\end{aligned}\]
 which then readily leads to the desired result.
\end{proof}

\begin{coro} \label{Cor:H->M2}
 For the linear forms $M_{2:\vec{n}}^{({\rm I})}$, type I polynomials $M_{2:\vec{n}}^{(i)}$, type II polynomials $M_{2:\vec{n}}^{({\rm II})}$ and recurrence coefficients $b_{\vec{n}}^{(M2),j}$; the following respective limiting relations from Hahn hold:
 \begin{subequations}
 \label{H->M2}
\begin{align}
\label{H->M2:MOP}
 \quad M_{2:\vec{n}}^{(\ast)}(x;\vec{\beta},c)= \lim_{N\to\infty} \kappa^{(\ast)}_N Q_{\vec{n}}^{(\ast)}\left(x;\vec{\beta}-\vec{1}_p,\frac{1-c}{c}N,N\right),\\
\label{H->M2:R} b_{\vec{n}}^{(M2),j}(\vec{\beta}, c) = \lim_{N \rightarrow \infty} b_{\vec{n}}^j \left( \vec{\beta}-\vec{1}_p, \frac{1 - c}{c} N, N \right),
\end{align}
 \end{subequations}
 where
 $$\kappa^{({\rm I})}_N = \kappa^{({\rm II})}_N = 1,\quad \kappa^{(i)}_N = \frac{1}{\sqrt{2\pi N} c^N (1-c)^{\frac{1-c}{c}N+\frac{1}{2}}}.$$
\end{coro}
\begin{rem}
 The limiting relation for the type II polynomials was already deduced in \cite{AskeyII}.
\end{rem}

\subsection{Hypergeometric Representations}

An expression for the monic type II polynomials can be derived from the one for the Hahn polynomials in \eqref{HahnTypeII} using Corollary \ref{Cor:H->M2}.
These polynomials are given by, cf. \cite{AskeyII}:
\begin{multline}
\label{M2:II}
 M_{2:\vec{n}}^{({\rm II})}(x;\vec{\beta}, c) 
 = \left( \dfrac{c}{c-1} \right)^{|\vec{n}|} \prod_{j=1}^p (\beta_j)_{n_j} \sum_{l_1=0}^{n_1} \cdots \sum_{l_p=0}^{n_p} \prod_{j=1}^p \dfrac{(-n_j)_{l_j}}{l_j!} \dfrac{\prod\limits_{i=1}^{p-1}(\beta_i+n_i)_{\sum_{j=i+1}^p l_j}}{\prod\limits_{i=1}^p(\beta_i)_{\sum_{j=i}^{p}l_j}} \left( \dfrac{c-1}{c} \right)^{\sum_{j=i}^{p}l_j}
 (-x)_{\sum_{j=i}^{p}l_j}
\end{multline}
From \eqref{HahnTypeIIWeighted}, the following alternative expression can also be obtained by applying the same limit.
\begin{pro}
The type II Meixner polynomials of the second kind can be written as
\begin{align}
 \label{M2:II_Weighted}
 M_{2:\vec{n}}^{({\rm II})}(x;\vec{\beta},c)=&
 \left(\dfrac{c}{c-1}\right)^{|\vec{n}|}\prod_{j=1}^p{(\beta_j)_{n_j}} \dfrac{1}{c^x} \,\pFq{p+1}{p}{-x,\vec{\beta}+\vec{n}}{\vec{\beta}}{1-c}.
\end{align}
\end{pro}

Below, we will derive a hypergeometric representation for the type I Meixner polynomials of the second kind. For $p=2$, an alternative representation was found in \cite[Proposition 5.1]{HahnI}.

\begin{teo}
The type I Meixner polynomials of the second kind are:
\begin{align}
 \label{M2:I}
 M_{2:\vec{n}}^{(i)}&(x;\vec{\beta},c)\\
 &=
 \dfrac{(-1)^{|\vec{n}|-1}(1-c)^{\beta_i+|\vec{n}|-1}}{c^{|\vec{n}|-1}(n_i-1)!\prod_{k=1,k\neq i}^p({\beta}_k-{\beta}_i)_{{n}_k}} \pFq{p+1}{p}{-n_i+1,(\beta_i+1)\vec{1}_{p-1}-\vec{\beta}^{*i}-\vec{n}^{*i},x+\beta_i}{(\beta_i+1)\vec{1}_{p-1}-\vec{\beta}^{*i},\beta_i}{1-c},
\end{align}
for \( i \in \{1, \ldots, p\} \).
\end{teo}

\begin{proof}
By substituting into the expression \eqref{HahnTypeI} as indicated in Corollary \ref{Cor:H->M2}, we obtain:
\begin{multline*}
 M_{2:\vec{n}}^{(i)}(x; \vec{\beta},c)\\
 = \lim_{N\rightarrow\infty}\dfrac{1}{\sqrt{2\pi N}c^{N}(1-c)^{\frac{1-c}{c}N+\frac{1}{2}}}
 \dfrac{(-1)^{|\vec{n}|-1}}{(n_i-1)! \prod_{k=1,k\neq i}^p({\beta}_k-{\beta}_i)_{{n}_k}}\dfrac{(N-|\vec{n}|+1)!\prod_{k=1}^p\left({\beta}_k+\frac{1-c}{c}N+|\vec{n}|-1\right)_{n_k}}{\left(\frac{1-c}{c}N+1\right)_{|\vec{n}|-1}\left(\beta_i+\frac{1-c}{c}N+|\vec{n}|-1\right)_{N+2-|\vec{n}|}}\\
 \times\pFq{p+2}{p+1}{-n_i+1,\beta_i+\frac{1-c}{c}N+|\vec{n}|-1,(\beta_i+1)\vec{1}_{p-1}-\vec{\beta}^{*i}-\vec{n}^{*i},x+\beta_i}{\beta_i+\frac{1}{c}N+1,(\beta_i+1)\vec{1}_{p-1}-\vec{\beta}^{*i},\beta_i}{1} .
\end{multline*}
By applying again Stirling's formula \eqref{stirling}, we find that
\[\begin{aligned}
 \dfrac{(N - |\vec{n}| + 1)! \prod_{k=1}^p \left( \beta_k + \frac{1-c}{c}N + |\vec{n}| - 1 \right)_{n_k}}{\left( \frac{1-c}{c}N + 1 \right)_{|\vec{n}| - 1} \left( \beta_i + \frac{1-c}{c}N + |\vec{n}| - 1 \right)_{N + 2 - |\vec{n}|}} \sim \sqrt{2 \pi N} (1 - c)^{\frac{1-c}{c}N + \beta_i + |\vec{n}| - \frac{1}{2}} c^{N - |\vec{n}| + 1},\quad N\rightarrow\infty.
\end{aligned}\]
Using this result, the limit can be applied straightforwardly to obtain the expression given by \eqref{M2:I}.
\end{proof}

\subsection{Integral Representations}


The limits in Corollary \ref{Cor:H->M2} that lead to the Meixner polynomials of the second kind will enable us to derive integral representations, beginning with those established for the multiple Hahn polynomials in Theorems \ref{HI_IR} and \ref{HII_IR}.

\begin{teo}
\label{M2I_IR}
Let $\Sigma$ be a clockwise contour in $\{t\in \mathbb{C} \mid \operatorname{Re}(t) > -1\}$ enclosing \(\bigcup_{i=1}^p[\beta_i-1,n_i+\beta_i-2]\) exactly once. Then, the type I linear forms are given by
 \[
 M_{2:\vec{n}}^{({\rm I})}(x;\vec{\beta},c) = \frac{(c-1)^{\sz{n}}}{c^{\sz{n}-1}} \frac{c^x}{\Gamma(x+1)} \bigintsss_\Sigma \frac{(1-c)^{t}}{\Gamma(t+1)} \frac{\Gamma(x+t+1)}{\prod_{j=1}^p (\beta_j-1-t)_{n_j}} \frac{\d t }{2 \pi \ii }.
 \]
In particular, the type I polynomials are given by
\[ M_{2:\vec{n}}^{(i)}(x;\vec{\beta},c) = \frac{(c-1)^{\sz{n}}}{c^{\sz{n}-1}} \frac{\Gamma(\beta_i)}{\Gamma(\beta_i+x)} \bigintsss_{\Sigma_i} \frac{(1-c)^{t}}{\Gamma(t+1)} \frac{\Gamma(x+t+1)}{\prod_{j=1}^p (\beta_j-1-t)_{n_j}} \frac{\d t }{2 \pi \ii },
 \]
where \(\Sigma_i\) is a clockwise contour in $\{t\in \mathbb{C} \mid \operatorname{Re}(t) > -1\}$ enclosing \(\cup_{k=0}^{n_i-1} \{\beta_i+k-1\}\) exactly once without enclosing any of the other points in \(\cup_{j=1}^p \cup_{k=0}^{n_j-1} \{\beta_j+k-1\}\).

\end{teo}
\begin{proof}
Consider the integral representation for the linear form \( Q_{\vec{n}}^{({\rm I})}\left(x;\vec{\beta}-\vec{1}_p, \frac{1-c}{c}N, N\right) \) provided in Theorem~\ref{HI_IR}. This representation already uses the contour \( \Sigma \) described above. We can use \eqref{GFB} to obtain the asymptotics for the $N$-dependent parts of the integrand and prefactor:
\begin{multline}
 \prod_{j=1}^p \left(\beta_j+\frac{1-c}{c}N+\sz{n}-1\right)_{n_j} \frac{\Gamma\left(t+\frac{1-c}{c}N+\sz{n}\right)}{\Gamma\left(\frac{1-c}{c}N+\sz{n}\right)} \frac{\Gamma(N-\sz{n}+2)}{\Gamma(N-x+1)} \frac{\Gamma\left(\frac{N}{c}-x+1\right)}{\Gamma\left(t+\frac{N}{c}+2\right)} \\
 \begin{aligned}&\sim \left(\frac{1-c}{c}N\right)^{\sz{n}+t} N^{-\sz{n}+1+x} \left(\frac{N}{c}\right)^{-x-t-1}= \frac{(1-c)^{\sz{n}+t}}{c^{\sz{n}-1}} c^x,\quad N\to\infty.
\end{aligned}
\end{multline}
Combining this with limit \eqref{H->M2:MOP} yields the desired formula for the type I linear forms. Once this result is established, one can note that only certain poles contribute to a specific type I polynomial, as the remaining poles correspond to different weight functions.
\end{proof}

\begin{teo}
 \label{M2II_IR}
Let \(\mathcal{C}\) be a counterclockwise contour enclosing \((- \infty, 0]\) exactly once. Then, for \(x \in \mathbb{N}_0\), the type II polynomials are given by
\[
M_{2:\vec{n}}^{({\rm II})}(x;\vec{\beta},c) = \left(\frac{c}{c-1}\right)^{\sz{n}} \frac{\Gamma(x+1)}{c^x} \bigintsss_{\mathcal{C}} \frac{\Gamma(s)}{(1-c)^{s}} \frac{\prod_{j=1}^p(\beta_j-s)_{n_j}}{\Gamma(x+s+1)} \frac{\d s}{2 \pi \ii }.
\]
\end{teo}

\begin{figure}[htbp]
	\begin{tikzpicture}[arrowmark/.style 2 args={decoration={markings,mark=at position #1 with \arrow{#2}}},scale=1]
		\begin{axis}[axis lines=middle,axis equal,grid=both,xmin=-6, xmax=2,ymin=-2, ymax=2,
			xticklabel,yticklabel,disabledatascaling,xlabel=$x$,ylabel=$y$,every axis x label/.style={
				at={(ticklabel* cs:1)},
				anchor=south west,
			},
			every axis y label/.style={
				at={(ticklabel* cs:1.0)},
				anchor=south west,
			},grid style={line width=.1pt, draw=Bittersweet!10},
			major grid style={line width=.2pt,draw=Bittersweet!50},
			minor tick num=4,
			enlargelimits={abs=2},
			axis line style={latex'-latex'},Bittersweet] 
			\node[anchor = north east,Bittersweet] at (axis cs: 5,5) {$\mathbb C$} ;
			\draw[ DarkSlateBlue,ultra thick, decoration={markings, mark=at position 0.15 with {\arrow[ultra thick]{Stealth}}},
			postaction={decorate},decoration={markings, mark=at position 0.45 with {\arrow[ultra thick]{Stealth}}},
			postaction={decorate},decoration={markings, mark=at position 0.65 with {\arrow[ultra thick]{Stealth}}},
			postaction={decorate},decoration={markings, mark=at position 0.95 with {\arrow[ultra thick]{Stealth}}},
			postaction={decorate}] (-5,-2) -- (3.5,-2)-- (3.5,2)-- (-5,2)--cycle;
			\draw [fill,black] (axis cs:{0,0}) circle [radius=1.5pt] node[Black,above right]{$0$};
			\draw [fill,black] (axis cs:{-4,0}) circle [radius=1.5pt] node[Black,above ]{$-x$};
			\draw [fill,black] (axis cs:{-5,0}) circle [radius=1.5pt] node[Black,above left]{$-M$};
			\draw [fill,black] (axis cs:{-4,-3.5}) node[Black,above left]{$M\to \infty$};

			\draw[ultra thick, DarkSlateBlue] (-7,3.4) -- (-6,3.4) node at (-5.5,3.4) {$\tilde{\mathcal C}$};
			\draw[ultra thick, DarkSlateBlue,dotted] (-7,4.1) -- (-6,4.1) node at (-5.5,4.1) {$\mathcal C$};
			\draw [DarkSlateBlue,ultra thick,dotted,decoration={markings, mark=at position 0.15 with {\arrow[ultra thick]{Stealth}}},
			postaction={decorate},decoration={markings, mark=at position 0.47 with {\arrow[ultra thick]{Stealth}}},
			postaction={decorate},decoration={markings, mark=at position 0.65 with {\arrow[ultra thick]{Stealth}}},
			postaction={decorate}] (-8,-2.1)--(3.6,-2.1)--(3.6,2.1)--(-8,2.1) ;

			\draw[DarkSlateBlue, thick,->] (-5,-1)--(-6.5,-1) ;
			\draw[DarkSlateBlue, thick,->] (-5,1)--(-6.5,1) ;
			
			
		\end{axis}
		\draw (3.5,-0.4) ; 
	\end{tikzpicture}
\caption{Integration contours for type II Meixner of the second kind}
\label{fig:1}
\end{figure}
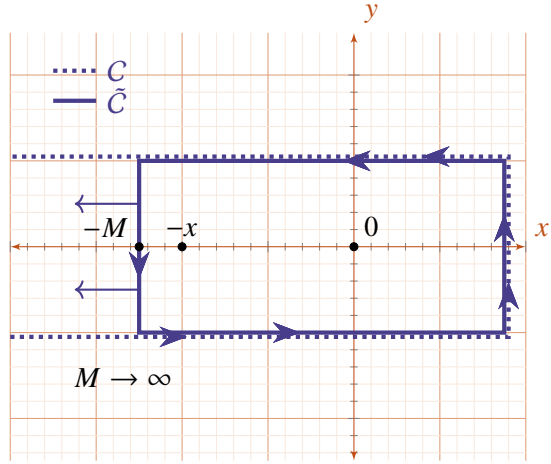

\begin{proof}
Consider the integral representation for \( Q_{\vec{n}}^{({\rm II})}\left(x;\vec{\beta}-\vec{1}_p, \frac{1-c}{c}N, N\right) \) provided in Theorem~\ref{HII_IR}. This representation involves a contour that can be deformed to a contour \( \tilde{\mathcal{C}} \) enclosing \([-x, 0]\) once, since for \( x \in \mathbb{N}_0 \), the integrand only has poles at \( s \in \{0, \ldots, -x\} \). We can use \eqref{GFB} to obtain the asymptotics for the $N$-dependent parts of the integrand and prefactor:
\begin{multline}
 \frac{1}{\prod_{j=1}^p \left(\beta_j+\frac{1-c}{c}N+\sz{n}\right)_{n_j}} \frac{\Gamma\left (\sz{n}+\frac{1-c}{c}N+1\right)}{\Gamma\left(s+\frac{1-c}{c}N+\sz{n}+1\right)} \frac{\Gamma(N-x+1)}{\Gamma(N-\sz{n}+1)} \frac{\Gamma\left(s+\frac{N}{c}+1\right)}{\Gamma\left(\frac{N}{c}-x+1\right)} \\
 \begin{aligned} &\sim \left(\frac{1-c}{c}N\right)^{-\sz{n}-s} N^{\sz{n}-x} \left(\frac{N}{c}\right)^{s+x}= \frac{c^{\sz{n}}}{(1-c)^{\sz{n}+s}}\frac{1}{c^x},\quad N\to\infty.
\end{aligned}
\end{multline}

Combining this with limit \eqref{H->M2:MOP} leads to the specified form of the integrand and prefactor. We then need to shift the intersection point of \(\tilde{\mathcal{C}}\) with the negative real line to \(- \infty\) to deform it into the contour \(\mathcal{C}\) described above (see Figure \ref{fig:1}). This can be done because, for \(s = -M + \ii t\), where \(t \in [-T, T]\) and \(M \to \infty\), the integrand decays sufficiently fast. Specifically, we have:
 \[
 \left| \frac{\Gamma(-M+\ii t)}{(1-c)^{-M+\ii t}} \frac{\prod_{j=1}^p (\beta_j+M-\ii t)_{n_j}}{\Gamma(x-M+\ii t+1)}\right| = \BO\left((1-c)^{M} M^{\sz{n}-x-1}\right),\quad M\to\infty.
\]
\end{proof}

\subsection{Recurrence Coefficients}

Finally, we present the explicit expressions for the recurrence coefficients:

\begin{teo}
The Meixner polynomials of the second kind, both type I \eqref{M2:I} and type II \eqref{M2:II}, each follow their respective recurrence relations \eqref{NNRR} with coefficients
\begin{align}
\label{M2:RR}
\quad
b_{\vec{n}}^{(M2),0}(k)&=
\begin{multlined}[t][.825\textwidth]
 (\beta_k+n_k)\bigg(
\dfrac{1}{1-c}\prod_{q=1}^p\dfrac{\beta_k-{\beta}_q+n_k+1}{\beta_k-{\beta}_q+n_k+1-n_q}-1\bigg)\\
+\dfrac{1}{1-c}
\sum_{i=1}^p 
\dfrac{(\beta_i+n_i-1)}{(\beta_i-\beta_k-n_k-1+n_i)}
\dfrac{\prod_{q=1}^p{(\beta_i-\beta_q+n_i)}}{\prod_{q=1,q\neq i}^p(\beta_i-{\beta}_q-n_q+n_i)},
\end{multlined}\\
 b_{\vec{n}}^{(M2),j}&=
 \begin{aligned}&\dfrac{c^j}{(1-c)^{j+1}}\sum_{i\in S(\pi,j)}
{(\beta_i+n_i-1)}
\dfrac{\prod_{q=1}^p(\beta_i-\beta_q+n_i)}{\prod_{q\in S(\pi,j),q\neq i}(\beta_i-{\beta}_q-n_q+n_i)},\quad j\in\{1,\ldots,p\}.
\end{aligned}
\end{align}
\end{teo}

\begin{proof}
It is straightforward to obtain these coefficients from the Hahn ones \eqref{HahnRecurrence} via limit \eqref{H->M2:R}.
\end{proof}

\section{Meixner of the First Kind 
}

For the multiple Meixner polynomials of the first kind, the weight functions are defined as
\begin{equation}
 \label{WeightsM1}
\begin{aligned}
 w_i^{(M1)}(x;\beta,c_i) &\coloneq \dfrac{\Gamma(\beta + x)}{\Gamma(\beta) \Gamma(x + 1)} c_i^x, & i &\in \{1, \ldots, p\}, & \Delta &= \mathbb{N}_0,
 \end{aligned}
\end{equation}
with \( \beta > 0 \), \( 0 < c_1, \ldots, c_p < 1 \) and, in order to ensure an AT system, \( c_i \neq c_j \) for \( i \neq j \). We will denote \(\vec{c}\coloneq(c_1,\cdots,c_p)\).

\begin{pro}
The following limiting relations between the Hahn \eqref{WeightsHahn} and the Meixner of the first kind weights hold 
\begin{equation}
 \label{H->M1:W} 
w_i^{(M1)}(x;\beta,c_i)=\lim_{N\to\infty} \sqrt{2\pi N} c_i^{N} (1-c_i)^{\frac{1-c_i}{c_i}N+\frac{1}{2}} w_i\left(N-x;\frac{1-c_i}{c_i}N,\beta-1,N\right),\quad i\in\{1,\ldots,p\} .
\end{equation} 
\end{pro}
\begin{proof}
We have
 \[w_i\left(N-x;\frac{1-c_i}{c_i}N,\beta-1,N\right) = \frac{\Gamma\left(\frac{N}{c_i}-x+1\right)}{\Gamma\left(\frac{1-c_i}{c_i}N+1\right)\Gamma(N-x+1)} \frac{\Gamma(x+\beta)}{\Gamma(\beta)\Gamma(x+1)}
 \]
and, as in the previous section, Stirling's formula \eqref{stirling} gives
 \begin{equation*}
\begin{aligned}
 \frac{\Gamma\left(\frac{N}{c_i}-x+1\right)}{\Gamma\left(\frac{1-c_i}{c_i}N+1\right)\Gamma(N-x+1)} &\sim 
 \dfrac{c_i^x}{\sqrt{2\pi N} c_i^N (1-c_i)^{\frac{1-c_i}{c_i}N+\frac{1}{2}}},& N&\to\infty.
\end{aligned}
\end{equation*}
Thus, the desired result follows.
\end{proof}

\begin{coro} 
\label{Cor:H->M1}
For the linear forms \( M_{1:\vec{n}}^{({\rm I})} \), type I polynomials \( M_{1:\vec{n}}^{(i)} \), type II polynomials \( M^{({\rm II})}_{1:\vec{n}} \) and recurrence coefficients $b_{\vec{n}}^{(M1),j}$; the following respective limiting relations from Hahn hold:
\begin{subequations}
\label{H->M1}
 \begin{align}
 \label{H->M1:MOP}
 \quad M_{1:\vec{n}}^{(\ast)}(x;\beta,\vec{c})&=\lim_{N\to\infty} \kappa^{(\ast)}_N Q_{\vec{n}}^{(\ast)}\left(N-x;\left\{\frac{1-c_i}{c_i}N\right\}_{i=1}^p,\beta-1,N\right), \\
\label{H->M1:R}
\quad b^{(M1),j}_{\vec{n}}(\beta,\vec{c})&=
 \begin{aligned}&
 \lim_{N\to\infty}(-1)^{j+1}\left(b^j_{\vec{n}}
 \left(\left\{\frac{1-c_i}{c_i}N\right\}_{i=1}^p,\beta-1,N
 \right)
 -N\delta_{j,0}
 \right),
\end{aligned}
 \end{align}
\end{subequations}
where
$$ \kappa^{({\rm I})}_N = (-1)^{\sz{n}-1},\quad \kappa^{({\rm II})}_N = (-1)^{\sz{n}},\quad \kappa^{(i)}_N = \frac{(-1)^{\sz{n}-1}}{\sqrt{2\pi N} c_i^{N} (1-c_i)^{\frac{1-c_i}{c_i}N+\frac{1}{2}}}. $$
\end{coro}
\begin{rem}
 The limiting relation for the type II polynomials was already deduced in \cite{AskeyII}.
\end{rem}

\subsection{Hypergeometric Representations}

The monic type II polynomials are, cf. \cite{AskeyII}:
\begin{equation}
 \label{M1:II}
 \begin{aligned}
 M_{1:\vec{n}}^{({\rm II})}(x; \beta, \vec{c})
 &= (\beta)_{|\vec{n}|} \prod_{i=1}^p \left( \dfrac{c_i}{c_i - 1} \right)^{n_i} \KF{1:1; \cdots; 1}{1:0; \cdots; 0}{-x: -n_1; \cdots; -n_p}{\beta: --; \cdots; --}{\dfrac{c_1 - 1}{c_1}, \ldots, \dfrac{c_p - 1}{c_p}}.
 \end{aligned}
\end{equation}

Below, we will derive a hypergeometric representation for the type I Meixner polynomials of the first kind. For $p=2$, an alternative representation was found in \cite[Proposition 4.1]{HahnI}.

\begin{teo}
The type I Meixner polynomials of the first kind are given by:
\begin{multline}
\label{M1:I}
 M_{1:\vec{n}}^{(i)} (x; \beta, \vec{c}) \\
 = \dfrac{(-1)^{n_i-1} (1-c_i)^{|\vec{n}| - 1 + \beta}}{(n_i - 1)! (\beta)_{|\vec{n}| - n_i}c_i^{n_i-1}} \prod_{q=1, q \neq i}^p \left( \dfrac{1-c_q}{c_i-c_q} \right)^{n_q} \KF{1:1; \cdots; 1}{1:0; \cdots; 0}{-n_i + 1: x+\beta;\,\vec{n}^{*i}}{\beta + |\vec{n}| - n_i: --; \cdots; --}{1-c_i, 
\left\{ \dfrac{(1 - c_i)c_q}{c_q- c_i} \right\}_{q=1,q \neq i}^p},
\end{multline}
for \( i \in\{1, \ldots, p \}\).
\end{teo}

\begin{proof}
Using \eqref{HahnTypeI} as indicated in Corollary \ref{Cor:H->M1}, we find:
\begin{multline}
M_{1:\vec{n}}^{(i)}(x;\beta,\vec{c})=\lim_{N\to\infty}\dfrac{1}{(n_i-1)!(\beta)_{|\vec{n}|-1}\sqrt{2\pi N} c_i^{N} (1-c_i)^{\frac{1-c_i}{c_i}N+\frac{1}{2}}} \\
\times\dfrac{\Gamma(N+2-|\vec{n}|)\Gamma\left(\frac{1-c_i}{c_i}N+\beta+|\vec{n}|-1\right)}{\Gamma\left(\frac{N}{c_i}+\beta+1\right)}\dfrac{\prod_{q=1}^p\left(\frac{1-c_q}{c_q}N+\beta+|\vec{n}|-1\right)_{n_q}}{\prod_{q=1,q\neq i}^p\left(\frac{c_i-c_q}{c_ic_q}N\right)_{{n}_q}}\\
 \times\sum_{l=0}^{n_i-1}\dfrac{(-n_i+1)_l}{l!}
 \underbrace{\dfrac{\left(\frac{1-c_i}{c_i}N+\beta+|\vec{n}|-1\right)_l}{\left(\frac{1-c_i}{c_i}N+1\right)_l}}
 _{=\sum_{k_i=0}^l\frac{(-l)_{k_i}}{k_i!}\frac{\left(-\beta-\sz{n}+2\right)_{k_i}}{\left(\frac{1-c_i}{c_i}N+1\right)_{k_i}}}
 \prod_{q=1,q\neq i}^p
 \underbrace{\dfrac{\left(\frac{c_q-c_i}{c_qc_i}N-{n}_q+1\right)_l}{\left(\frac{c_q-c_i}{c_qc_i}N+1\right)_l}}
 _{=\sum_{k_q=0}^l\frac{(-l)_{k_q}}{k_q!}\frac{(n_q)_{k_q}}{\left(\frac{c_q-c_i}{c_qc_i}N+1\right)_{k_q}}}
 \dfrac{\left(\frac{N}{c_i}-x+1\right)_l}{\left(\frac{N}{c_i}+\beta+1\right)_l} .
\end{multline}

Substituting in the previous fractions the Chu--Vandermonde formula \eqref{eq:Gauss_hypergeometric} and simplifying, we obtain:
\begin{multline}
M_{1:\vec{n}}^{(i)}(x;\beta,\vec{c})=\lim_{N\to\infty}\dfrac{1}{(n_i-1)!(\beta)_{|\vec{n}|-1}\sqrt{2\pi N} c_i^{N} (1-c_i)^{\frac{1-c_i}{c_i}N+\frac{1}{2}}}\\
\times\dfrac{\Gamma(N+2-|\vec{n}|)\Gamma\left(\frac{1-c_i}{c_i}N+\beta+|\vec{n}|-1\right)}{\Gamma\left(\frac{N}{c_i}+\beta+1\right)}\dfrac{\prod_{q=1}^p\left(\frac{1-c_q}{c_q}N+\beta+|\vec{n}|-1\right)_{n_q}}{\prod_{q=1,q\neq i}^p\left(\frac{c_i-c_q}{c_ic_q}N\right)_{{n}_q}}\\
 \times\Bigg(
 {\sum_{k_1=0}^{n_i-1}\cdots\sum_{k_p=0}^{n_i-1-k_1-\cdots-k_{p-1}}\frac{1}{k_1!\cdots k_p!}\frac{\left(-\beta-\sz{n}+2\right)_{k_i}}{\left(\frac{1-c_i}{c_i}N+1\right)_{k_i}}}
 \\
 \times \prod_{q=1,q\neq i}^p
{\frac{(n_q)_{k_q}}{\left(\frac{c_q-c_i}{c_qc_i}N+1\right)_{k_q}}}
 \sum_{l=0}^{n_i-1}\dfrac{(-n_i+1)_l}{l!}(-l)_{k_1+\cdots+k_p}\dfrac{\left(\frac{N}{c_i}-x+1\right)_l}{\left(\frac{N}{c_i}+\beta+1\right)_l} +\operatorname{O}\left(\frac{1}{N^{n_i}}\right) \Bigg) .
\end{multline}

The Pochhammer symbol \( (-l)_{k_1} \cdots (-l)_{k_p} \) can be expressed as \( (-l)_{k_1 + \cdots + k_p} \) along with additional lower-order terms, which are incorporated in the error term \( \operatorname{O}\left(\frac{1}{N^{n_i}}\right) \). We will now simplify the \( l \)-dependent sum further by applying Chu--Vandermonde formula to streamline the expression:
\begin{multline*}
 \sum_{l=0}^{n_i-1} \dfrac{\left(-n_i+1\right)_l}{l!}
 {{(-l)_{k_1+\cdots+k_p}}}
 \dfrac{\left(\frac{N}{c_i}-x+1\right)_l}{\left(\frac{N}{c_i}+\beta+1\right)_l}
 =\left(-1\right)^{k_1+\cdots+k_p}\left(-n_i+1\right)_{k_1+\cdots+k_p}\dfrac{\left(\frac{N}{c_i}-x+1\right)_{k_1+\cdots+k_p}}{\left(\frac{N}{c_i}+\beta+1\right)_{k_1+\cdots+k_p}}\\
 \quad\quad\times\sum_{l=0}^{n_i-1-k_1-\cdots-k_p}\dfrac{\left(-n_i+1+k_1+\cdots+k_p\right)_{l}}{l!}
\dfrac{\left(\frac{N}{c_i}-x+1+k_1+\cdots+k_p\right)_{l}}{\left(\frac{N}{c_i}+\beta+1+k_1+\cdots+k_p\right)_{l}}\\
 =(-1)^{k_1+\cdots+k_p}\left(-n_i+1\right)_{k_1+\cdots+k_p}\dfrac{\left(\frac{N}{c_i}-x+1\right)_{k_1+\cdots+k_p}}{\left(\frac{N}{c_i}+\beta+1\right)_{n_i-1}}
{(x+\beta)_{n_i-1-k_1-\cdots-k_p}}.
\end{multline*}
Consequently,
\begin{multline}
M_{1:\vec{n}}^{(i)}(x;\beta,\vec{c})\\=\lim_{N\to\infty}\dfrac{1}{(n_i-1)!(\beta)_{|\vec{n}|-1}\sqrt{2\pi N} c_i^{N} (1-c_i)^{\frac{1-c_i}{c_i}N+\frac{1}{2}}}\dfrac{\Gamma(N+2-|\vec{n}|)\Gamma\left(\frac{1-c_i}{c_i}N+\beta+|\vec{n}|-1\right)}{\Gamma\left(\frac{N}{c_i}+\beta+1\right)}\dfrac{\prod_{q=1}^p\left(\frac{1-c_q}{c_q}N+\beta+|\vec{n}|-1\right)_{n_q}}{\prod_{q=1,q\neq i}^p\left(\frac{c_i-c_q}{c_ic_q}N\right)_{{n}_q}}\\
 \times\Bigg(
\dfrac{1}{\left(\frac{N}{c_i}+\beta+1\right)_{n_i-1}} {\sum_{k_1=0}^{n_i-1}\cdots\sum_{k_p=0}^{n_i-1-k_1-\cdots-k_{p-1}}\frac{(-1)^{k_1+\cdots+k_p}\left(-n_i+1\right)_{k_1+\cdots+k_p}}{k_1!\cdots k_p!}\frac{\left(-\beta-\sz{n}+2\right)_{k_i}}{\left(\frac{1-c_i}{c_i}N+1\right)_{k_i}}}
\\\times \prod_{q=1,q\neq i}^p
{\frac{(n_q)_{k_q}}{\left(\frac{c_q-c_i}{c_qc_i}N+1\right)_{k_q}}}
 \left(\frac{N}{c_i}-x+1\right)_{k_1+\cdots+k_p}
{(x+\beta)_{n_i-1-k_1-\cdots-k_p}}+\operatorname{O}\left(\frac{1}{N^{n_i}}\right) \Bigg).
\end{multline}
With Stirling's formula \eqref{stirling}, we find that the constant part behaves for $N\to\infty$ as
\begin{equation*}
 \dfrac{(1-c_i)^{\beta+\sz{n}-2}}{(n_i-1)!(\beta)_{|\vec{n}|-1} c_i^{2(n_i-1)}\prod_{q=1,q\neq i}^p\left({c_i-c_q}\right)^{{n}_q}}
 \prod_{q=1}^p\left({1-c_q}\right)^{n_q}\,N^{n_i-1}.
\end{equation*}
Thus, it is straightforward to apply the limit to this expression and obtain
\begin{multline}
M_{1:\vec{n}}^{(i)}(x;\beta,\vec{c})= \dfrac{(1-c_i)^{\beta+\sz{n}-2}}{(n_i-1)!(\beta)_{|\vec{n}|-1} c_i^{n_i-1}\prod_{q=1,q\neq i}^p\left({c_i-c_q}\right)^{{n}_q}}
 \prod_{q=1}^p\left({1-c_q}\right)^{n_q}\\
{\sum_{k_1=0}^{n_i-1}\cdots\sum_{k_p=0}^{n_i-1-k_1-\cdots-k_{p-1}}\frac{\left(-n_i+1\right)_{k_1+\cdots+k_p}}{k_1!\cdots k_p!}\frac{\left(-\beta-\sz{n}+2\right)_{k_i}}{(c_i-1)^{k_i}}}
 \prod_{q=1,q\neq i}^p
\frac{c_q^{k_q}(n_q)_{k_q}}{\left(c_i-c_q\right)^{k_q}}
{(x+\beta)_{n_i-1-k_1-\cdots-k_p}}.
\end{multline}
A more convenient form of this limit emerges by performing index transformations: \( k_1 \rightarrow n_i - 1 - l_1 \) and \( k_i \rightarrow l_{i-1} - l_i \) for \( i \in \{2, \ldots, p\} \). Consequently, the sums are converted to the desired form given 
in
\eqref{M1:I}.
\end{proof}

\subsection{Integral Representations}

The limits in Corollary \ref{Cor:H->M1}, which lead to the Meixner polynomials of the first kind, enable us to derive integral representations based on those provided for the multiple Hahn polynomials in Theorems \ref{HI_IR} and \ref{HII_IR}.

\begin{teo}
 \label{M1I_IR}
Let \(\Sigma^\ast\) be a clockwise contour in $\{t\in \mathbb{C} \mid 0<\operatorname{Re}(t)<1\}$ enclosing \(\{c_j\}_{j=1}^p\) exactly once. Then the type I linear forms
are given by
\[ 
M_{1:\vec{n}}^{({\rm I})}(x;\beta,\vec{c}) = \prod_{j=1}^p (1-c_j)^{n_j} \frac{\Gamma(x+\beta)}{\Gamma(x+1)\Gamma(\beta+\sz{n}-1)} \int_{\Sigma^\ast} \frac{(1-t)^{\sz{n}+\beta-2}\, t^{x}}{\prod_{j=1}^p (t-c_j)^{n_j}} \frac{\d t }{2 \pi \ii }.
\]
In particular, the type I polynomials are given by
\[ M_{1:\vec{n}}^{(i)}(x;\beta,\vec{c}) = \frac{\prod_{j=1}^p (1-c_j)^{n_j}}{(\beta)_{\sz{n}-1}} \frac{1}{c_i^x} \int_{\Sigma^\ast_i} \frac{(1-t)^{\sz{n}+\beta-2} t^{x}}{\prod_{j=1}^p (t-c_j)^{n_j}} \frac{\d t }{2 \pi \ii } ,
 \]
where \(\Sigma_i^\ast\) is a clockwise contour in $\{t\in \mathbb{C} \mid 0<\operatorname{Re}(t)<1\}$ enclosing \(c_i\) exactly once without enclosing any of the other points in \(\{c_j\}_{j=1}^p\).
\end{teo}

\begin{proof}
 Consider the integral representation for the linear form 
 \( Q_{\vec{n}}^{({\rm I})}\left(N-x;\left\{\frac{1-c_i}{c_i}N\right\}_{i=1}^p, \beta-1, N\right) \)
 given in Theorem \ref{HI_IR}. It involves a clockwise contour enclosing \(\bigcup_{j=1}^p\left[\frac{1-c_j}{c_j}N, n_j + \frac{1-c_j}{c_j}N - 1\right]\) exactly once, with \(\operatorname{Re}(t) > -1\). By changing variables \(t\mapsto Nt\), we can deform this contour to a contour \(\Sigma\) that encloses \(\bigcup_{j=1}^p \left[\frac{1}{c_j} - 1, \frac{1}{c_j} + 1\right] \) exactly once, with \(\operatorname{Re}(t) > 0\). We can use \eqref{GFB} to obtain the asymptotics for the \( N\)-dependent parts of the integrand and prefactor:
\begin{multline*}
 N \frac{\Gamma(N-\sz{n}+2)}{\Gamma(N-x+1)} \frac{\Gamma(Nt+\beta+\sz{n}-1)}{\Gamma(Nt+1)} \frac{\Gamma(N-x+Nt+1)}{\Gamma(Nt+\beta+N+1)} \frac{\prod_{j=1}^p \left(\frac{1-c_j}{c_j}N+\beta+\sz{n}-1\right)_{n_j}}{\prod_{j=1}^p \left(\frac{1-c_j}{c_j}N-Nt\right)_{n_j}} \\
 \begin{aligned}
 &\sim N^{-\sz{n}+x+2} (Nt)^{\beta+\sz{n}-2} ((1+t)N)^{-x-\beta} \frac{\prod_{j=1}^p \left(\frac{1-c_j}{c_j}\right)^{n_j}}{\prod_{j=1}^p \left(\frac{1-c_j}{c_j}-t\right)^{n_j}}= \frac{t^{\sz{n}+\beta-2}}{(1+t)^{x+\beta}} \frac{\prod_{j=1}^p \left(\frac{1-c_j}{c_j}\right)^{n_j}}{\prod_{j=1}^p \left(\frac{1-c_j}{c_j}-t\right)^{n_j}},\quad N\to\infty.
\end{aligned}\end{multline*}
Combining these results with limit \eqref{H->M1:MOP} produces
\[ 
M_{1:\vec{n}}^{({\rm I})}(x;\beta,\vec{c}) = \frac{\Gamma(x+\beta)}{\Gamma(x+1)\Gamma(\beta+\sz{n}-1)} \prod_{j=1}^p \left(\frac{1-c_j}{c_j}\right)^{n_j}\bigintsss_{\Sigma} \dfrac{t^{\sz{n}+\beta-2}}{(1+t)^{\beta+x}} \frac{1}{\prod_{j=1}^p \left(\frac{1-c_j}{c_j}-t\right)^{n_j}} \frac{\d t }{2 \pi \ii }. 
\]
It then remains to apply the change of variables \( t \mapsto \frac{1}{t} - 1 \). This transformation results in a contour that can be deformed to \(\Sigma^\ast\), and after some computations the stated result about the type I linear forms follows. Once this result is established, one can note that only certain poles contribute to a specific type I polynomial, as the remaining poles correspond to different weight functions.
\end{proof}

An application of the residue theorem leads to the formulas for the type I polynomials below.
\begin{coro}
The following Rodrigues-type formula for the type I polynomials holds:
\begin{equation}\label{eq:Rodrigues first kind Meixner}
M_{1:\vec{n}}^{(i)}(x;\beta,\vec{c}) = - \frac{\prod_{j=1}^p (1-c_j)^{n_j}}{(n_i-1)!(\beta)_{\sz{n}-1}} \frac{1}{c_i^x} \frac{\d^{n_i-1}\quad}{\d c_i^{n_i-1}} \left(\frac{(1-c_i)^{\sz{n}+\beta-2} c_i^{x}}{\prod_{j=1,j\neq i}^p (c_i-c_j)^{n_j}}\right).
\end{equation}
\end{coro}
\begin{coro}
The following alternative hypergeometric representation of \eqref{M1:I} for type I Meixner of the first kind polynomials holds:
 \begin{multline}\label{M1:I_alt}
 M_{1:\vec{n}}^{(i)}(x; \beta, \vec{c}) \\
 = \dfrac{(-1)^{n_i - 1} (1 - c_i)^{|\vec{n}| - 1 + \beta}}{(n_i - 1)! (\beta)_{|\vec{n}| - n_i}} \prod_{q=1, q \neq i}^p \left( \dfrac{1 - c_q}{c_i - c_q} \right)^{n_q} \KF{1:1; \cdots; 1}{1:0; \cdots; 0}{-n_i + 1: -x;\,\vec{n}^{*i}}{\beta + |\vec{n}| - n_i: --; \cdots; --}{\dfrac{c_i - 1}{c_i}, \left\{ \dfrac{c_i - 1}{c_i - c_q} \right\}_{q=1,q \neq i}^p}.
 \end{multline}
\end{coro}
\begin{proof}
Using Leibniz rule, we can compute the $(n_i-1)$-th derivative in \eqref{eq:Rodrigues first kind Meixner} as
\begin{align*}
 &\sum_{k=0}^{n_i-1} \binom{n_i-1}{k} \frac{\d^{n_i-1-k}\quad}{\d c_i^{n_i-1-k}} \left(\left(1-c_i\right)^{|\hat{n}|+\beta-2}\right)
 \frac{\d^{k}\quad}{\d c_i^k}\left(c_i^x \prod_{\substack{j=1 \\
 j \neq i}}^n\left(c_i-c_j\right)^{-n_j}\right) \\
 & =\begin{multlined}[t][.9\textwidth]
 (n_i-1)! \sum_{k=0}^{n_i-1} \sum_{|\vec l|=k} \frac{(|\vec n|-n_i+\beta+k)_{n_i-1-k}}{\left(n_i-1-k\right)!} \frac{\left(x-l_i+1\right)_{l_i}}{l_{i}!} \prod_{\substack{j=1 \\ j \neq i}}^p \frac{\left(-n_j-l_j+1\right)_{l_j}}{l_j!} \\
 \times(-1)^{n_i-1-k}\left(1-c_i\right)^{|\vec n|-n_i+\beta+k-1} c_i^{x- l_i} \prod_{\substack{j=1 \\ j \neq i}}^p\left(c_i-c_j\right)^{-n_j-l_j} 
 \end{multlined}\\
 & =\begin{multlined}[t][.9\textwidth] (-1)^{n_i-1} (n_i-1)! \frac{\Gamma(|\bar{n}|+\beta-1)}{\Gamma\left(|\bar{n}|-n_i+\beta\right)}\left(1-c_i\right)^{\left|\vec n\right|-n_i+\beta-1} c_i^x \prod_{\substack{j=1 \\ j \neq i}}^p\left(c_i-c_j\right)^{-n_ j}\\ \times \sum_{k=0}^{n_i-1} \sum_{|\vec{l}|=k} \frac{\left(-n_i+1\right)_k}{(|\vec{n}|-n_i+\beta)_k}\frac{(-x)_{ l_i}}{l_i!}\left(\frac{c_i-1}{c_i}\right)^{l_i}\prod_{\substack{j=1 \\ j \neq i}}^p\frac{(n_j)_{l_ j }}{l_j!} {\left(\frac{c_i-1}{c_i-c_j}\right)^{l_j}} .
 \end{multlined} 
 \end{align*} 
Since $(-n_i+1)_k$ vanishes for $k\in\Z_{\geq n_i}$, we may replace $\sum_{k=0}^{n_i-1}$ by $\sum_{k\geq 0}$. Afterwards, we can replace the double sum $\sum_{k\geq 0} \sum_{|\vec{l}|=k}$ by the multiple sum $\sum_{l_1\geqslant 0}\cdots\sum_{l_p\geqslant 0}$, which then leads to the stated Kampé de Fériet series. 
\end{proof}

\begin{teo}
 \label{M1II_IR}
Let \(\mathcal{C}^\ast\) be a clockwise contour in $\{s\in \mathbb{C} \mid \operatorname{Re}(s) < 1\}$ enclosing the origin exactly once. Then, for \(x \in \mathbb{N}\), the type II polynomials are given by
\[
 M_{1:\vec{n}}^{({\rm II})}(x;\beta,\vec{c}) = \frac{ \Gamma(\beta+\sz{n})}{\prod_{j=1}^p (1-c_j)^{n_j}} \frac{\Gamma(x+1)}{\Gamma(x+\beta)} \bigintsss_{\mathcal{C}^\ast} \frac{\prod_{j=1}^p (s-c_j)^{n_j}}{(1-s)^{\sz{n}+\beta} s^{x+1}} \frac{\d s }{2 \pi \ii }.
 \]
\end{teo}
\begin{proof}
Consider the integral representation for 
\( Q_{\vec{n}}^{({\rm II})}\left(N-x; \left\{\frac{1-c_i}{c_i}N\right\}_{i=1}^p, \beta-1, N\right) \)
given in Theorem \ref{HII_IR}. This representation involves a counterclockwise contour that encloses \([-N, 0]\) exactly once. 
After performing the change of variables \(s \mapsto Ns\), we can deform this contour to one that encloses \([-1, 0]\) exactly once, intersecting the real axis at \(c < -1\) and \(d > 0\). The integrand becomes
 \[ 
 \frac{N}{2\pi \ii} \frac{\Gamma(Ns) \Gamma(Ns+N+\beta)}{\Gamma(Ns+\sz{n}+\beta)} \frac{\prod_{j=1}^p \left(\frac{1-c_j}{c_j}N+1-Ns\right)_{n_j}}{\Gamma(N-x+Ns+1)}.
 \]
 At this stage, Stirling's formula cannot be applied because the contour intersects the branch cut of the function \(x \mapsto x^\beta\). To resolve this, we apply the change of variables \(s \mapsto \frac{1}{s} - 1\), which transforms the integrand to
 \begin{align*}
 - \frac{N s^{-2}}{2\pi \ii} &\frac{\Gamma\left(N\left(\frac{1}{s}-1\right)\right) \Gamma\left(\frac{N}{s}+\beta\right)}{\Gamma\left(N\left(\frac{1}{s}-1\right)+\sz{n}+\beta\right)} \frac{\prod_{j=1}^p \left(\frac{1}{c_j}N+1-\frac{N}{s}\right)_{n_j}}{\Gamma\left(-x+\frac{N}{s}+1\right)}. \end{align*}
The contour can then be deformed to a rectangle intersecting the real axis at \(\frac{1}{1+c}\in
(-\infty, 0) \)
and \(\frac{1}{1+d}\in(0,1)\). We can then split the rectangle into two parts: a vertical line \(\mathcal{R}^0\) (of finite length independent of \(c\)) that intersects the real axis at \(\frac{1}{1+c}\), and the remaining part, \(\mathcal{R}^\ast\) (see Figure \ref{fig:2}). As \(c \to -1\), the integral over \(\mathcal{R}^0\) converges to 0 because the integrand is a rational function of order \(\operatorname{o}(|s|^{-1})\) as \(|s| \to \infty\). The contour \(\mathcal{R}^\ast\) then becomes a contour enclosing \(\left(-\infty, \frac{1}{1+d}\right)\) exactly once, intersecting the real axis only at \(\frac{1}{1+d}\). 
\begin{figure}[htbp]
	\begin{tikzpicture}[arrowmark/.style 2 args={decoration={markings,mark=at position #1 with \arrow{#2}}},scale=1]
		\begin{axis}[axis lines=middle,axis equal,grid=both,xmin=-8, xmax=4,ymin=-2, ymax=2,
			xticklabel,yticklabel,disabledatascaling,xlabel=$x$,ylabel=$y$,every axis x label/.style={
				at={(ticklabel* cs:1)},
				anchor=south west,
			},
			every axis y label/.style={
				at={(ticklabel* cs:1.0)},
				anchor=south west,
			},grid style={line width=.1pt, draw=Bittersweet!10},
			major grid style={line width=.2pt,draw=Bittersweet!50},
			minor tick num=4,
			enlargelimits={abs=2},
			axis line style={latex'-latex'},Bittersweet] 
			\node[anchor = north east,Bittersweet] at (axis cs: 5,5) {$\mathbb C$} ;
			\draw[ DarkSlateBlue,ultra thick, decoration={markings, mark=at position 0.15 with {\arrow[ultra thick]
{Stealth}
			}},
			postaction={decorate},decoration={markings, mark=at position 0.47 with {\arrow[ultra thick]
			{Stealth}
			}},
			postaction={decorate},decoration={markings, mark=at position 0.65 with {\arrow[ultra thick]
			{Stealth}
			}},
			postaction={decorate},decoration={markings, mark=at position 0.925 with {\arrow[gray,ultra thick]
			{Stealth}}},
			postaction={decorate}] (-7,-2) -- (3.5,-2)-- (3.5,2)-- (-7,2)--cycle;
			\draw [fill,black] (axis cs:{0,0}) circle [radius=1.5pt] node[Black,above right]{$0$};
			\draw [fill,black] (axis cs:{5,0}) circle [radius=1.5pt] node[Black,above ]{$1$};
			\draw [fill,black] (axis cs:{3.5,0}) circle [radius=1.5pt] node[Black,below right ]{$\frac{1}{1+d}$};
			\draw [fill,black] (axis cs:{-7,0}) circle [radius=1.5pt] node[Black,below right]{$-\frac{1}{1+c}$};
			\draw [fill,black] (axis cs:{-5.5,-3.5}) node[Black,above left]{$c\to -1$};
			\draw[ultra thick, gray] (-7,-2)--(-7,2);
			\draw [DarkSlateBlue,ultra thick,dashed] (-7,-2)--(-10.5,-2) ;
			\draw [DarkSlateBlue,ultra thick,dashed] (-7,2)--(-10.5,2) ;

			
			\draw[ultra thick, gray] (-8,6.1) -- (-7,6.1) node at (-6,6.2) {$\mathcal R^0$};
			\draw[ultra thick, DarkSlateBlue] (-8,5.1) -- (-7,5.1) node at (-6,5.2) {$\mathcal R^*$};
			
			\draw[DarkSlateBlue, thick,->] (-7,-1)--(-8.5,-1) ;
			\draw[DarkSlateBlue, thick,->] (-7,1)--(-8.5,1) ;
			
			
		\end{axis}
		\draw (3.5,-0.4) ; 
	\end{tikzpicture}
	\caption{Integration contours for type II Meixner of the first kind}
	\label{fig:2}
\end{figure}
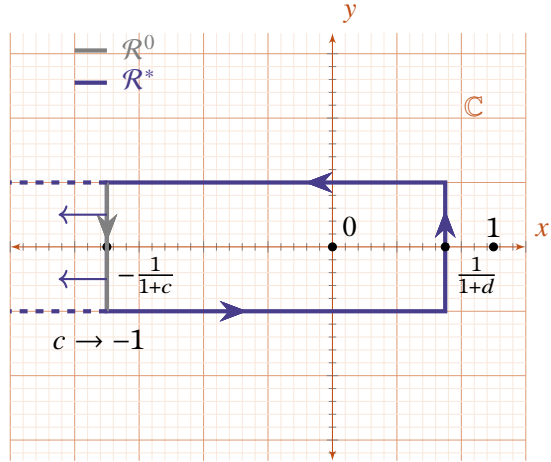
On this contour, we can use use \eqref{GFB} to obtain the asymptotics for the $N$-dependent parts of the integrand and prefactor:
\begin{multline*}
	N \frac{\Gamma(N-x+1)}{\Gamma(N-\sz{n}+1)} \frac{\Gamma\left(N\left(\frac{1}{s}-1\right)\right)}{\Gamma\left(N\left(\frac{1}{s}-1\right)+\sz{n}+\beta\right)} \frac{\Gamma\left(\frac{N}{s}+\beta\right)}{\Gamma\left(-x+\frac{N}{s}+1\right)} \frac{\prod_{j=1}^p \left(\frac{1}{c_j}N+1-\frac{N}{s}\right)_{n_j}}{\prod_{j=1}^p \left(\frac{1-c_j}{c_j}N+\beta+\sz{n}\right)_{n_j}} \\
	\begin{aligned}&\sim N^{-x+\sz{n}+1} \left(N\left(\frac{1}{s}-1\right)\right)^{-\sz{n}-\beta} \left(\frac{N}{s}\right)^{\beta+x-1} \frac{\prod_{j=1}^p \left(\frac{1}{c_j}-\frac{1}{s}\right)^{n_j}}{\prod_{j=1}^p \left(\frac{1-c_j}{c_j}\right)^{n_j}} = \frac{s^{\sz{n}{-x}+1}}{(1-s)^{\sz{n}+\beta}} \frac{\prod_{j=1}^p \left(\frac{1}{c_j}-\frac{1}{s}\right)^{n_j}}{\prod_{j=1}^p \left(\frac{1-c_j}{c_j}\right)^{n_j}},\quad N\to \infty.
	\end{aligned}
\end{multline*}
Combining these results with limit \eqref{H->M1:MOP} then yields the stated formula, but with the contour \(\mathcal{R}^\ast\) instead of $\mathcal{C}^\ast$ (see Figure \ref{fig:3}). However, since $\frac{1}{1+d}\in(0,1)$ and the integrand is $\operatorname{o}(|s|^{-1})$ as $|s|\to\infty$, we can close the contour and deform it to $\mathcal{C}^\ast$ (see Figure \ref{fig:3}).
\end{proof}

\begin{figure}[htbp]
	\begin{tikzpicture}[arrowmark/.style 2 args={decoration={markings,mark=at position #1 with \arrow{#2}}},scale=1.1]
		\begin{axis}[axis lines=middle,axis equal,grid=both,xmin=-8, xmax=4,ymin=-2, ymax=2,
			xticklabel,yticklabel,disabledatascaling,xlabel=$x$,ylabel=$y$,every axis x label/.style={
				at={(ticklabel* cs:1)},
				anchor=south west,
			},
			every axis y label/.style={
				at={(ticklabel* cs:1.0)},
				anchor=south west,
			},grid style={line width=.1pt, draw=Bittersweet!10},
			major grid style={line width=.2pt,draw=Bittersweet!50},
			minor tick num=4,
			enlargelimits={abs=2},
			axis line style={latex'-latex'},Bittersweet] 
			\node[anchor = north east,Bittersweet] at (axis cs: 5,5) {$\mathbb C$} ;
			\draw[ gray,ultra thick, decoration={markings, mark=at position 0.15 with {\arrow[ultra thick]
			{Stealth}}},
			postaction={decorate},decoration={markings, mark=at position 0.47 with {\arrow[ultra thick]
			{Stealth}}},
			postaction={decorate},decoration={markings, mark=at position 0.65 with {\arrow[ultra thick]
			{Stealth}}},
			postaction={decorate},decoration={markings, mark=at position 0.963 with {\arrow[gray,ultra thick]
			{Stealth}}},
			postaction={decorate}] (-7,-2) -- (3.5,-2)-- (3.5,2)-- (-7,2)--cycle;
			\draw [fill,black] (axis cs:{0,0}) circle [radius=1.5pt] node[Black,above right]{$0$};
			\draw [fill,black] (axis cs:{5,0}) circle [radius=1.5pt] node[Black,above ]{$1$};
			\draw [fill,black] (axis cs:{3.5,0}) circle [radius=1.5pt] node[Black,below right ]{$\frac{1}{1+d}$};
			\draw[ultra thick, gray] (-7,-2)--(-7,2);
			\draw [DarkSlateBlue,ultra thick,dotted,decoration={markings, mark=at position 0.15 with {\arrow[ultra thick]
			{Stealth}
			}},
			postaction={decorate},decoration={markings, mark=at position 0.47 with {\arrow[ultra thick]
			{Stealth}}},
			postaction={decorate},decoration={markings, mark=at position 0.65 with {\arrow[ultra thick]
			{Stealth}}},
			postaction={decorate}] (-10.5,-2.1)--(3.6,-2.1)--(3.6,2.1)--(-10.5,2.1) ;

			
			\draw[ultra thick, gray] (-8,6.1) -- (-7,6.1) node at (-6,6.2) {$\mathcal C^*$};
			\draw[ultra thick, DarkSlateBlue,dotted] (-8,5.1) -- (-7,5.1) node at (-6,5.2) {$\mathcal R^*$};
			
			\draw[DarkSlateBlue, thick,->] (-9,-1)--(-8.,-1) ;
			\draw[DarkSlateBlue, thick,->] (-9,1)--(-8.,1) ;
			
			
		\end{axis}
		\draw (3.5,-0.4) ; 
	\end{tikzpicture}
	\caption{Integration contours for type II Meixner of the first kind}
	\label{fig:3}
\end{figure}
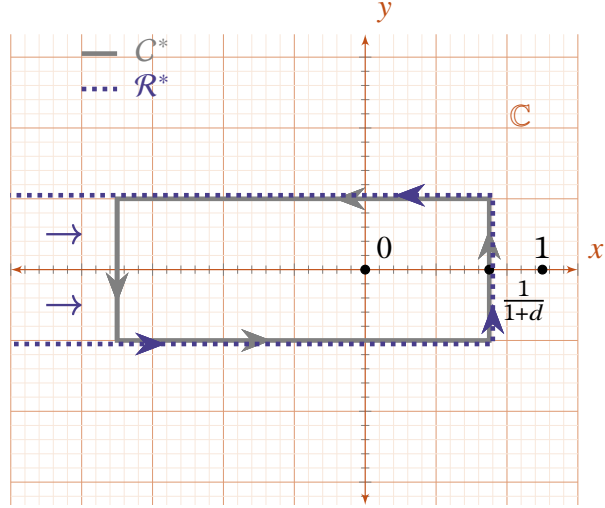

\subsection{Recurrence Coefficients}

We now provide the explicit expressions for the recurrence coefficients.

\begin{teo}
 The Meixner polynomials of the first kind, both type I \eqref{M1:I}, and type II \eqref{M1:II}, each follow their respective recurrence relations \eqref{NNRR} with coefficients
 \begin{equation}
 \label{M1:R}
 \begin{aligned}
 b_{\vec{n}}^{(M1),0}(k)=&\dfrac{(\beta+|\vec{n}|)c_k}{1-c_k}+\sum_{i=1}^p\dfrac{n_i}{1-c_i},
 \\
 b_{\vec{n}}^{(M1),j}=&(\beta+|\vec{n}|-j)_{j}\sum_{i\in S(\pi,j)}
 \dfrac{n_ic_i}{(1-c_i)^{j+1}}
 \prod_{q\in S^{\textsf c}(\pi,j)}\dfrac{c_i-c_q}{1-c_q},& j\in&\{1,\ldots,p\}.
 \end{aligned}
 \end{equation}
\end{teo}

\begin{proof} 
For \( j \in \{1, \ldots, p \}\), applying limit \eqref{H->M1:R} is relatively straightforward. However, for \( j = 0 \), the situation is more complex and requires additional analysis. In this instance, the expression simplifies to:
\begin{multline*}
b_{\vec{n}}^{(M1),0}(k)=\lim_{N\to\infty}\Bigg(
\frac{N}{c_k}+n_k+1-(n_k+1)\left(\frac{N}{c_k}+\beta+n_k+1\right)
\dfrac{\frac{1-c_k}{c_k}N+n_k+1}{\frac{1-c_k}{c_k}N+\beta+n_k+|\vec{n}|+1}\prod_{q=1,q\neq k}^p\dfrac{\frac{c_q-c_k}{c_qc_k}N+n_k+1}{\frac{c_q-c_k}{c_qc_k}N+n_k+1-n_q}
\\
+
n_k\left(\frac{N}{c_k}+\beta+n_k\right)\dfrac{\frac{1-c_k}{c_k}N+n_k}{\frac{1-c_k}{c_k}N+\beta+n_k+|\vec{n}|-1}
\prod_{q=1,q\neq k}^p\dfrac{\frac{c_q-c_k}{c_qc_k}N+n_k}{\frac{c_q-c_k}{c_qc_k}N-n_q+n_k}\Bigg)\\
-\underbrace{\lim_{N\to\infty}
\left(\frac{1-c_k}{c_k}N+\beta+n_k+|\vec{n}|\right)\sum_{i=1,i\neq k}^p 
\dfrac{n_i\left(\frac{1-c_i}{c_i}N+n_i\right)\left(\frac{N}{c_i}+\beta+n_i\right)}{\left(\frac{c_k-c_i}{c_ic_k}N-n_k-1+n_i\right)\left(\frac{1-c_i}{c_i}N+\beta+n_i+|\vec{n}|-1\right)_{2}}
\prod_{q=1,q\neq i}^p\dfrac{{\frac{c_q-c_i}{c_qc_i}N+n_i}}{\frac{c_q-c_i}{c_qc_i}N-n_q+n_i}}_{=
 (1-c_k)\sum_{i=1,i\neq k}^p 
 \frac{n_ic_i}{(c_k-c_i)(1-c_i)}}.
\end{multline*}

The second term's limit can be applied directly, as shown. We will now focus on the first term. To compute its limit, we need to reformulate the expression inside as follows:
\begin{align*} 
&\begin{multlined}[t][\textwidth]
 \frac{N}{c_k}+n_k+1-(n_k+1)\left(\frac{N}{c_k}+\beta+n_k+1\right)
\left(1-\dfrac{\beta+\sz{n}}{\frac{1-c_k}{c_k}N+\beta+n_k+|\vec{n}|+1}\right)\prod_{q=1,q\neq k}^p\left(1+\dfrac{n_q}{\frac{c_q-c_k}{c_qc_k}N+n_k+1-n_q}\right)
\\
+
n_k\left(\frac{N}{c_k}+\beta+n_k\right)\left(1-\dfrac{\beta+\sz{n}-1}{\frac{1-c_k}{c_k}N+\beta+n_k+|\vec{n}|-1}\right)
\prod_{q=1,q\neq k}^p\left(1+\dfrac{n_q}{\frac{c_q-c_k}{c_qc_k}N+n_k-n_q}\right)
\end{multlined}\\
&=\begin{multlined}[t][.75\textwidth]
 \frac{N}{c_k}+n_k+1-(n_k+1)\left(\frac{N}{c_k}+\beta+n_k+1\right)
\left(1-\dfrac{\beta+\sz{n}}{\frac{1-c_k}{c_k}N+\beta+n_k+|\vec{n}|+1}+\sum_{q=1,q\neq k}^p\dfrac{n_q}{\frac{c_q-c_k}{c_qc_k}N+n_k+1-n_q}+\operatorname O\left(\dfrac{1}{N^2}\right)\right)
\\+
n_k\left(\frac{N}{c_k}+\beta+n_k\right)\left(1-\dfrac{\beta+\sz{n}-1}{\frac{1-c_k}{c_k}N+\beta+n_k+|\vec{n}|-1}+
\sum_{q=1,q\neq k}^p\dfrac{n_q}{\frac{c_q-c_k}{c_qc_k}N+n_k-n_q}+\operatorname O\left(\dfrac{1}{N^2}\right)\right)
\end{multlined}\\
&=\begin{multlined}[t][\textwidth]-n_k-\beta+(n_k+1)(\beta+\sz{n})\dfrac{\frac{N}{c_k}+\beta+n_k+1}{\frac{1-c_k}{c_k}N+\beta+n_k+|\vec{n}|+1}-(n_k+1)\sum_{q=1,q\neq k}^pn_q\dfrac{\frac{N}{c_k}+\beta+n_k+1}{\frac{c_q-c_k}{c_qc_k}N+n_k+1-n_q}
\\
-n_k(\beta+\sz{n}-1)\dfrac{\frac{N}{c_k}+\beta+n_k}{\frac{1-c_k}{c_k}N+\beta+n_k+|\vec{n}|-1}+
n_k\sum_{q=1,q\neq k}^pn_q\dfrac{\frac{N}{c_k}+\beta+n_k}{\frac{c_q-c_k}{c_qc_k}N+n_k-n_q}+\operatorname O\left(\dfrac{1}{N}\right)\end{multlined}\\
&\sim
\begin{aligned}
& \sz{n}+\dfrac{(\beta+\sz{n}+n_k)c_k}{1-c_k}-\sum_{q=1,q\neq k}^p\dfrac{n_qc_q}{c_q-c_k},& N&\to\infty.
\end{aligned}
\end{align*}
Combining this result with the second term and simplifying yields the recurrence coefficients \eqref{M1:R}.
\end{proof}

\section{Kravchuk 
}
For the multiple Kravchuk polynomials, the weight functions are defined by
\begin{equation}
\label{WeightsK}
\begin{aligned}
 w^{(K)}_i(x;\uppi_i,N)&\coloneq\dfrac{\Gamma(N+1)}{\Gamma(x+1)\Gamma(N-x+1)} \uppi_i^{x}(1-\uppi_i)^{N-x}, & i&\in\{1,\ldots,p\},& \Delta&=\{0,\ldots,N\} ,
\end{aligned}
\end{equation}
with \( 0 <\uppi_1, \ldots,\uppi_p < 1 \). In order to ensure an AT system, we require that $\sz{n}\leqslant N\in\mathbb N_0$ and \(\uppi_i \neq\uppi_j \) for \( i \neq j \). We will denote \(\vec{\uppi}\coloneq(\uppi_1,\cdots,\uppi_p)\).

\begin{pro}
The following limiting relations between the Hahn \eqref{WeightsHahn} and Kravchuk weights hold
 \[
 w_i^{(K)}(x; \uppi_i, N) = \lim_{\tau \to \infty} \frac{\Gamma(N + 1) (1 - \uppi_i)^N}{\tau^N}\, w_i\left(x; \frac{\uppi_i}{1 - \uppi_i} \tau, \tau, N\right),\quad i\in\{1,\ldots,p\}.
 \]
\end{pro}

\begin{proof}
It follows from the asymptotics for $\tau\rightarrow\infty$
 \[
 w_i\left(x; \frac{\uppi_i}{1 - \uppi_i} \tau, \tau, N\right) = \frac{1}{\Gamma(x+1)\Gamma(N-x+1)} 
 {\left( \frac{\uppi_i}{1 - \uppi_i} \tau + 1\right)_x} (\tau+1)_{N-x}\sim \frac{\tau^{N}}{\Gamma(x+1)\Gamma(N-x+1)} 
 { \frac{\uppi_i^x}{(1 - \uppi_i)^x}}. 
 \] 
\end{proof}

\begin{coro}
\label{Cor:H->K}
 For the linear forms $K_{\vec{n}}^{({\rm I})}$, type I polynomials $K_{\vec{n}}^{(i)}$, type II polynomials $K^{({\rm II})}_{\vec{n}}$ and recurrence coefficients $b_{\vec{n}}^{(K),j}$; the following respective limiting relations from Hahn hold:
\begin{subequations}
\label{H->K}
 \begin{align}
 \label{H->K:MOP}
 K_{\vec{n}}^{(\ast)}(x;\vec{\uppi},N)&=\lim_{\tau\to\infty} \kappa^{(\ast)}_\tau Q_{\vec{n}}^{(\ast)}\left(x;\left\{\frac{\uppi_i}{1-\uppi_i}\tau\right\}_{i=1}^p,\tau,N\right) ,
 \\
 \label{H->K:R}
 b_{\vec{n}}^{(K),j}(\vec{\uppi},N) &= 
 \lim_{\tau\rightarrow\infty}b_{\vec{n}}^j\left(\left\{\frac{\uppi_i}{1-\uppi_i}\tau\right\}_{i=1}^p,\tau,N\right),
 \end{align}
\end{subequations}
where 
$$\kappa^{({\rm I})}_\tau = \kappa^{({\rm II})}_\tau = 1,\quad \kappa^{(i)}_\tau = \frac{\tau^N}{\Gamma(N+1)(1-\uppi_i)^{N}}.$$
\end{coro}
\begin{rem}
 The limiting relation for the type II polynomials was already deduced in \cite{AskeyII}.
\end{rem}

\subsection{Hypergeometric Representations}

The monic type II polynomials can be derived from the Hahn polynomials \eqref{HahnTypeII} by taking the limit in Corollary \eqref{Cor:H->K}, cf. \cite{AskeyII}:
\begin{equation}
\label{K:II}
 K_{\vec{n}}^{({\rm II})}(x;\vec{\uppi},N)
 =(-N)_{|\vec{n}|}\prod_{i=1}^p\uppi_i^{n_i}\,\KF{1:1;\cdots;1}{1:0;\cdots;0}{-x:-n_1;\cdots;-n_p}{-N:--;\cdots;--}{\dfrac{1}{\uppi_1},\ldots,\dfrac{1}{\uppi_p}}.
\end{equation}

Below, we will derive a hypergeometric representation for the type I Kravchuk polynomials. For $p=2$, an alternative representation was found in \cite[Proposition 6.1]{HahnI}.

\begin{teo}
The type I Kravchuk polynomials are
\begin{multline}
\label{K:I}
K_{\vec{n}}^{(i)}(x;\vec{\uppi},N)=\dfrac{(-1)^{n_i-1}}{(n_i-1)!(-N)_{|\vec{n}|-n_i}
(1-\uppi_i)^{n_i-1}}
\prod_{q=1,q\neq i}^p\dfrac{1}{(\uppi_q-\uppi_i)^{n_q}}\\
 \times\KF{1:1;\cdots;1}{1:0;\cdots;0}{-n_i+1:-x;\;\vec{n}^{*i}}{-N+|\vec{n}|-n_i:--;\cdots;--}{\frac{1}{\uppi_i},\left\{\frac{1-\uppi_q}{\uppi_i-\uppi_q}\right\}_{q=1,q\neq i}^p},
\end{multline}
for \( i\in\{1,\ldots,p\}\).
\end{teo}

\begin{proof}

Substituting into the expression \eqref{HahnTypeI} as indicated in Corollary \ref{Cor:H->K}, we obtain:
\begin{align*}
K_{\vec{n}}^{(i)}(x;\vec{\uppi},N)=\lim_{\tau\rightarrow\infty}\dfrac{1}{N!}\dfrac{\tau^N}{(1-\uppi_i)^N}\dfrac{(-1)^{|\vec{n}|-1}(N+1-|\vec{n}|)!}{(n_i-1)!(\tau+1)_{|\vec{n}|-1}\left(\frac{1}{1-\uppi_i}\tau+|\vec{n}|\right)_{N+2-|\vec{n}|}}\dfrac{\prod_{q=1}^p\left(\frac{1}{1-\uppi_q}\tau+|\vec{n}|\right)_{n_q}}{\prod_{q=1,q\neq i}^p\left(\frac{\uppi_q-\uppi_i}{(1-\uppi_q)(1-\uppi_i)}\tau\right)_{{n}_q}}\\
 \times\sum_{l=0}^{n_i-1}\dfrac{(-n_i+1)_l}{l!}\underbrace{\dfrac{\left(\frac{1}{1-\uppi_i}\tau+|\vec{n}|\right)_l}{\left(\frac{1}{1-\uppi_i}\tau+N+2\right)_l}}_{=\sum_{k_i=0}^l\frac{(-l)_{k_i}}{k_i!}\frac{(N-|\vec{n}|+2)_{k_i}}{\left(\frac{1}{1-\uppi_i}\tau+N+2\right)_{k_i}}}
 \prod_{q=1,q\neq i}^p\underbrace{\dfrac{\left(\frac{\uppi_i-\uppi_q}{(1-\uppi_q)(1-\uppi_i)}\tau-{n}_q+1\right)_l}{\left(\frac{\uppi_i-\uppi_q}{(1-\uppi_q)(1-\uppi_i)}\tau+1\right)_l}}_{=\sum_{k_q=0}^l\frac{(-l)_{k_q}}{k_q!}\frac{(n_q)_{k_q}}{\left(\frac{\uppi_i-\uppi_q}{(1-\uppi_i)(1-\uppi_q)}\tau+1\right)_{k_q}}}
 \dfrac{\left(x+\frac{\uppi_i}{1-\uppi_i}\tau+1\right)_l}{\left(\frac{\uppi_i}{1-\uppi_i}\tau+1\right)_l}.
\end{align*}
We can now rewrite the fractions within the sum, as previously indicated, using Gauss's hypergeometric formula \eqref{eq:Gauss_hypergeometric} to simplify the expressions and obtain:
\begin{multline*}
K_{\vec{n}}^{(i)}(x;\vec{\uppi},N)=\lim_{\tau\rightarrow\infty}\dfrac{1}{N!}\dfrac{\tau^N}{(1-\uppi_i)^N}\dfrac{(-1)^{|\vec{n}|-1}(N+1-|\vec{n}|)!}{(n_i-1)!(\tau+1)_{|\vec{n}|-1}\left(\frac{1}{1-\uppi_i}\tau+|\vec{n}|\right)_{N+2-|\vec{n}|}}\dfrac{\prod_{q=1}^p\left(\frac{1}{1-\uppi_q}\tau+|\vec{n}|\right)_{n_q}}{\prod_{q=1,q\neq i}^p\left(\frac{\uppi_q-\uppi_i}{(1-\uppi_q)(1-\uppi_i)}\tau\right)_{{n}_q}}\\
 \times\Bigg(\sum_{k_1=0}^{n_i-1}\cdots\sum_{k_p=0}^{n_i-1-k_1-\cdots-k_{p-1}}{\frac{1}{k_i!}\frac{(N-|\vec{n}|+2)_{k_i}}{\left(\frac{1}{1-\uppi_i}\tau+N+2\right)_{k_i}}}
 \prod_{q=1,q\neq i}^p{\frac{1}{k_q!}\frac{(n_q)_{k_q}}{\left(\frac{\uppi_i-\uppi_q}{(1-\uppi_i)(1-\uppi_q)}\tau+1\right)_{k_q}}}\\
 \times
 \sum_{l=0}^{n_i-1}\dfrac{(-n_i+1)_l}{l!}(-l)_{k_1+\cdots+k_p}\dfrac{\left(x+\frac{\uppi_i}{1-\uppi_i}\tau+1\right)_l}{\left(\frac{\uppi_i}{1-\uppi_i}\tau+1\right)_l}+\operatorname O \left(\frac{1}{\tau^{n_i}}\right)\Bigg).
\end{multline*}
The Pochhammer product \((-l)_{k_1} \cdots (-l)_{k_p}\) is equivalent to \((-l)_{k_1 + \cdots + k_p}\) plus lower-order Pochhammer symbols, with these contributions included in \(\operatorname{O}\left(\dfrac{1}{\tau^{n_i}}\right)\). To further simplify the \(l\)-labeled sum, we will apply \eqref{eq:Gauss_hypergeometric} once more:
\begin{align*}
 \sum_{l=0}^{n_i-1}&\dfrac{(-n_i+1)_l}{l!}
 {{(-l)_{k_1+\cdots+k_p}}}
 \dfrac{\left(x+\frac{\uppi_i}{1-\uppi_i}\tau+1\right)_l}{\left(\frac{\uppi_i}{1-\uppi_i}\tau+1\right)_l}\\
 &=\begin{multlined}[t][.85\textwidth]
 (-1)^{k_1+\cdots+k_p}(-n_i+1)_{k_1+\cdots+k_p}\dfrac{\left(x+\frac{\uppi_i}{1-\uppi_i}\tau+1\right)_{k_1+\cdots+k_p}}{\left(\frac{\uppi_i}{1-\uppi_i}\tau+1\right)_{k_1+\cdots+k_p}}\\
\times\sum_{l=0}^{n_i-1-k_1-\cdots-k_p}\dfrac{(-n_i+1+k_1+\cdots+k_p)_{l}}{l!}
\dfrac{\left(x+\frac{\uppi_i}{1-\uppi_i}\tau+1+k_1+\cdots+k_p\right)_{l}}{\left(\frac{\uppi_i}{1-\uppi_i}\tau+1+k_1+\cdots+k_p\right)_{l}}
 \end{multlined}\\
 &=(-1)^{k_1+\cdots+k_p}(-n_i+1)_{k_1+\cdots+k_p}\dfrac{\left(x+\frac{\uppi_i}{1-\uppi_i}\tau+1\right)_{k_1+\cdots+k_p}}{\left(\frac{\uppi_i}{1-\uppi_i}\tau+1\right)_{n_i-1}}
{(-x)_{n_i-1-k_1-\cdots-k_p}}.
\end{align*}
As a result, the expression is simplified to:
\begin{multline*}
K_{\vec{n}}^{(i)}(x;\vec{\uppi},N)=\lim_{\tau\rightarrow\infty}\dfrac{1}{N!}\dfrac{\tau^N}{(1-\uppi_i)^N}\dfrac{(-1)^{|\vec{n}|-1}(N+1-|\vec{n}|)!}{(n_i-1)!(\tau+1)_{|\vec{n}|-1}\left(\frac{1}{1-\uppi_i}\tau+|\vec{n}|\right)_{N+2-|\vec{n}|}}\dfrac{\prod_{q=1}^p\left(\frac{1}{1-\uppi_q}\tau+|\vec{n}|\right)_{n_q}}{\prod_{q=1,q\neq i}^p\left(\frac{\uppi_q-\uppi_i}{(1-\uppi_q)(1-\uppi_i)}\tau\right)_{{n}_q}}\\
 \times\Bigg(\dfrac{1}{\left(\frac{\uppi_i}{1-\uppi_i}\tau+1\right)_{n_i-1}}
\sum_{k_1=0}^{n_i-1}\cdots\sum_{k_p=0}^{n_i-1-k_1-\cdots-k_{p-1}}(-1)^{k_1+\cdots+k_p}\frac{(-n_i+1)_{k_1+\cdots+k_p}}{k_1!\cdots k_p!}\left(x+\frac{\uppi_i}{1-\uppi_i}\tau+1\right)_{k_1+\cdots+k_p}\\\times
 \frac{(N-|\vec{n}|+2)_{k_i}}{\left(\frac{1}{1-\uppi_i}\tau+N+2\right)_{k_i}}
 \prod_{q=1,q\neq i}^p{\frac{(n_q)_{k_q}}{\left(\frac{\uppi_i-\uppi_q}{(1-\uppi_i)(1-\uppi_q)}\tau+1\right)_{k_q}}}
{(-x)_{n_i-1-k_1-\cdots-k_p}}+\operatorname O \left(\frac{1}{\tau^{n_i}}\right)\Bigg).
\end{multline*}
For \( \tau \rightarrow \infty \), the constant part behaves as follows:
\[
\frac{1}{(n_i-1)!(-N)_{|\vec{n}|-1}(1-\uppi_i)^{2(n_i-1)}} \prod_{q=1, q \neq i}^p \frac{1}{(\uppi_q -\uppi_i)^{n_q}}\,\tau^{n_i-1}.
\]
Using this expression, the limit can be applied to obtain the desired result:
\begin{multline*}
K_{\vec{n}}^{(i)}(x;\vec{\uppi},N)=\dfrac{1}{(n_i-1)!(-N)_{|\vec{n}|-1}
\uppi_i^{n_i-1}(1-\uppi_i)^{n_i-1}}
\prod_{q=1,q\neq i}^p\dfrac{1}{(\uppi_q-\uppi_i)^{n_q}}\\
\times\sum_{k_1=0}^{n_i-1}\cdots\sum_{k_p=0}^{n_i-1-k_1-\cdots-k_{p-1}}\frac{(-n_i+1)_{k_1+\cdots+k_p}}{k_1!\cdots k_p!}{(-\uppi_i)^{k_1+\cdots+k_p}}{(N-|\vec{n}|+2)_{k_i}}\\\times\prod_{q=1,q\neq i}^p(n_q)_{k_q}\left(\frac{1-\uppi_q}{\uppi_i-\uppi_q}\right)^{k_q}
{(-x)_{n_i-1-k_1-\cdots-k_p}}.
\end{multline*}
To express the limit more conveniently, we can perform index changes as follows: \( k_1 \rightarrow n_i - 1 - l_1 \) and \( k_i \rightarrow l_{i-1} - l_i \) for \( i \in \{2, \ldots, p \}\). With these transformations, the sums can be restructured to yield the desired expression \eqref{K:I}.
\end{proof}

\subsection{Integral Representations}
By applying the limits in Corollary \ref{Cor:H->K}, we can derive integral representations from those given for the multiple Hahn polynomials in Theorems \ref{HI_IR} and \ref{HII_IR}.
\begin{teo}
\label{KI_IR}
Let \( \Sigma^\ast \) be a clockwise contour in $\{t\in \mathbb{C} \mid \operatorname{Re}(t)>0\}$ enclosing \( \left\{ \frac{\uppi_j}{1 - \uppi_j}\right\}_{j=1}^p \) exactly once. Then the type I linear forms
are given by
\begin{equation}\label{eq:K}
K_{\vec{n}}^{({\rm I})}(x;\vec{\uppi},N) = \frac{(N-\sz{n}+1)!}{\prod_{j=1}^p (1-\uppi_j)^{n_j}} \frac{1}{\Gamma(x+1)\Gamma(N-x+1)} \bigintss_{\Sigma^\ast} \frac{(1+t)^{\sz{n}-N-2}\, t^{x}}{\prod_{j=1}^p \left(t-\frac{\uppi_j}{1-\uppi_j}\right)^{n_j}} \frac{\d t }{2 \pi \ii }.
\end{equation}
In particular, the type I polynomials are given by
\[ 
K_{\vec{n}}^{(i)}(x;\vec{\uppi},N) = \frac{(N-\sz{n}+1)!}{N! \prod_{j=1}^p (1-\uppi_j)^{n_j}} \frac{1}{\uppi_i^x (1-\uppi_i)^{N-x}} \bigintss_{\Sigma^\ast_i} \frac{(1+t)^{\sz{n}-N-2}\, t^{x}}{\prod_{j=1}^p \left(t-\frac{\uppi_j}{1-\uppi_j}\right)^{n_j}} \frac{\d t }{2 \pi \ii },
 \]
where \( \Sigma_i^\ast \) is a clockwise contour in $\{t\in \mathbb{C} \mid \operatorname{Re}(t)>0\}$ enclosing \( \frac{\uppi_i}{1 - \uppi_i}\) exactly once without enclosing any of the other points in \( \left\{ \frac{\uppi_j}{1 - \uppi_j}\right\}_{j=1}^p \).
\end{teo}
\begin{proof}
 Consider the contour integral for \(Q_{\vec{n}}^{({\rm I})}\left(x;\left\{\frac{\uppi_i}{1 - \uppi_i}\tau\right\}_{i=1}^p, \tau, N\right)\) given in Theorem \ref{HI_IR}. This integral involves a clockwise contour that encloses exactly once the set \(\bigcup_{j=1}^p\left[\frac{\uppi_j}{1 - \uppi_j}\tau, n_j + \frac{\uppi_j}{1 - \uppi_j} \tau - 1\right]\) with 
 \( \operatorname{Re}(t) > -1 \). 
After the change of variables \(t \mapsto \tau t\), the contour can be deformed to one enclosing \(\bigcup_{j=1}^p\left[\frac{\uppi_j}{1 - \uppi_j}, 
n_j -1 + \frac{\uppi_j}{1 - \uppi_j} \right]\) with
 \( \operatorname{Re}(t) > 0 \). 
We can use \eqref{GFB} to obtain the asymptotics for the $\tau$-dependent parts of the integrand and prefactor:
\begin{multline}
 \tau \frac{\Gamma(\tau+N-x+1)}{\Gamma(\tau+\sz{n})} \frac{\Gamma(\tau t + \tau + \sz{n})}{\Gamma(\tau t + \tau + N + 2)} \frac{\Gamma(\tau t + x + 1)}{\Gamma(\tau t + 1)} \frac{\prod_{j=1}^p \left(\frac{1}{1 - \uppi_j} \tau + \sz{n}\right)_{n_j}}{\prod_{j=1}^p \left(\frac{\uppi_j}{1 - \uppi_j} \tau - \tau t\right)_{n_j}} \\\begin{aligned}
 &\sim \tau^{N-x-\sz{n}+2} ((1+t)\tau)^{\sz{n}-N-2} (\tau t)^{x} \frac{\prod_{j=1}^p \left(\frac{1}{1 - \uppi_j}\right)^{n_j}}{\prod_{j=1}^p \left(\frac{\uppi_j}{1 - \uppi_j} - t\right)^{n_j}}= \frac{t^x}{(1+t)^{N-\sz{n}+2}} \frac{\prod_{j=1}^p \left(\frac{1}{1 - \uppi_j}\right)^{n_j}}{\prod_{j=1}^p \left(\frac{\uppi_j}{1 - \uppi_j} - t\right)^{n_j}},\quad \tau\to\infty .
\end{aligned}
\end{multline}
 To obtain the desired result about the type I linear forms, we use limit \eqref{H->K:MOP} and deform the contour to \(\Sigma^\ast\). Once these integral representations are established, one can note that only certain poles contribute to a specific type I polynomial, with the remaining poles relating to different weight functions.
\end{proof}

An application of the residue theorem leads to the following formulas for the type I polynomials.
\begin{coro}
The following Rodrigues-type formula for the type I polynomials \eqref{K:I} holds:
 \begin{align*}
 K_{\vec{n}}^{(i)}(x;\vec{\uppi},N) = \frac{(-1)^{\sz{n}}}{(-N)_{\sz{n}-1} (n_i-1)! \prod_{j=1}^p (1-\uppi_j)^{n_j}} \frac{1}{\uppi_i^x (1-\uppi_i)^{N-x}} \left. \frac{\d^{n_i-1}}{\d t^{n_i-1}} \right|_{t=\frac{\uppi_i}{1-\uppi_i}} \left(\frac{(1+t)^{\sz{n}-N-2}\, t^{x}}{\prod_{j=1,j\neq i}^p \left(t-\frac{\uppi_j}{1-\uppi_j}\right)^{n_j}}\right).
 \end{align*} 
\end{coro}

\begin{teo}
Let \( \mathcal{C}^\ast\) be an counterclockwise contour enclosing the origin exactly once. Then, for \(x \in \mathbb{N}_0\), the type II polynomials are given by
 \[
 K_{\vec{n}}^{({\rm II})}(x;\vec{\uppi},N) = \frac{\prod_{j=1}^p (1-\piup_j)^{n_j}}{(N-\sz{n})!} \Gamma(x+1)\Gamma(N-x+1) \bigintss_{\mathcal{C}^\ast} \frac{(1+s)^{N-\sz{n}}}{s^{x+1}}\prod_{j=1}^p \left(s-\frac{\uppi_j}{1-\uppi_j}\right)^{n_j} \frac{\d s }{2 \pi \ii }.
 \]
\end{teo}
\begin{proof}
Consider the integral representation for
\( Q_{\vec{n}}^{({\rm II})} \left(x;\left\{\frac{\uppi_i}{1-\uppi_i}\tau\right\}_{i=1}^p, \tau, N\right) \) 
given in Theorem \ref{HII_IR}. This representation involves an counterclockwise contour enclosing the interval \([-N,0]\). After changing variables with \( s \mapsto \tau s \), we can deform this contour to a contour \(\mathcal{C}\) enclosing \([-N,0]\). We can use \eqref{GFB} to obtain the asymptotics for the $\tau$-dependent parts of the integrand and prefactor:
\begin{multline*}
 \tau \frac{\Gamma(\sz{n}+\tau+1)}{\Gamma(N-x+\tau+1)} \frac{\Gamma(\tau s + N + \tau + 1) }{\Gamma(\tau s +\sz{n}+\tau+1)} \frac{\Gamma(\tau s)}{\Gamma(\tau s+x+1)} \frac{\prod_{j=1}^p \left(\frac{\uppi_j}{1 - \uppi_j}\tau + 1 - \tau s\right)_{n_j}}{\prod_{j=1}^p \left(\frac{1}{1 - \uppi_j}\tau + \sz{n} + 1\right)_{n_j}} \\\begin{aligned}&\sim
\tau^{\sz{n}-N+x+1} ((1+s)\tau)^{N-\sz{n}} (\tau s)^{-x-1} \frac{\prod_{j=1}^p \left(\frac{\uppi_j}{1 - \uppi_j} - s\right)^{n_j}}{\prod_{j=1}^p \left(\frac{1}{1 - \uppi_j}\right)^{n_j}}= \frac{(1+s)^{N-\sz{n}}}{s^{x+1}} \frac{\prod_{j=1}^p \left(\frac{\uppi_j}{1 - \uppi_j} - s\right)^{n_j}}{\prod_{j=1}^p \left(\frac{1}{1 - \uppi_j}\right)^{n_j}},\quad \tau\to\infty.
\end{aligned}
\end{multline*}
Combining these results with limit \eqref{H->K:MOP} and deforming the contour to \(\mathcal{C}^\ast\) then yields the desired result.
\end{proof}

\subsection{Recurrence Coefficients}

We will now present explicit expressions for the recurrence coefficients.

\begin{teo}
The Kravchuk multiple orthogonal polynomials, of type I \eqref{K:I} and type II \eqref{K:II}, satisfy their respective recurrence relations \eqref{NNRR} with coefficients:
\begin{align}
\label{K:R}
&\begin{aligned}
b_{\vec{n}}^{(K),0}(k)&=(N-|\vec{n}|)\uppi_k+
\sum_{i=1}^p n_i(1-\uppi_i),\\
 b_{\vec{n}}^{(K),j}&=(N-|\vec{n}|+1)_{j}\sum_{i\in S(\pi,j)}
{n_i\uppi_i(1-\uppi_i)}
\prod_{q\in S^{\textsc{c}}(\pi,j)}{(\uppi_i-\uppi_q)},\quad j\in\{1,\ldots,p\}.
\end{aligned}
\end{align}
\end{teo}
\begin{proof}
Here we have to apply the limit \eqref{H->K:R} to the Hahn coefficients \eqref{HahnRecurrence}. For \( j \in \{1, \ldots, p\} \), applying the limit is relatively straightforward. On the other hand, for \( j = 0 \), additional steps are necessary. In this scenario, the expression simplifies to:
\begin{multline*}
 b_{\vec{n}}^{(K),0}(k)
 =\lim_{\tau\rightarrow\infty}\Bigg(
\dfrac{(n_k+1)\left(\frac{\uppi_k}{1-\uppi_k}\tau+n_k+1\right)\left(\frac{1}{1-\uppi_k}\tau+n_k+N+2\right)}{\frac{1}{1-\uppi_k}\tau+n_k+|\vec{n}|+2}
\prod_{q=1,q\neq k}^p\dfrac{\frac{\uppi_k-\uppi_q}{(1-\uppi_k)(1-\uppi_q)}\tau+n_k+1}{\frac{\uppi_k-\uppi_q}{(1-\uppi_k)(1-\uppi_q)}\tau-n_q+n_k+1}\\
-
\dfrac{n_k\left(\frac{\uppi_k}{1-\uppi_k}\tau+n_k\right)\left(\frac{1}{1-\uppi_k}\tau+n_k+N+1\right)}{\frac{1}{1-\uppi_k}\tau+n_k+|\vec{n}|}
\prod_{q=1,q\neq k}^p\dfrac{\frac{\uppi_k-\uppi_q}{(1-\uppi_k)(1-\uppi_q)}\tau+n_k}{\frac{\uppi_k-\uppi_q}{(1-\uppi_k)(1-\uppi_q)}\tau-n_q+n_k}-\left(\frac{\uppi_k}{1-\uppi_k}\tau+n_k+1\right)\Bigg)\\
+\underbrace{\lim_{\tau\rightarrow\infty}\left(\frac{1}{1-\uppi_k}\tau+|\vec{n}|+n_k+1\right)\sum_{i=1,i\neq k}^p\dfrac{n_i\left(\frac{\uppi_i}{1-\uppi_i}\tau+n_i\right)\left(\frac{1}{1-\uppi_i}\tau+N+n_i+1\right)}{\left(\frac{\uppi_i-\uppi_k}{(1-\uppi_i)(1-\uppi_k)}\tau-n_k-1+n_i\right)\left(\frac{1}{1-\uppi_i}\tau+|\vec{n}|+n_i\right)_2}\prod_{q=1,q\neq i}^p\dfrac{\frac{\uppi_i-\uppi_q}{(1-\uppi_i)(1-\uppi_q)}\tau+n_i}{\frac{\uppi_i-\uppi_q}{(1-\uppi_i)(1-\uppi_q)}\tau-n_q+n_i}}_{=
\sum_{i=1,i\neq k}^p 
\frac{n_i\uppi_i(1-\uppi_i)}{\uppi_i-\uppi_k}}.
\end{multline*}
The limit of the second term on the right-hand side is straightforward to evaluate, as previously demonstrated. The focus now shifts to the first term. To find this limit, we need to reframe the expression as follows:
\begin{multline*}
 \begin{multlined}[t][\textwidth]
 (n_k+1)\left(\frac{\uppi_k}{1-\uppi_k}\tau+n_k+1\right)\left(1+\dfrac{N-|\vec{n}|}{\frac{1}{1-\uppi_k}\tau+n_k+|\vec{n}|+2}\right)
\prod_{q=1,q\neq k}^p\left(1+\dfrac{n_q}{\frac{\uppi_k-\uppi_q}{(1-\uppi_k)(1-\uppi_q)}\tau-n_q+n_k+1}\right)\\
-
n_k\left(\frac{\uppi_k}{1-\uppi_k}\tau+n_k\right)\left(1+\dfrac{N-|\vec{n}|+1}{\frac{1}{1-\uppi_k}\tau+n_k+|\vec{n}|}
\right)\prod_{q=1,q\neq k}^p\left(1+\dfrac{n_q}{\frac{\uppi_k-\uppi_q}{(1-\uppi_k)(1-\uppi_q)}\tau-n_q+n_k}\right)-\left(\frac{\uppi_k}{1-\uppi_k}\tau+n_k+1\right)
 \end{multlined}\\
\begin{aligned}
 &=\begin{multlined}[t][.8\textwidth]
 (n_k+1)\left(\frac{\uppi_k}{1-\uppi_k}\tau+n_k+1\right)\left(1+\dfrac{N-|\vec{n}|}{\frac{1}{1-\uppi_k}\tau+n_k+|\vec{n}|+2}+
\sum_{q=1,q\neq k}^p\dfrac{n_q}{\frac{\uppi_k-\uppi_q}{(1-\uppi_k)(1-\uppi_q)}\tau-n_q+n_k+1}+\operatorname O\left(\dfrac{1}{\tau^2}\right)\right)\\
-
n_k\left(\frac{\uppi_k}{1-\uppi_k}\tau+n_k\right)\left(1+\dfrac{N-|\vec{n}|+1}{\frac{1}{1-\uppi_k}\tau+n_k+|\vec{n}|}
+\sum_{q=1,q\neq k}^p\dfrac{n_q}{\frac{\uppi_k-\uppi_q}{(1-\uppi_k)(1-\uppi_q)}\tau-n_q+n_k}+\operatorname O\left(\dfrac{1}{\tau^2}\right)\right)-\left(\frac{\uppi_k}{1-\uppi_k}\tau+n_k+1\right)
\end{multlined}\\
&=\begin{multlined}[t][.8\textwidth]n_k+(n_k+1)\left(N-\sz{n}\right)\dfrac{\frac{\uppi_k}{1-\uppi_k}\tau+n_k+1}{\frac{1}{1-\uppi_k}\tau+n_k+|\vec{n}|+2}
+
(n_k+1)\sum_{q=1,q\neq k}^pn_q\dfrac{\frac{\uppi_k}{1-\uppi_k}\tau+n_k+1}{\frac{\uppi_k-\uppi_q}{(1-\uppi_k)(1-\uppi_q)}\tau-n_q+n_k+1}\\
-
n_k\left(N-|\vec{n}|+1\right)\dfrac{\frac{\uppi_k}{1-\uppi_k}\tau+n_k}{\frac{1}{1-\uppi_k}\tau+n_k+|\vec{n}|}
-
n_k\sum_{q=1,q\neq k}^pn_q\dfrac{\frac{\uppi_k}{1-\uppi_k}\tau+n_k}{\frac{\uppi_k-\uppi_q}{(1-\uppi_k)(1-\uppi_q)}\tau-n_q+n_k}+\operatorname O\left(\dfrac{1}{\tau}\right)\end{multlined}\\
&\sim (N-\sz{n})\uppi_k+n_k(1-\uppi_k)+\uppi_k\sum_{i=1,i\neq k}^p n_i\frac{1-\uppi_i}{\uppi_k-\uppi_i},\quad \tau\to\infty.
\end{aligned}
\end{multline*}
Combining with the other term and simplifying yields the recurrence relation given in \eqref{K:R}.
\end{proof}

\section{Charlier 
}

For the multiple Charlier polynomials, the weight functions are given by
\begin{align}
 \label{WeightsC}
 w_i^{(C)}(x; a_i) &= \frac{a_i^x}{\Gamma(x+1)},\quad i \in \{1, \ldots, p\},\quad \Delta = \mathbb{N}_0.
\end{align}
Here \( a_1, \ldots, a_p > 0\) and, to ensure an AT system, \( a_i \neq a_j \) for all \( i \neq j \). We will denote \(\vec{a}\coloneq(a_1,\cdots,a_p)\).

\begin{pro}
The following limiting relations between the Meixner of second kind \eqref{WeightsM2}, Meixner of first kind \eqref{WeightsM1}, Kravchuk \eqref{WeightsK} and Charlier weights hold:
\[\begin{aligned}
 w_i^{(C)}(x;a_i)&=\lim_{c\to0} w_i^{(M2)}\left(x;\frac{a_i}{c},c\right)= \lim_{\beta\to\infty} w_i^{(M1)}\left(x;\beta,\frac{a_i}{\beta}\right)=\lim_{N\rightarrow\infty}\Exp{a_i}w^{(K)}_i\left(x;\frac{a_i}{N},N\right),\quad i\in\{1,\ldots,p\}.
\end{aligned}\]
\begin{proof}
This follows from the asymptotics: 
 \begin{align*}
 w_i^{(M2)}\left(x;\frac{a_i}{c},c\right)&=\begin{aligned}
 &\frac{\Gamma\left(x+\frac{a_i}{c}\right)}{\Gamma(x+1)\Gamma\left(\frac{a_i}{c}\right)}c^x\sim\frac{a_i^x}{\Gamma(x+1)},& c&\to0,
 \end{aligned}\\
 w_i^{(M1)}\left(x;\beta,\frac{a_i}{\beta}\right)&=\begin{aligned}
 &\frac{\Gamma(\beta+x)}{\Gamma(x+1)\Gamma(\beta)}\frac{a_i^x}{\beta^x}\sim\frac{a_i^x}{\Gamma(x+1)},& \beta&\to\infty,
 \end{aligned}\\
 w_i^{(K)}\left(x;\frac{a_i}{N},N\right)&=\begin{aligned}&
 \frac{\Gamma(N+1)}{\Gamma(x+1)\Gamma(N-x+1)}\frac{a_i^x}{N^x}\left(1-\frac{a_i}{N}\right)^{N-x}\sim\frac{a_i^x\Exp{-a_i}}{\Gamma(x+1)},& N&\to\infty.
 \end{aligned}
 \end{align*}
\end{proof}
\end{pro}

\begin{coro}
\label{Cor:M/K->C}
 For the linear forms $C_{\vec{n}}^{({\rm I})}$, type I polynomials $C_{\vec{n}}^{(i)}$, type II polynomials $C_{\vec{n}}^{({\rm II})}$ and recurrence coefficients $b_{\vec{n}}^{(C),j}$; the following respective limiting relations from Meixner second kind, Meixner first kind and Kravchuk hold:
\begin{subequations}
\label{M/K->C}
 \begin{align}
 \label{M/K->C:MOP}
 C_{\vec{n}}^{(\ast)}(x;\vec{a})&=\lim_{c\to0} M_{2:\vec{n}}^{(\ast)}\left(x;\frac{\vec{a}}{c},c\right)= \lim_{\beta\to\infty} M_{1:\vec{n}}^{(\ast)}\left(x;\beta,\frac{\vec{a}}{b}\right)= \lim_{N \rightarrow \infty} \kappa_N^{(\ast)}K_{\vec{n}}^{(\ast)} \left(x; \frac{\vec{a}}{N}, N \right) , \\
 \label{M/K->C:R}
 b_{\vec{n}}^{(C),j}(\vec{a}) &=\lim_{c\to0}b_{\vec{n}}^{(M2),j}\left(\frac{\vec{a}}{c},c\right)=\lim_{\beta\rightarrow\infty}b_{\vec{n}}^{(M1),j}\left(\beta,\frac{\vec{a}}{\beta}\right)=
 \begin{aligned}
 &\lim_{N\rightarrow\infty}b_{\vec{n}}^{(K),j}\left(\frac{\vec{a}}{N},N\right),
 \end{aligned}
 \end{align}
\end{subequations}
where
\begin{equation*}
 \kappa_N^{({\rm I})}=\kappa_N^{({\rm II})}=1,\quad \kappa_N^{(i)}=\Exp{-a_i}.
\end{equation*}
\end{coro}
\begin{rem}
 The limiting relations arising from Meixner first kind and Kravchuk for the type II polynomials were already deduced at \cite{AskeyII}.
\end{rem}


\subsection{Hypergeometric representation}

Through the previous limits, the monic Charlier type II polynomials are, cf. \cite{AskeyII}:
\begin{equation}
 \label{C:II}\begin{aligned}
 C_{\vec{n}}^{({\rm II})}(x;\vec{a})
 &=\prod_{i=1}^p\left(-a_i\right)^{n_i}\KF{1:1;\cdots;1}{0:0;\cdots;0}{-x:-n_1;\cdots;-n_p}{--:--;\cdots;--}{-\dfrac{1}{a_1},\ldots,-\dfrac{1}{a_p}}.
\end{aligned}
\end{equation}

Below, we will derive a hypergeometric representation for the type I Charlier polynomials. For $p=2$, an alternative representation was found in \cite[Proposition 9.1]{HahnI}.

\begin{teo}
The type I Charlier polynomials can be expressed as follows:
\begin{align}
 \label{C:I}
 C_{\vec{n}}^{(i)}(x;\vec{a}) = \frac{(-1)^{n_i-1} \Exp{-a_i}}{(n_i-1)! \prod_{q=1, q \neq i}^p (a_i - a_q)^{n_q}} 
 & \KF{1:1;\cdots;1}{0:0;\cdots;0}{-n_i+1:-x;\;\vec{n}^{*i}}{--:--;\cdots;--}{-\frac{1}{a_i}, \left\{\frac{1}{a_q - a_i}\right\}_{q=1,q\neq i}^p},
\end{align}
for \( i \in \{1, \ldots, p\} \).
\end{teo}

\begin{proof}
 Consider the limits in Corollary \ref{Cor:M/K->C}.
 The one from Meixner second kind \eqref{M2:I} can be done applying the same technique as in sections $\S 4$ and $\S 5$. The one from Kravchuk \eqref{K:I} is straightforward. The one from Meixner first kind is straightforward only 
 with expression \eqref{M1:I_alt} but not with \eqref{M1:I}. The three limits lead to~\eqref{C:I}.
\end{proof}

\subsection{Integral representations}

For the type I polynomials, we obtain the integral representations below.

\begin{teo}
\label{CI_IR}
Let \( \Sigma^\ast \) be an clockwise contour in $\{t\in \mathbb{C} \mid 0<\operatorname{Re}(t)<1\}$ enclosing \(\{a_j\}_{j=1}^p\) exactly once. Then the type I linear forms
are expressed as
\[ 
C_{\vec{n}}^{({\rm I})}(x;\vec{a}) = \frac{1}{\Gamma(x+1)} \bigintsss_{\Sigma^\ast}\frac{\operatorname{e}^{-t} t^{x}}{\prod_{j=1}^p (t-a_j)^{n_j}} \frac{\d t }{2 \pi \ii }.
\]
In particular, the type I polynomials are given by
\[ 
C_{\vec{n}}^{(i)}(x;\vec{a}) = \frac{1}{a_i^x} \bigintsss_{\Sigma^\ast_i}\frac{\operatorname{e}^{-t} t^{x}}{\prod_{j=1}^p (t-a_j)^{n_j}} \frac{\d t }{2 \pi \ii },
 \]
where \(\Sigma_i^\ast\) is a clockwise contour in $\{t\in \mathbb{C} \mid 0<\operatorname{Re}(t)<1\}$ enclosing $a_i$ exactly once without enclosing any of the other points in \(\{a_j\}_{j=1}^p\).
\end{teo}

\begin{proof}
 Consider the integral representation for \(M_{1:\vec{n}}^{({\rm I})}\left(x;\beta, \frac{\vec{a}}{\beta}\right)\) given in Theorem \ref{M1I_IR}. This representation involves an clockwise contour enclosing once \(\left\{\frac{a_1}{\beta}, \ldots, \frac{a_p}{\beta}\right\}\) with \(0 < \operatorname{Re}(t) < 1\). After the change of variables \(t \mapsto \frac{t}{\beta}\), we can deform the contour to \(\Sigma^\ast\). We can use \eqref{GFB} to obtain the asymptotics for the $\beta$-dependent parts of the integrand and prefactor:
\[\begin{aligned}
 \frac{1}{\beta} \prod_{j=1}^p (1 - a_j / \beta)^{n_j} \frac{\Gamma(x+\beta)}{\Gamma(\beta + \sz{n} - 1)} \frac{(1-t/\beta)^{\sz{n}+\beta-2} (t/\beta)^{x}}{\prod_{j=1}^p (t/\beta - a_j/\beta)^{n_j}} &\sim \frac{\operatorname{e}^{-t} t^x}{\prod_{j=1}^p (t - a_j)^{n_j}}, & \beta&\to\infty.
\end{aligned}\]
 Hence together with limit \eqref{M/K->C:MOP}, the stated integral representation for the linear forms follows. From this, we note that only certain poles contribute to a specific type I polynomial as the remaining poles are related to different weights.
\end{proof}

Applying the residue theorem leads to the Rodrigues-type formula below.

\begin{coro}
For the type I polynomials \eqref{C:I}, we have the Rodrigues-type formula
\[ 
 C_{\vec{n}}^{(i)}(x;\vec{a}) = - \frac{1}{(n_i-1)!} \frac{1}{a_i^x} \frac{\d^{n_i-1}}{\d a_i^{n_i-1}} \left(\frac{\operatorname{e}^{-a_i}a_i^x}{\prod_{j=1,j\neq i}^p (a_i-a_j)^{n_j}}\right).
 \]
\end{coro}

For the type II Charlier polynomials, we find the following integral representation.
\begin{teo}
Let \( \mathcal{C}^\ast \) be a clockwise contour in $\{s\in \mathbb{C} \mid 0<\operatorname{Re}(s)<1\}$ enclosing the origin exactly once. Then, for \(x \in \mathbb{N}_0\), the type II polynomials are given by 
 \[
 C_{\vec{n}}^{({\rm II})}(x;\vec{a}) = \Gamma(x+1) \bigintssss_{\mathcal{C}^\ast}\frac{\operatorname{e}^{s} }{s^{x+1}}\prod_{j=1}^p(s-a_j)^{n_j} \frac{\d s }{2 \pi \ii }.
 \]
\end{teo}
\begin{proof}
 Consider the integral representation for \( M_{1:\vec{n}}^{({\rm II})}\left(x;\beta,\frac{\vec{a}}{\beta}\right) \) given in Theorem \ref{M1II_IR}. This representation involves a clockwise contour enclosing the origin with \(\operatorname{Re}(s) < 1\). After changing variables to \( s \mapsto \frac{s}{\beta} \), we can deform the contour to \( \mathcal{C}^\ast \). We can use \eqref{GFB} to obtain the asymptotics for the $\beta$-dependent parts of the integrand and prefactor:
\begin{align*}
 & \frac{1}{\beta} \frac{1}{\prod_{j=1}^p (1 - a_j/\beta)^{n_j}} \frac{\Gamma(\beta + \sz{n})}{\Gamma(x + \beta)} \frac{\prod_{j=1}^p (s/\beta - a_j/\beta)^{n_j}}{(1 - s/\beta)^{\sz{n} + \beta} (s/\beta)^{x + 1}} \sim \frac{\operatorname{e}^s}{s^{x+1}} \prod_{j=1}^p (s - a_j)^{n_j} ,\quad \beta\to\infty.
\end{align*} 
 Combining this with limit \eqref{M/K->C:MOP} yields the desired result.
\end{proof}

\subsection{Recurrence coefficients}
For the Charlier recurrence coefficients, we have the following explicit expressions.

\begin{teo}
The Charlier polynomials, both of type I \eqref{C:I} and type II \eqref{C:II}, satisfy respective recurrence relations \eqref{NNRR} with coefficients:
\begin{equation}
\label{C:R}
\begin{aligned}
b_{\vec{n}}^{(C),0}(k)&=a_k+
\sz{n},\\
 b_{\vec{n}}^{(C),j}&=\sum_{i\in S(\pi,j)}
{n_ia_i}
\prod_{q\in S^{\textsc{c}}(\pi,j)}{(a_i-a_q)},& j&\in\{1,\ldots,p\}.
\end{aligned}
\end{equation}
\end{teo}

\begin{proof}
Consider limits \eqref{M/K->C:R}. The one from Meixner second kind \eqref{M2:RR} can be done by applying the same technique as in sections $\S4$ and $\S5$, while the ones from Meixner first kind \eqref{M1:R} and Kravchuk \eqref{K:R} are straightforward.
All of them lead to \eqref{C:R}.
\end{proof}


\section{From Discrete to Continuous Families}

Recent research has uncovered explicit formulations for the type I classical multiple orthogonal polynomials and their recurrence for the continuous families of Jacobi–Piñeiro, Laguerre, and Hermite polynomials (see \cite{CMOPI} and \cite{HS:JP-L1}). Then, integral representations for the linear forms and type II polynomials were found for the Jacobi--Piñeiro and Laguerre first kind families \cite{BDFMW}. For the Laguerre first kind and Hermite families, such representations had already been described by Bleher and Kuijlaars in \cite{B-K} in the context of random matrices. These integral representations coincide with those obtained by taking the appropriate limits in the multiple Askey scheme. A notable distinction is that Bleher and Kuijlaars have some additional conditions on the parameters, which simplifies the analysis for the type II polynomials and allows for the use of a simpler contour which is not applicable in the more general setting. The discrete families discussed in this study are related to these continuous families through limit processes, as illustrated by Figure \ref{figure:Askey scheme}.
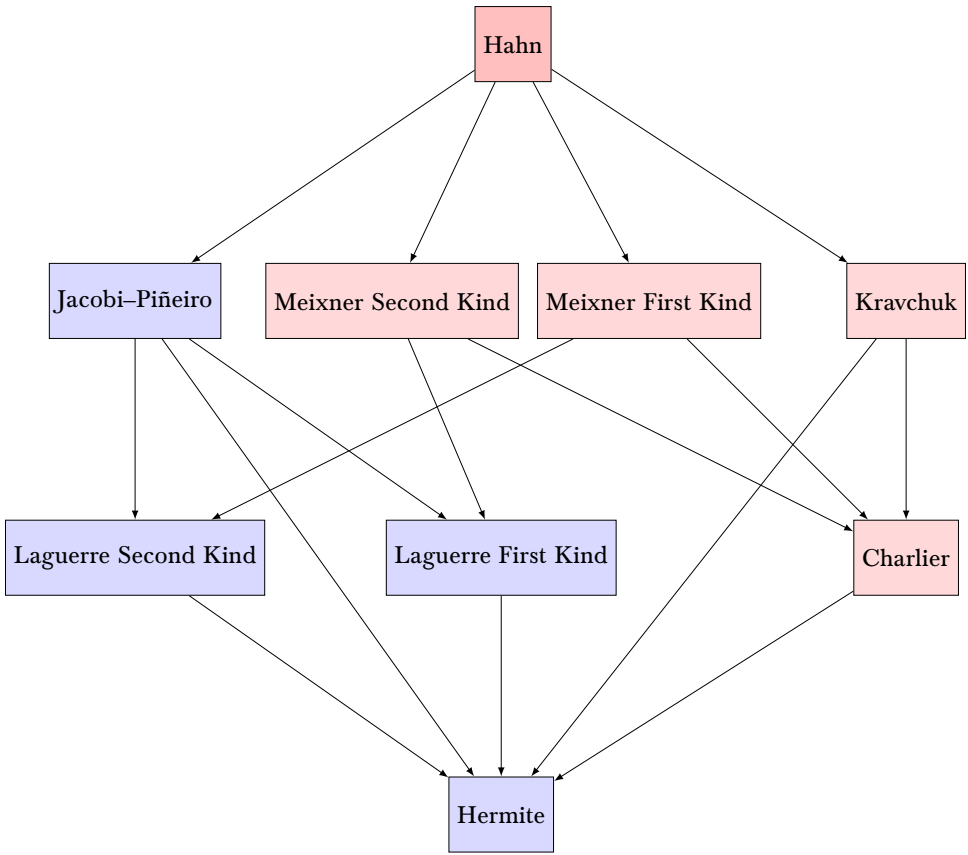
\begin{figure}[htbp]
 \scalebox{.865}{\begin{tikzpicture}[node distance=3.95cm]
 \node[fill=blue!15,block] (a) {Jacobi--Piñeiro}; 
 \node[ fill=red!15,block, right of=a] (b) {Meixner Second Kind};
 \node[ fill=red!15,block, right of = b] (c) {Meixner First Kind};
 \node[ fill=red!15,block, right of=c] (f) {Kravchuk}; 
 \node[ fill=red!25,block,above of = c,left=1.5cm] (d) {Hahn};
 
 \node[fill=blue!15,block, below of = a] (g) {Laguerre Second Kind};
 \node[fill=blue!15,block, below of = c,left=0.5cm] (e) {Laguerre First Kind};
 
 \node[ fill=red!15,block, below of =f] (i) {Charlier};
 \node[ fill=blue!15,block, below of =e] (j) {Hermite};
 \draw[-latex] (d)--(a); 
 \draw[-latex] (d)--(b);
 \draw[-latex] (d)--(c);
 \draw[-latex] (d)--(f);
 \draw[-latex] (a)--(e);
 \draw[-latex] (b)--(e);
 \draw[-latex] (c)--(i);
 \draw[-latex] (c)--(g);
 \draw[-latex] (f)--(i);
 \draw[-latex] (a)--(g);
 \draw[-latex] (e)--(j);
 \draw[-latex] (a)--(j);
 \draw[-latex] (f)--(j);
 \draw[-latex] (i)--(j);
 \draw[-latex] (g)--(j);
 \draw[-latex] (b)--(i);
 \end{tikzpicture} }
\caption{Askey scheme for Hahn descendants}
\label{figure:Askey scheme}
\end{figure}
In what follows, we detail the limits that must be taken in the previous partial multiple Askey scheme from discrete to continuous families, with a focus on the set of weights, the type I and type II polynomials, and the recurrence coefficients.

\subsection{Jacobi--Piñeiro}
For this family, the weight functions are defined as
\begin{align}
 \label{WeightsJP}
 w_i^{(JP)}(x;\alpha_i,\beta)=x^{\alpha_i}(1-x)^\beta, \quad i \in \{1,\ldots,p\}, \quad \Delta=[0,1],
\end{align}
where \( \alpha_1, \ldots, \alpha_p, \beta > -1 \) and, to ensure an AT system, \( \alpha_i - \alpha_j \not\in \mathbb{Z} \) for \( i \neq j \). We will denote \(\vec{\alpha}\coloneq(\alpha_1,\cdots,\alpha_p)\). 

The Jacobi--Piñeiro family can be reached from Hahn \eqref{WeightsHahn} through the following limit.

\begin{pro}
 The following limiting relations between the Hahn \eqref{WeightsHahn} and Jacobi--Piñeiro weights hold
 \begin{equation}
 \label{H->JP:W}
 w_i^{(JP)}(x;\alpha_i,\beta)=\lim_{N\rightarrow\infty} \dfrac{\Gamma(\alpha_i+1)\Gamma(\beta+1)}{N^{\alpha_i+\beta}}w_i(Nx;\alpha_i,\beta,N),\quad i\in\{1,\ldots,p\}.
 \end{equation}
\end{pro}

\begin{proof}
 This follows from the asymptotics
 \begin{equation*}
 w_i(Nx;\alpha_i,\beta,N)=\dfrac{\Gamma(\alpha_i+Nx+1)}{\Gamma(Nx+1)\Gamma(\alpha_i+1)}\dfrac{\Gamma(N-Nx+\beta+1)}{\Gamma(N-Nx+1)\Gamma(\beta+1)}\sim\dfrac{N^{\alpha_i+\beta}}{\Gamma(\alpha_i+1)\Gamma(\beta+1)}x^{\alpha_i}(1-x)^{\beta},
 \end{equation*}
 \( N\rightarrow\infty \).
\end{proof}

\begin{coro}
\label{Cor:H->JP}
For the linear forms \( P_{\vec{n}}^{({\rm I})} \), type I polynomials \( P_{\vec{n}}^{(i)} \), the type II polynomials \( P_{\vec{n}}^{({\rm II})} \) and the recurrence coefficients \(b_{\vec{n}}^{(JP),j}\); the following respective limiting relations from Hahn hold:
\begin{subequations}
 \begin{align}
 P^{(\ast)}_{\vec{n}}(x;\vec{\alpha},\beta)&=\lim_{N\rightarrow\infty} \kappa^{(\ast)}_N Q_{\vec{n}}^{(\ast)}(Nx;\vec{\alpha},\beta,N),\quad 
 b_{\vec{n}}^{(JP),j}(\vec{\alpha},\beta) = \lim_{N\rightarrow\infty}\dfrac{1}{N^{j+1}}b^j_{\vec{n}}(\vec{\alpha},\beta,N),
\end{align}
\end{subequations}
where 
$$ \kappa^{({\rm I})}_N = \frac{1}{\kappa^{({\rm II})}_N} = N^{|\vec{n}|},\quad \kappa^{(i)}_N = \frac{N^{\alpha_i+\beta+|\vec{n}|}}{\Gamma(\alpha_i+1)\Gamma(\beta+1)}. $$
\end{coro}

\begin{proof}
In \cite[Theorem 3.8]{BDFMW}, one derived the following integral representation for the linear forms
\begin{align*}
 P_{\vec{n}}^{({\rm I})}(x) = (-1)^{\sz{n}}(1-x)^{\beta} \frac{\prod_{i=1}^p(\alpha_i+\beta+\sz{n})_{n_i}}{\Gamma(\beta+\sz{n})}\bigintssss_\Sigma \frac{\Gamma(t+\beta+\sz{n})}{\Gamma(t+1)} \frac{x^t}{\prod_{i=1}^p(\alpha_i-t)_{n_i}} \frac{ \operatorname d t }{2 \pi \ii },
\end{align*}
 where \( \Sigma\subset\{t\in \mathbb{C} \mid \operatorname{Re}(t)>-1\}\) is a clockwise contour enclosing once \(\cup_{i=1}^p[\alpha_i,n_i+\alpha_i-1]\). This expression can also be obtained by applying the previous mentioned limit over the integral representation for the Hahn linear forms 
 in Theorem \ref{HI_IR}. The remaining limits can then be deduced from this one.
\end{proof}
\begin{rem}
 The limiting relation for the type II polynomials was already deduced at \cite{AskeyII}.
\end{rem}

\subsection{Laguerre of the First Kind}
The weight functions are defined as
\begin{equation}
\label{WeightsL1}
\begin{aligned}
 w_i^{(L1)}(x;\alpha_i)&=x^{\alpha_i}\Exp{-x},& i&\in\{1,\ldots,p\},&\Delta&=[0,\infty),
\end{aligned}
\end{equation}
where \( \alpha_1,\ldots,\alpha_p>-1\) and, to ensure an AT system, \( \alpha_i-\alpha_j\not\in\mathbb Z\) for \( i\neq j\). We will denote \(\vec{\alpha}\coloneq(\alpha_1,\cdots,\alpha_p)\). The Laguerre of the first kind family can be derived from the Jacobi--Piñeiro and Meixner of the second kind family through certain limiting relations. The limit from Jacobi--Piñeiro has already been studied in \cite{AskeyII} for the type II polynomials, in \cite{HS:JP-L1} for the type I polynomials and in \cite{CMOPI} for the recurrence coefficients. Here, we will show the following.
\begin{pro}
 The following limiting relations between the Meixner of second kind \eqref{WeightsM2} and the Laguerre of the first kind weights hold
 \begin{equation}
 \label{M2->L1:W}
\begin{aligned}
 w_i^{(L1)}(x;\alpha_i)&=\lim_{c\rightarrow1}(1-c)^{\alpha_i}\Gamma(\alpha_i+1)w_i^{(M2)}\left(\frac{x}{1-c};\alpha_i+1,c\right),& i&\in\{1,\ldots,p\}.
\end{aligned}
 \end{equation}
\end{pro}

\begin{proof}
This follows from the asymptotics
 \begin{equation*}
\begin{aligned}
 w_i^{(M2)}\left(\frac{x}{1-c};\alpha_i+1,c\right)&=\frac{\Gamma\left(\alpha_i+\frac{x}{1-c}+1\right)}{\Gamma(\alpha_i+1)\Gamma\left(\frac{x}{1-c}+1\right)}c^{\frac{x}{1-c}}\sim \frac{1}{\Gamma(\alpha_i+1)(1-c)^{\alpha_i}}\,x^{\alpha_i}\Exp{-x},& c&\rightarrow1.
\end{aligned}
 \end{equation*}
\end{proof}

\begin{coro}
For the linear forms \( L_{1:\vec{n}}^{({\rm I})} \), type I polynomials \( L_{1:\vec{n}}^{(i)} \), the type II polynomials \( L^{({\rm II})}_{1:\vec{n}} \) and the recurrence coefficients \(b^{(L1),j}_{\vec{n}}\); the following respective limiting relations from Meixner of the second kind hold:
\begin{subequations}
 \label{M2->L1}
\begin{align}
 \label{M2->L1:MOP}
 L_{1:\vec{n}}^{(\ast)}(x;\vec{\alpha}) &= \lim_{c \to 1} \kappa^{(\ast)}_c M_{2:\vec{n}}^{(\ast)}\left(\frac{x}{1 - c};\vec{\alpha}+\vec{1}_p, c\right),\quad 
 b^{(L1),j}_{\vec{n}}(\vec{\alpha})=\lim_{c\rightarrow1}{(1-c)^{j+1}b^{(M2),j}_{\vec{n}}(\vec{\alpha}+\vec{1}_p,c)}
\end{align}
\end{subequations}
where
$$\kappa^{({\rm I})}_c= \frac{1}{\kappa^{({\rm II})}_c} = \frac{1}{(1 - c)^{|\vec{n}|}},\quad \kappa^{(i)}_c= \frac{1}{\Gamma(\alpha_i+1)(1 - c)^{|\vec{n}|+\alpha_i}}.$$
\end{coro}

\begin{proof}
In \cite[Theorem 3.3]{BDFMW}, one obtained the following integral representation for the linear forms
\begin{align*}
 L_{1:\vec{n}}^{({\rm I})}(x) = (-1)^{\sz{n}}\Exp{-x} \int_\Sigma \frac{1}{\Gamma(t+1)} \frac{x^t}{\prod_{i=1}^p(\alpha_i-t)_{n_i}} \frac{ \operatorname d t }{2 \pi \ii },
\end{align*}
 where \( \Sigma\) is a clockwise contour in $\{t\in \mathbb{C} \mid \operatorname{Re}(t)>-1\}$ enclosing \(\cup_{i=1}^p[\alpha_i,n_i+\alpha_i-1]\) exactly once. This expression can also be obtained by applying the previous mentioned limit over the integral representation for the second kind Meixner linear forms in Theorem \ref{M2I_IR}. The remaining limits can then be derived from this one.
\end{proof}
\begin{rem}
 The limiting relation for the type II polynomials was already suggested at \cite{AskeyII}.
\end{rem}

\subsection{Laguerre of the Second Kind}

The Laguerre of the second kind weights are given by
\begin{align}
\label{WeightsL2}
 w_i(x;\alpha_0,c_i)=x^{\alpha_0}\Exp{-c_i x},\quad i\in\{1,\ldots,p\},\quad\Delta=[0,\infty),
\end{align}
with \( \alpha_0>-1\), \(c_1,\cdots,c_p>0\) and, to ensure an AT system, \( c_i\neq c_j\) for \( i\neq j\). We will denote \(\vec{c}\coloneq(c_1,\ldots,c_p)\). The Laguerre of the second kind family can be reached either from Jacobi--Piñeiro and Meixner of the first kind through respective limits. The Jacobi--Piñeiro limit has already been studied in \cite{AskeyII} for the type II polynomials and in \cite{CMOPI} for the type I polynomials and recurrence coefficients. Here, we will show the following.
\begin{pro}
The following limiting relations between the Meixner of the first kind \eqref{WeightsM1} and the Laguerre of the second kind weights hold
\[
\begin{aligned}
 w_i^{(L2)}(x;\alpha_0,c_i)&=\lim_{\tau\to\infty} \frac{\Gamma(\alpha_0+1)}{\tau^{\alpha_0}} w_i^{(M1)}\left(\tau x;\alpha_0+1,\frac{\tau}{\tau+c_i}\right),& i&\in{\{1,\ldots,p\}}.
\end{aligned}
\]
\end{pro}
\begin{proof} 
This follows from the asymptotics
\[\begin{aligned}
 w_i^{(M1)}\left(\tau x;\alpha_0+1,\frac{\tau}{\tau+c_i}\right)&= \frac{\Gamma(\tau x+\alpha_0+1)}{\Gamma(\alpha_0+1)\Gamma(\tau x+1)} \left(\frac{\tau}{\tau+c_i}\right)^{\tau x} \sim \frac{\tau^{\alpha_0} }{\Gamma(\alpha_0+1)}\,x^{\alpha_0}\operatorname{e}^{-c_ix},
\end{aligned}\]
\( \tau \to\infty \).
\end{proof}

\begin{coro} 
\label{Cor:M1->L2}
For the linear forms \( L_{2:\vec{n}}^{({\rm I})} \), type I polynomials \( L_{2:\vec{n}}^{(i)} \), type II polynomials \( L_{2:\vec{n}}^{({\rm II})} \) and recurrence coefficients $b^{(L2),j}_{\vec{n}}$; the following respective limit relations from Meixner first kind hold:
\begin{subequations}
\begin{align}
L_{2:\vec{n}}^{(\ast)}(x;\alpha_0,\vec{c})&= \lim_{\tau\to\infty} \kappa^{(\ast)}_\tau M_{1:\vec{n}}^{(\ast)}\left(\tau x;\alpha_0+1,\left\{\frac{\tau}{\tau+c_i}\right\}_{i=1}^p\right), \quad
 b^{(L2),j}_{\vec{n}}(\alpha_0,\vec{c})&=\lim_{\tau\rightarrow\infty}\dfrac{1}{\tau^{j+1}}b^{(M1),j}_{\vec{n}}\left(\alpha_0+1,\left\{\frac{\tau}{\tau+c_1}\right\}_{i=1}^p\right),
\end{align}
\end{subequations}
where
$$ \kappa^{({\rm I})}_\tau = \frac{1}{\kappa^{({\rm II})}_\tau} = \tau^{|\vec{n}|},\quad \kappa^{(i)}_\tau = \frac{\tau^{|\vec{n}|+\alpha_0}}{\Gamma(\alpha_0+1)}. $$
\end{coro}

\begin{proof}
In \cite[Theorem 3.2]{B-K}, one found the following integral representation for the linear forms
\begin{equation}
\label{L2I_IR}
L_{2:\vec{n}}^{({\rm I})}(x;\alpha_0,\vec{c}) = \frac{(-1)^{\sz{n}} \prod_{i=1}^p c_i^{n_i} }{\Gamma(\sz{n}+\alpha_0)} x^{\alpha_0} \int_{\Sigma^\ast} \frac{t^{\sz{n}+\alpha_0-1} \operatorname{e}^{-xt}}{\prod_{j=1}^p (t-c_j)^{n_j}} \frac{\d t }{2 \pi \ii }.
\end{equation}
where \( \Sigma^\ast \) is a clockwise contour in $\{t\in \mathbb{C} \mid \operatorname{Re}(t)>0\}$ enclosing \(\{c_i\}_{i=1}^p\) exactly once. This expression can also be obtained by applying the previous mentioned limit over the integral representation for the second kind Meixner linear forms in Theorem \ref{M1I_IR}.
\end{proof}
\begin{rem}
 The previous limit relation for the type II polynomials was already suggested at \cite{AskeyII}.
\end{rem}

\subsection{Hermite}
For the Hermite family, the weights are
\begin{equation}
 \label{WeightsHe}
\begin{aligned}
 w_i^{(H)}(x;c_i)&=\Exp{-x^2+c_ix},& i&\in\{1,\ldots,p\},& \Delta&=\R,
\end{aligned}
\end{equation}
with \(\vec{c}\coloneq(c_1,\ldots,c_p)\in\mathbb R^p\) and, to ensure an AT system, \( c_i\neq c_j\) for \( i\neq j\). According to the Askey scheme; the Hermite family should be reachable from Jacobi--Piñeiro, Laguerre of the first kind, Laguerre of the second kind, Kravchuk and Charlier. Limits from the first two families for the type II polynomials were given in \cite{CVA} for the case $p=2$, and they are straightforward to generalize to general $p\geqslant2$. More recently, in \cite{CMOPI} the similar limits were given for the type I polynomials and the recurrence coefficients.
In what follows, we will complete the Askey scheme by describing the three remaining limits. We will need the following auxiliary results.

\begin{lemma}
 \label{lemmaGamma1}
 Suppose that $\beta,y,z,a,b,c\in\mathbb R$ and $x>0$, then
 \begin{align*}
 \left(x\beta+y\sqrt{\beta}+z\right)^{a\beta+b\sqrt{\beta}+c}\sim
(x\beta)^{a\beta+b\sqrt{\beta}+c}\Exp{
\frac{ay}{x}\sqrt{\beta}+\frac{az}{x}-\frac{ay^2}{2x^2}+\frac{by}{x}},\quad \beta\to\infty.
 \end{align*}
\end{lemma}

\begin{proof}
 Note that:
 \begin{align*}
\left(x\beta+y\sqrt{\beta}+z\right)^{a\beta+b\sqrt{\beta}+c}&=(x\beta)^{a\beta+b\sqrt{\beta}+c}\left(1+\frac{y}{x\sqrt{\beta}}+\frac{z}{x\beta}\right)^{a\beta+b\sqrt{\beta}+c}\\
&=(x\beta)^{a\beta+b\sqrt{\beta}+c}\exp{\left(\left({a\beta+b\sqrt{\beta}+c}\right)\ln{\left(1+\frac{y}{x\sqrt{\beta}}+\frac{z}{x\beta}\right)}\right)}.
 \end{align*}
We can then use the power series of the logarithm
\begin{equation*}
 \ln(1+t)=\sum_{k=1}^{\infty}\frac{(-1)^{k+1}}{k}t^k
\end{equation*}
to obtain
\begin{align*}
\left(x\beta+y\sqrt{\beta}+z\right)^{a\beta+b\sqrt{\beta}+c}
&=(x\beta)^{a\beta+b\sqrt{\beta}+c}\exp{\left(\left({a\beta+b\sqrt{\beta}+c}\right)\left(\frac{y}{x\sqrt{\beta}}+\frac{z}{x\beta}-\frac{y^2}{2x^2\beta}+\operatorname{O}\left(\frac{1}{\sqrt{\beta}^3}\right)\right)\right)}\\
&=(x\beta)^{a\beta+b\sqrt{\beta}+c}\exp{\left(
\frac{ay}{x}\sqrt{\beta}+\frac{az}{x}-\frac{ay^2}{2x^2}+\frac{by}{x}+\operatorname{O}\left(\frac{1}{\sqrt{\beta}}\right)\right)}.
 \end{align*}
 \end{proof}

Combining this result with Stirling's formula \eqref{stirling} gives us
\begin{lemma}
 \label{lemmaGamma2}
 Suppose that $\beta,y,z\in\mathbb R$ and $x>0$, then
 \begin{equation*}
\begin{aligned}
 \Gamma\left(x\beta+y\sqrt{\beta}+z\right)&\sim\sqrt{2\pi}(x\beta)^{x\beta+y\sqrt{\beta}+z-\frac{1}{2}}\Exp{
\frac{y^2}{2x}-x\beta},& \beta&\to\infty.
\end{aligned}
 \end{equation*}
\end{lemma}

We are now ready to establish the limiting relations.
\begin{pro}
The following limiting relations between the Kravchuk \eqref{WeightsK}, Charlier \eqref{WeightsC}, Laguerre of the second kind \eqref{WeightsL2} and Hermite weights hold
\begin{align}
 w_i^{(H)}(x;c_i)&=\lim_{N\to\infty}\sqrt{\frac{\pi N}{2}}\Exp{\frac{c_i^2}{4}}w_i^{(K)}\left(\frac{N}{2}+x\sqrt{\frac{N}{2}};\frac{1}{2}+\frac{c_i}{2\sqrt{2N}},N\right)
 \\&=\lim_{\beta\rightarrow\infty}\sqrt{2\pi\beta}\Exp{-\beta-c_i\sqrt{\frac{\beta}{2}}+\frac{c_i^2}{4}}w_i^{(C)}\left(\beta+x\sqrt{2\beta};\beta+c_i\sqrt{\frac{\beta}{2}}\right)\\
 &=\lim_{\beta\to\infty}\beta^{-\beta}\Exp{\beta-c_i\sqrt{\frac{\beta}{2}}}w_i^{(L2)}\left(\beta+x\sqrt{2\beta};\beta,1-\frac{c_i}{\sqrt{2\beta}}\right),\quad i\in\{1,\ldots,p\}.
\end{align}
\end{pro}

\begin{proof}
Using the two previous lemmas, one can find
 \begin{align*}
 w_i^{(K)}\left(\frac{N}{2}+x\sqrt{\frac{N}{2}};\frac{1}{2}+\frac{c_i}{2\sqrt{2N}},N\right)&=\dfrac{\Gamma(N+1)}{\Gamma\left(\frac{N}{2}+x\sqrt{\frac{N}{2}}+1\right)\Gamma\left(\frac{N}{2}-x\sqrt{\frac{N}{2}}+1\right)} \left(\frac{1}{2}+\frac{c_i}{2\sqrt{2N}}\right)^{\frac{N}{2}+x\sqrt{\frac{N}{2}}}\left(\frac{1}{2}-\frac{c_i}{2\sqrt{2N}}\right)^{\frac{N}{2}-x\sqrt{\frac{N}{2}}}\\
 &\sim\sqrt{\frac{2}{\pi N}}\Exp{-\frac{c_i^2}{4}-x^2+c_ix},\quad N\to\infty,\\
 w_i^{(C)}\left(\beta+x\sqrt{2\beta};\beta+c_i\sqrt{\frac{\beta}{2}}\right)&=\frac{\left(\beta+c_i\sqrt{\frac{\beta}{2}}\right)^{\beta+x\sqrt{2\beta}}}{\Gamma\left(\beta+x\sqrt{2\beta}+1\right)}\sim
 \frac{1}{\sqrt{2\pi\beta}}\Exp{\beta+c_i\sqrt{\frac{\beta}{2}}-\frac{c_i^2}{4}-x^2+c_ix}
 ,\quad \beta\to\infty,\\
 w_i^{(L2)}\left(\beta+x\sqrt{2\beta};\beta,1-\frac{c_i}{\sqrt{2\beta}}\right)&=\left(\beta+x\sqrt{2\beta}\right)^{\beta}\Exp{-\left(1-\frac{c_i}{\sqrt{2\beta}}\right)\left(\beta+x\sqrt{2\beta}\right)}\sim\beta^\beta\Exp{-\beta+c_i\sqrt{\frac{\beta}{2}}-x^2+c_ix},\quad\beta\to\infty.
 \end{align*}
\end{proof}

\begin{coro}
\label{Cor:K/C/L2->He}
For the linear forms \( H_{\vec{n}}^{({\rm I})} \), type I polynomials \( H_{\vec{n}}^{(i)} \), type II polynomials \( H_{\vec{n}}^{({\rm II})} \) and recurrence coefficients \(b^{(H),j}_{\vec{n}}\); the following respective limiting relations from Kravchuk, Charlier and Laguerre of the second kind hold:
\begin{subequations}
 \label{K/C/L2->He}
\begin{align}
 \label{K/C/L2->He:LF}
 H_{\vec{n}}^{(\ast)}(x;\vec{c})&=\lim_{N\to\infty}
\kappa^{(K:\ast)}_{N} K_{\vec{n}}^{(\ast)}\left(\frac{N}{2}+x\sqrt{\frac{N}{2}};\frac{\vec{1}_p}{2}+\frac{\vec{c}}{2\sqrt{2N}},N\right)
=\lim_{\beta\rightarrow\infty}\kappa^{(C:\ast)}_{\beta}
C_{\vec{n}}^{(\ast)}\left(\beta+x\sqrt{2\beta};\vec{1}_p\beta +\vec{c}\sqrt{\frac{\beta}{2}}\right) 
\\
&=\lim_{\beta\rightarrow\infty}\kappa^{(L2:\ast)}_{\beta} L_{2:\vec{n}}^{(\ast)}\left(\beta+x\sqrt{2\beta};\beta,\vec{1}_p-\frac{\vec{c}}{\sqrt{2\beta}}\right),
\end{align}
\begin{align} \label{K/C/L2->He:R}
 b_{\vec{n}}^{(H),j}(\vec{c})&=\lim_{N\to\infty} {\sqrt{\frac{2}{N}}^{j+1}}\left(b_{\vec{n}}^{(K),j}\left(\frac{\vec{1}_p}{2}+\frac{\vec{c}}{2\sqrt{2N}},N\right){-\frac{N}{2}}\delta_{j,0}\right)
 =\lim_{\beta\rightarrow\infty}{\frac{1}{\sqrt{2\beta}^{j+1}}}\left(b_{\vec{n}}^{(C),j}\left(\vec{1}_p\beta +\vec{c}\sqrt{\frac{\beta}{2}}\right){-\beta}\delta_{j,0}\right)\\
 &=\lim_{\beta\to\infty}
 \frac{1}{\sqrt{2\beta}^{j+1}}\left(b_{\vec{n}}^{(L2),j}\left(\beta,\vec{1}_p-\frac{\vec{c}}{\sqrt{2\beta}}\right)-\beta\delta_{j,0}\right),
\end{align}
\end{subequations}
where
\begin{align}
 \kappa^{(K:I)}_{N} &= \frac{1}{\kappa^{(K:II)}_{N}} = \sqrt{\frac{N}{2}}^{\sz{n}},\quad \kappa^{(K:i)}_{N} = {\sqrt{\frac{N}{2}}^{\sz{n}-1}}\frac{\Exp{-\frac{c_i^2}{4}}}{\sqrt{\pi}},\\ \kappa^{(C:I)}_{\beta} &= \frac{1}{\kappa^{(C:II)}_{\beta}} = \kappa^{(L2:I)}_{\beta} = \frac{1}{\kappa^{(L2:II)}_{\beta}} = \sqrt{2\beta}^{\sz{n}},\quad
 \kappa^{(C:i)}_{\beta} &= \frac{{\sqrt{2\beta}^{\sz{n}-1}}}{\sqrt{\pi}}\Exp{\beta+c_i\sqrt{\frac{\beta}{2}}-\frac{c_i^2}{4}}, \quad
 \kappa^{(L2:i)}_{\beta} = \sqrt{2\beta}^{\sz{n}}\beta^{\beta}\Exp{-\beta+c_i\sqrt{\frac{\beta}{2}}}.
\end{align}
\end{coro}
\begin{proof}
In \cite[Theorem 2.3]{B-K}, one obtained the following integral representation for the linear forms
\[
H_{\vec{n}}^{({\rm I})}(x; \vec{c}) = \frac{2^{|\vec{n}| - 1}}{\sqrt{\pi}} \bigintsss_{\Sigma^\ast} \frac{\operatorname{e}^{-\left(x-\frac{t}{2}\right)^2}}{\prod_{i=1}^p (t - c_i)^{n_i}} \frac{\d t}{2 \pi \ii}.
\]
where \( \Sigma^\ast \) is a clockwise contour in $\{t\in \mathbb{C} \mid 0<\operatorname{Re}(t)>0\}$ enclosing \(\{c_i\}_{i=1}^p\) exactly once. This expression can also be obtained by applying the stated respective limits over the integral representation for the Kravchuk (Theorem \ref{KI_IR}), Charlier (Theorem \ref{CI_IR}) and Laguerre second kind \eqref{L2I_IR} linear forms. The remaining limits then follow from these ones.
\end{proof}

\section*{Conclusions and Outlook}

In this paper, explicit hypergeometric and integral representations are provided for the discrete descendants of the Hahn polynomials within the Askey scheme, including the multiple Meixner polynomials of the first and second kinds, as well as the Kravchuk and Charlier polynomials. Additionally, the corresponding near-neighbor recurrence coefficients are presented. The multiple Askey scheme has been completed with corresponding limits.

A intriguing and largely unexplored area lies in mixed multiple orthogonal polynomials. These extend the classical families of multiple orthogonal polynomials discussed in this paper to incorporate mixed multiple cases. Currently, explicit hypergeometric expressions for mixed Hahn multiple orthogonal polynomials are yet to be discovered, as are contour integral representations for these polynomials. This domain offers a promising avenue for future research, potentially unveiling new insights into the interactions between different orthogonal polynomial families and their applications.

Another promising line of research is to discover the explicit bidiagonal factorization for the recursion matrix on the stepline. This involves characterizing the regions of positivity and total positivity of this banded Hessenberg matrix. Additionally, constructing explicit Markov chains in terms of these multiple orthogonal polynomials is an intriguing avenue, where one can characterize their properties and factorizations in terms of pure birth and pure death chains. This could lead to a deeper understanding of the underlying stochastic processes and their connections to multiple orthogonal polynomials.

\section*{Acknowledgments} AB acknowledges Centre for Mathematics of the University of Coimbra funded
by the Portuguese Government through FCT/MCTES, DOI: 10.54499/UIDB/00324/2020.

AF and JEFD acknowledge Center for Research and Development
in Mathematics and Applications (CIDMA) from University of Aveiro funded by the Portuguese Foundation
for Science and Technology (FCT) through projects DOI: 10.54499/UIDB/04106/2020 and DOI:
10.54499/UIDP/04106/2020. Additionally, JEFD acknowledges PhD contract
DOI: 10.54499/UI/BD/152576/2022 from FCT.

JEFD and MM acknowledge research project [PID2021- 122154NB-I00], \emph{Ortogonalidad y Aproximación con Aplicaciones en Machine Learning y Teoría de la Probabilidad} funded by \href{https://doi.org/10.13039/501100011033}{MICIU/AEI/10.13039/501100011033} and by ``ERDF A Way of making Europe''.

TW acknowledges PhD project 3E210613 funded by BOF KU Leuven.


\end{document}